\documentclass[12pt]{amsart}

\usepackage{amsmath}
\usepackage{amsfonts}
\usepackage{amssymb}
\usepackage{latexsym}
\usepackage{graphicx} 
\usepackage{enumerate}
\usepackage{multirow}
\usepackage{subfigure}

\numberwithin{equation}{section}

\newtheorem{theorem}{Theorem}[section]
\newtheorem{proposition}[theorem]{Proposition}
\newtheorem{lemma}[theorem]{Lemma}

\newtheorem{remark}[theorem]{Remark}
\newtheorem{definition}[theorem]{Definition}

\usepackage{mathtools}

\makeatletter
\DeclareRobustCommand\widecheck[1]{{\mathpalette\@widecheck{#1}}}
\def\@widecheck#1#2{%
    \setbox\z@\hbox{\m@th$#1#2$}%
    \setbox\tw@\hbox{\m@th$#1%
       \widehat{%
          \vrule\@width\z@\@height\ht\z@
          \vrule\@height\z@\@width\wd\z@}$}%
    \dp\tw@-\ht\z@
    \@tempdima\ht\z@ \advance\@tempdima2\ht\tw@ \divide\@tempdima\thr@@
    \setbox\tw@\hbox{%
       \raise\@tempdima\hbox{\scalebox{1}[-1]{\lower\@tempdima\box
\tw@}}}%
    {\ooalign{\box\tw@ \cr \box\z@}}}
\makeatother

\newcommand{\ds}{\displaystyle}

\newcommand{\ep}{\epsilon}
\newcommand{\om}{\omega}

\DeclareMathOperator{\sech}{sech}

\DeclareMathOperator{\diag}{diag}
\DeclareMathOperator{\sgn}{sgn}

\DeclareMathOperator{\spn}{span}
\DeclareMathOperator{\range}{range}

\newcommand{\per}{{\text{per}}}

\newcommand{\be}{\begin{equation}}
\newcommand{\ee}{\end{equation}}
\newcommand{\bes}{\begin{equation*}}
\newcommand{\ees}{\end{equation*}}
\newcommand{\mand}{\quad \text{and}\quad}

\newcommand{\R}{{\bf{R}}}
\newcommand{\C}{{\bf{C}}}
\newcommand{\T}{{\bf{T}}}
\newcommand{\Z}{{\bf{Z}}}

\newcommand{\ub}{{\bf{u}}}

\newcommand{\hb}{{\bf{h}}}
\newcommand{\hbg}{{\grave{\bf{h}}}}
\newcommand{\hg}{{\grave{h}}}

\newcommand{\gb}{{\bf{g}}}

\newcommand{\pb}{{\bf{p}}}

\newcommand{\jb}{{{\bf{e}}_2}}
\newcommand{\ib}{{{\bf{e}}_1}}

\newcommand{\ga}{{\grave{a}}}
\newcommand{\geta}{{\grave{\eta}}}
\newcommand{\getab}{{\grave{\etab}}}
\newcommand{\gj}{{\grave{j}}}
\newcommand{\gl}{{\grave{l}}}
\newcommand{\gxi}{{\grave{\xi}}}
\newcommand{\gpsib}{{\grave{\psib}}}

\newcommand{\muz}{{\mu,0}}
\newcommand{\muo}{{\mu,1}}
\newcommand{\muj}{{\mu,j}}

\newcommand{\Fo}{{\mathfrak{F}}}
\newcommand{\G}{{\mathcal{G}}}

\renewcommand{\S}{{\mathcal{S}}}
\renewcommand{\O}{{\mathcal{O}}}
\renewcommand{\H}{{\mathcal{H}}}
\renewcommand{\L}{{\mathcal{L}}}
\newcommand{\X}{{\mathfrak{X}}}
\renewcommand{\P}{{\mathcal{P}}}
\newcommand{\RR}{{\mathcal{R}}}
\newcommand{\Xc}{{\mathcal{X}}}
\newcommand{\D}{{\mathcal{D}}}
\newcommand{\U}{{\mathcal{U}}}
\newcommand{\B}{{\mathcal{B}}}

\renewcommand{\r}{{\mathfrak{r}}}

\newcommand{\varphib}{{\boldsymbol \varphi}}
\newcommand{\sigmab}{{\boldsymbol \sigma}}
\newcommand{\xib}{{\boldsymbol \xi}}

\newcommand{\psib}{{\boldsymbol \psi}}

\newcommand{\etab}{{\boldsymbol \eta}}
\newcommand{\nub}{{\boldsymbol \nu}}
\newcommand{\rhob}{{\boldsymbol \rho}}
\newcommand{\thetab}{{\boldsymbol \theta}}

\newcommand{\bunderbrace}[2]{%
  \begin{array}[t]{@{}c@{}}
  \underbrace{#1}\\
  #2
  \end{array}
}

\renewcommand{\tilde}{\widetilde}
\renewcommand{\hat}{\widehat}

\title[Nanopteron solutions of FPUT]{Nanopteron solutions of diatomic Fermi-Pasta-Ulam-Tsingou lattices with small mass-ratio}
\author{Aaron Hoffman}\address{Drexel University, Philadelphia PA, \tt{jdoug@math.drexel.edu}}
\author{J. Douglas Wright}\address{Olin College of Engineering, Needham MA, \tt{aaron.hoffman@olin.edu}}
\begin{document}
\thanks{
We would like acknowledge the National Science Foundation which has generously supported
the work through grants DMS-1105635 and DMS-1511488.  A large debt of gratitude is owed to Anna Vainchstein and Yuli Starosvetsky 
for setting us on the path that led to this article.
}

\begin{abstract}
Consider an infinite chain of masses, each connected to its nearest neighbors by a (nonlinear) spring. This is a Fermi-Pasta-Ulam-Tsingou lattice. 
We prove the existence of traveling waves in the setting where the masses alternate in size. In particular we address the limit where the mass ratio tends to zero. The problem is inherently singular and we find that the traveling waves are not true solitary waves but rather ``nanopterons", which is to say, waves which asymptotic at spatial infinity to very small amplitude periodic waves. Moreover, we can only find solutions when the mass ratio lies in a certain open set. The difficulties in the problem all revolve around understanding Jost solutions of a nonlocal Schr\"odinger operator in its semi-classical limit. 
\end{abstract}

\maketitle

{Arrange infinitely many particles on a horizontal line, each attached to its nearest neighbors
by a spring with a nonlinear restoring force. Constrain the motion of the particles to be within the line.
This system is called a Fermi-Pasta-Ulam (FPU) or (more recently \cite{dauxois}) a
Fermi-Pasta-Ulam-Tsingou (FPUT) lattice and it is one of the paradigmatic models
for nonlinear and dispersive waves. 
In this article, we consider the existence of traveling waves in a diatomic (or ``dimer") FPUT lattice
when the ratio of the masses is nearly zero. By this we mean that
the masses of the particles alternate between $m_1$ and $m_2$ along the chain and 
$$
\mu:= {m_2 \over m_1} \ll 1.
$$
The springs are all identical materially. The force they exert, when stretched by an amount $r$ from their equilibrium
length, is 
$
F_s( r ) := - k_s r - b_s r^2
$
where $k_s > 0$ and $b_s \ne 0$.
See \cite{brillouin} for an overview of this problem's history and \cite{porter2} for a discussion of technological
applications of such a system.
}

Newton's second law gives the equations of motion.
After nondimensionalization, these read
\be\begin{split}\label{dimer newton}
\mathfrak{m}_j \ddot{y}_j  &= -r_{j-1} - r_{j-1}^2 + r_j + r_{j}^2.
\end{split}\ee
Here $j \in \Z$ and 
$
r_j:=y_{j+1} - y_j.
$
When $j$ is odd
$
\mathfrak{m}_j = 
1$  and when $j$ is even,
$\mathfrak{m}_j=\mu$.
In the above, $y_j$ is the nondimensional displacement from equilibrium of the $j$th particle.
See Figure \ref{DFPUT} for a schematic.
\begin{figure}[t]
\centering
    \includegraphics[width=6in]{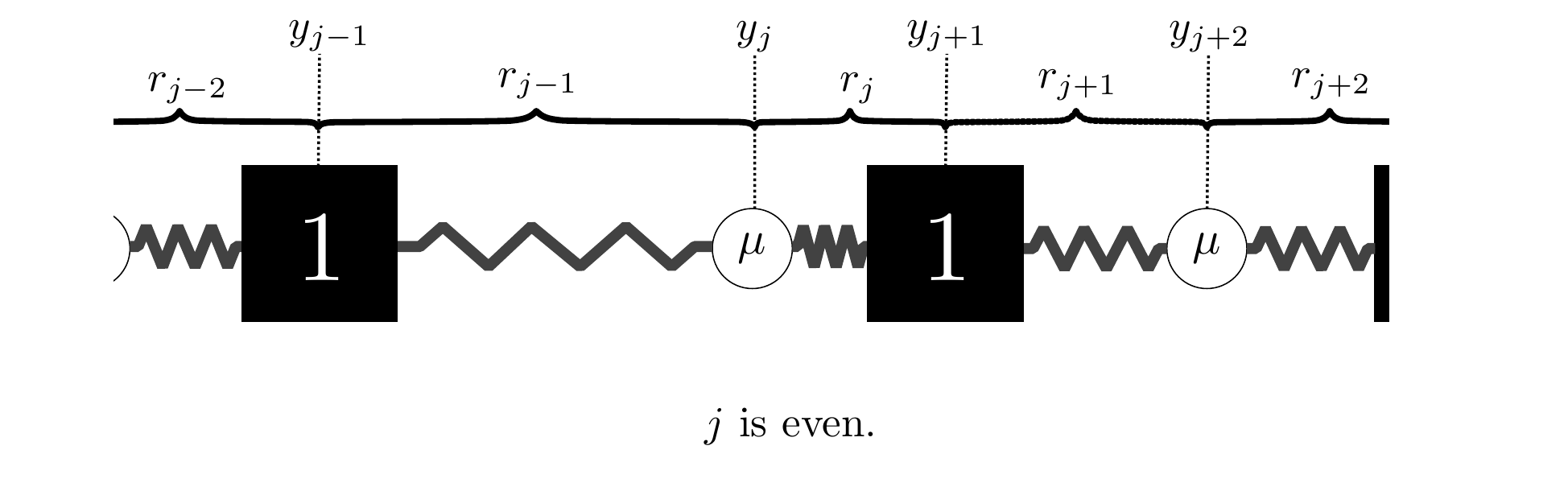}
 \caption{  \it A snippet of the diatomic Fermi-Pasta-Ulam-Tsingou lattice.
 \label{DFPUT}}
\end{figure}

In  the monatomic case (where $\mu =1$), this system
famously possesses localized traveling wave solutions.
Formal arguments suggesting their existence date back to \cite{kruskal}. The first rigorous
proofs
 can be found in \cite{toda} (for a very special alternate nonlinearity) and \cite{friesecke-wattis} (for more general convex nonlinearities).
The articles  \cite{friesecke-pego1} \cite{friesecke-pego2} \cite{friesecke-pego3} \cite{friesecke-pego4}
 demonstrate that these traveling waves are asymptotically stable.
 
Putting $\mu = 0$ 
amounts to removing the smaller masses but leaving the springs attached. 
Which is to say that we have a monatomic lattice with a modified  spring force. The results in \cite{friesecke-wattis} and \cite{friesecke-pego1} apply in this setting as well and so there is a localized traveling wave 
solution for \eqref{dimer newton} when $\mu=0$.

The central question of this article is this: {\it does the $\mu = 0$ traveling wave solution persist when  $0 < \mu \ll 1$?}
In \eqref{dimer newton} the small parameter $\mu$ multiplies a second derivative and as such the system is singularly perturbed. An attempt to answer the question by way of regular perturbation theory ({\it i.e.}~an implicit function theorem argument) is doomed to failure.
Nonetheless, we have an answer: {\it the localized traveling wave at $\mu = 0$ perturbs into a nanopteron\footnote{A nanopteron is the superposition of a localized function and an extremely small amplitude spatially periodic piece.} 
for $\mu$ in an open set of postive numbers whose closure contains zero.}

Several recent articles, specifically \cite{VSWP} and \cite{porter}, have carried out detailed formal asymptotics and 
performed careful numerics for this problem. They strongly indicate traveling waves solutions for \eqref{dimer newton} are nanopterons, at least for most values of the mass ratio $\mu$; this article represents a
rigorous mathematical validation of those predictions.
We will in particular comment on the results of \cite{VSWP} below in Remark \ref{AY} in Section \ref{strategy}.  

Nanopteron solutions are one of the many outcomes  one may find for singularly perturbed systems of differential equations \cite{boyd}.
The ``usual" way to prove their existence for a set of ordinary differential equations is through either geometric singular perturbation theory or matched asymptotics \cite{holmes}. However, our problem is infinite dimensional which complicates using those sorts of tools.
The method by which we prove our main result is a modification of one developed by Beale in \cite{beale2} to study the existence of traveling waves in the capillary-gravity problem.\footnote{That is, one-dimensional free surface water waves which are acted
by the restoring forces of gravity and surface tension.}  His method, which is functional analytic in nature, was subsequently deployed to show the existence of nanopteron solutions in several other singularly perturbed problems ({\it e.g.} \cite{amick-toland}), including
one closely related to the one here in \cite{FW}.\footnote{In that article, the existence of traveling waves 
of nanopteron type for \eqref{dimer newton} is established but in a rather different limit: $\mu>0$ is fixed but 
the the wavespeed is taken just above the lattice's speed of sound.}

The common feature of problems with nanopteron solutions is that the singular perturbation manifests itself as a high frequency solution of the linearization.
This in turn implies that a certain solvability condition must be met by solutions of the nonlinear problem.
The key difference between what transpires here versus in \cite{beale2} \cite{amick-toland} \cite{FW} is that in our problem the high frequency linear solution is not a pure sinusoid but rather a Jost solution for a nonlocal Schr\"odinger operator. This is no minor thing: the ability to meet the solvability condition is more subtle and, as  a consequence, the method only gives solutions for $\mu$ in the aforementioned open set (which we call $M_c$)
as opposed to for all $\mu$ sufficiently close to zero.
As we shall demonstrate, $M_c$ is an infinite union of finite open intervals which aggregate at $\mu = 0$. 

This article is organized as follows:
\begin{itemize}
\item Section \ref{equation of motion} contains a reformulation of the equations
of motion which 
 has a simple form when $\mu = 0$. We state a nontechnical version of our  main result in terms of
 this formulation in Theorem \ref{nontech main}.
Several symmetries of the governing equations are also discussed.
\item Section \ref{preliminaries} sets up the function analytic framework we work in and contains 
a number of simple estimates we will use repeatedly.
\item Section \ref{strategy} is a ``birds-eye view" of the strategy of our proof, which is based on 
Beale's work in \cite{beale2}. This section
also discusses the places where and why substantial adjustments to his method have to be made.
\item Section \ref{refined} is the first technical part of the proof and contains a ``refinement" of the $\mu =0$ monatomic
approximation. This is the first building block of the nanopteron solutions.
\item Section \ref{periodic solutions}  concerns spatially periodic traveling wave solutions of \eqref{DFPUT}
and it is here that we state
Theorem \ref{periodic solutions exist}, a novel existence result. These
solutions ultimately give rise to the periodic part of the nanopteron, but are of independent interest (see \cite{qin} \cite{betti-pelinovsky}).
\item Section \ref{beales ansatz} we put together the refined leading order limit and periodic solutions
and derive the first form of the governing equations for the nanopterons.
\item Section \ref{light operator} deals with  the singular part of the linearization. It is in this section where the ultimate form of the nanopteron equations appear.
\item Section \ref{enchilada} contains a statement of all the most important estimates (Lemma \ref{mover}), the
 technical statement of our main result (Theorem \ref{main result}) and the proof of that theorem given the estimates.
 \item Finally, we have Appendices \ref{A appendix}-\ref{estimate appendix}, which contain all the technical details
 of the proofs and estimates from the main part of the paper.
\end{itemize}

\section{The main result.}\label{equation of motion}
\subsection{The equations of motion.}
Before we state our main theorem, we reformulate \eqref{dimer newton}.
Let
\be\label{ansatz}
y_j(t) = \begin{cases}Y_{1}(j,t) &\text{when $j$ is odd} \\ Y_2(j,t) &\text{when $j$ is even.} \end{cases}
\ee
To be clear, $Y_1(j,t)$ is defined only for $j$ odd and $Y_2(j,t)$ for $j$ even. Computing $r_j$ in terms of $Y_1$ and $Y_2$ gives 
$$
r_j = \begin{cases}S^1 Y_2(j)-Y_1(j) &\text{when $j$ is odd} \\ S^1 Y_1(j)-Y_2(j) &\text{when $j$ is even} \end{cases}
$$
where $S^d$ is the ``shift by $d$" map, specifically $S^d f( \cdot)  := f(\cdot + d)$.

With this,  \eqref{dimer newton} reads
\be\begin{split}\label{dn2}
 \ddot{Y}_1 & = -2 Y_1 + 2A Y_2  +4(\delta Y_2) ( A Y_2 - Y_1)\\
\mu \ddot{Y}_2 & = 2 A Y_1 - 2Y_2  + 4(\delta Y_1) ( A Y_1 - Y_2)
\end{split}\ee
where
\be\label{Ad}
A :={1 \over 2} \left(S^{1} + S^{-1}\right) \mand
\delta :={1 \over 2} \left(S^1 -S^{-1}\right).
\ee

If $\mu = 0$ then the second equation in \eqref{dn2} is satisfied when $Y_2 = AY_1$.
We want to change variables in a way that exploits this, so we set
$$
\rho_2 := Y_2 - AY_1.
$$
Notice $\rho_2(j,t)$ is defined for even integers $j$.
Plugging this into \eqref{dn2} we get
\be\begin{split}\label{dn3}
 \ddot{Y}_1 & = 2 \delta^2 Y_1 + 2A \rho_2  +4(A \delta Y_1  + \delta \rho_2) ( \delta^2 Y_1 + A\rho_2)\\
\mu \ddot{Y}_2 & = -2 \rho_2 - 4(\delta Y_1) \rho_2 - 2\mu A\left[\delta^2 Y_1 + A\rho_2+2(A \delta Y_1 + \delta \rho_2)
(\delta^2 Y_1 + A\rho_2)\right].
\end{split}\ee
In computing the above we have made judicious use of the identity $1+\delta^2 = A^2$.

On the right hand side of \eqref{dn3}, $Y_1$ always appears with at least one $\delta$ applied.
Thus we define $$\rho_1:=\delta Y_1.$$  Like $\rho_2(j,t)$, $\rho_1(j,t)$ is defined for even integers $j$.
With this, if we apply $\delta$ to the first equation we get
\bes\begin{split}
  \ddot{\rho}_1 & = 2 \delta^2 \rho_1 + 2A\delta \rho_2  + 4\delta[(A \rho_1 + \delta \rho_2) ( \delta \rho_1 + A\rho_2)]\\
\mu \ddot {\rho}_2 & = -2\rho_2 - 4 \rho_1 \rho_2 -2\mu A[ \delta \rho_1 + A\rho_2  + 2(A \rho_1 + \delta \rho_2) ( \delta \rho_1 + A\rho_2)].
\end{split}\ees

Next we use the identity $\ds
(Ag_1 + \delta g_2)(\delta g_1 + A g_2) = {1 \over 2} \delta  (g_1^2 +g_2^2) + A (g_1 g_2)
$
on the right hand side and we arrive at
\be\begin{split}\label{dn5}
  \ddot \rho_1 & = 2 \delta^2 \rho_1 + 2A\delta \rho_2  + 2\delta^2 (\rho_1^2 +\rho_2^2) + 4A\delta (\rho_1 \rho_2)\\
\mu \ddot \rho_2 & = -2 \mu A \delta \rho_1 -2(1+\mu A^2) \rho_2 -2\mu A  \delta  (\rho_1^2 +\rho_2^2) - 4(1+ \mu A^2)( \rho_1 \rho_2) .
\end{split}\ee
This system, which is posed for $j$ in the even lattice $2\Z$, is equivalent to 
the equations of motion \eqref{dimer newton}.

The formulas for the variables $\rho_1$ and $\rho_2$ may
seem to be somewhat nonintuitive, but they in fact have simple physical meanings. 
A chase through their definitions shows that $\rho_1$ and $\rho_2$ are found
from the ``stretch" variables $r_j = y_{j+1} - y_j$ by
\bes
\rho_1(j) = {1 \over 2} \left( r_j+r_{j-1}\right)
\mand
\rho_2(j) = -{1 \over 2} \left(r_{j}-r_{j-1}\right).
\ees
Since $j$ is an even number in the above, we see that $\rho_1(j)$ is simply half the distance from a heavy mass
to the next heavy mass. Which is to say it is the distance to the midway point between heavy masses.
And $\rho_2(j)$ measures how far the light mass in between those heavy ones is from this midpoint.
See Figure \ref{rhofig}.
\begin{figure}
\centering
    \includegraphics[width=6in]{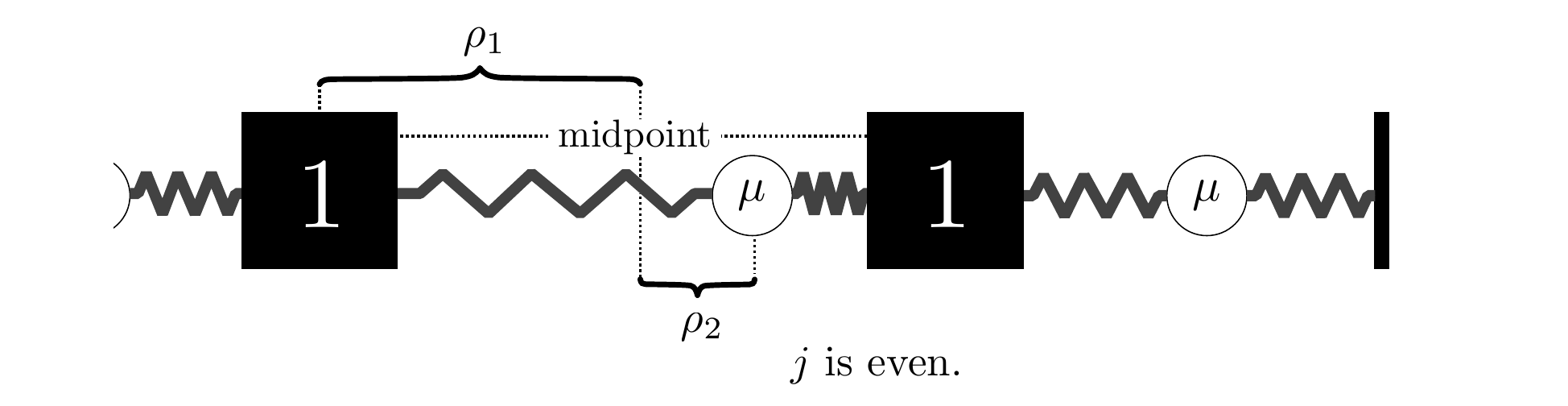}
 \caption{  \it Sketch of $\rho_1$ and $\rho_2$.
 \label{rhofig}}
\end{figure}
This point of view is why we will sometimes refer to $\rho_1$ as the ``heavy variable"; it is determined
fully by the locations of  the heavy particles alone. We will also call first equation in \eqref{dn5}  the ``heavy equation."
On the other hand, $\rho_2$ specifies the location of the lighter particles
and so we call it the ``light variable" and the second equation in \eqref{dn5}
the ``light equation."

\subsection{Our result.}
The system \eqref{dn5}
has a simple structure at $\mu = 0$: 
the second equation reduces to 
\be\label{rho2 at zero}
\rho_2 + 2\rho_1 \rho_2 =0
\ee which can be solved by taking $\rho_2(x) \equiv 0$. 
Physically this means the light (in this case, massless) particles are located exactly halfway between their bigger brethren. 
The heavy equation is then simply \be\label{rho1 at zero}
\ddot \rho_1 = 2 \delta^2 \left(\rho_1 +  \rho_1^2\right)\ee which  (see for instance \cite{friesecke-pego1}) coincides with the equations
of motion for a  monatomic lattice with restoring force given by $-2\rho - 2\rho^2$.

Since the $\mu=0$ problem is equivalent to a monatomic FPUT lattice, we can summon the results of \cite{friesecke-wattis} and \cite{friesecke-pego1}
to get an exact traveling wave solution to it.
\begin{theorem} [Friesecke \& Pego]\label{sigma exists}
There exists $c_1 > c_0$, where
\be \label{speed of sound}c_0:= \text{``the speed of sound"}:=\sqrt{2}, \ee 
for which  $|c| \in (c_0,c_1]$ 
implies the existence of a positive, even, smooth, bounded and unimodal function, $\sigma_c(x)$, 
such that
$\rho_1(j,t) = \sigma_c(j-ct)$
satisfies \eqref{rho1 at zero}. 
Moreover $\sigma_c(x)$ is exponentially localized
in the following sense: there exists
$b_c>0$ such that, for any $s \ge 0$ and $p \in [1,\infty]$, we have
\be\label{sigma is localized}
\left \| \cosh^{b_c}(\cdot)  {\sigma^{(s)}_c(\cdot) } \right\|_{L^p(\R)} < \infty.
\ee
\end{theorem}

Thus we see that putting
\be\label{mu zero}
\rho_1(j,t) = \sigma_c(j-ct) \mand \rho_2(j,t) = 0
\ee
solves \eqref{dn5} when $\mu =0$.

Now we can state our main thereom. It says that this solution 
perturbs into a nanopteron for $\mu > 0$; see Figure \ref{nanopic}.
\begin{theorem}\label{nontech main}
For $|c| \in (c_0,c_1]$
there exists an open set $M_c \subset \R_+$ whose
closure contains zero and for which $\mu \in M_c$ implies the following. There exist smooth functions 
$\Upsilon_{c,\mu,1}(x)$, $\Upsilon_{c,\mu,2}(x)$, $\Phi_{c,\mu,1}(x)$ and $\Phi_{c,\mu,2}(x)$ such that 
putting
$$\rho_1(j,t) = \sigma_c(j-ct)+\Upsilon_{c,\mu,1}(j-ct)+\Phi_{c,\mu,1}(j-ct)$$
and
$$\rho_2(j,t) = \Upsilon_{c,\mu,2}(j-ct)+\Phi_{c,\mu,2}(j-ct) $$ 
  solves \eqref{dn5}. Moreover, $\Upsilon_{c,\mu,1}(x)$ and $\Upsilon_{c,\mu,2}(x)$
 are exponentially localized and have amplitudes of $\O(\mu)$.
  Finally, $\Phi_{c,\mu,1}(x)$ and $\Phi_{c,\mu,2}(x)$
 are high frequency (specifically $\O(\mu^{-1/2})$) periodic functions of $x$ 
 whose amplitudes
are small beyond all orders of $\mu$.
\end{theorem}

\begin{figure}
{\centering
    \includegraphics[width=6in]{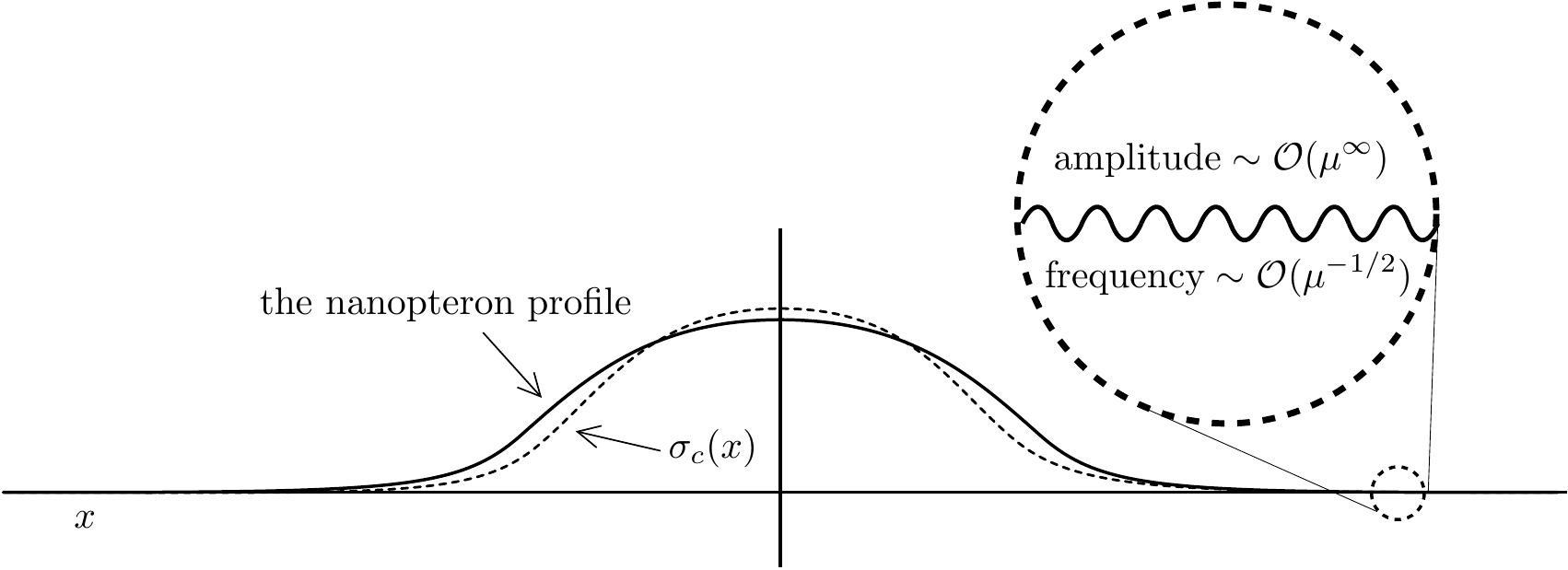}
 \caption{  \it  Sketch of the nanopteron profile for the first component, $\rho_1$. The periodic
 part is so small as to be invisible relative to the core, hence the inset. By $\O(\mu^{\infty})$
 we mean ``small beyond all orders of $\mu$."}
 \label{nanopic}}
\end{figure}

We do not prove this theorem in the coordinates $\rho_1$ and $\rho_2$, but rather we make an additional near identity
change of variables.

\subsection{``Almost diagonalization"}
Denoting\footnote{Generally, we represent maps from $\R \to \R^2$ by bold letters, for instance $\rhob(x)$ or $\hb(x)$. 
Likewise, the first and second components of such functions, as shown here, will be represented in the regular font with a subscript ``$1$" and ``$2$", respectively.}
$
\rhob=\ds\left(\begin{array}{c} \rho_1 \\ \rho_2 \end{array} \right),
$
if we put
\be\label{here is Q}
I_\mu:=\diag(1,\mu),\quad
\D_\mu := \left[ \begin{array}{cc}
-2 \delta^2 & -2 A \delta\\
2\mu A\delta &2(1+\mu A^2) 
\end{array}
\right]
\mand 
Q_0(\gb,\grave{\gb}) = \left(\begin{array}{c}
g_1 \grave{g}_1 + g_2 \grave{g}_2\\
g_1 \grave{g}_2 + g_2 \grave{g}_1 
\end{array}
\right)
\ee
then \eqref{dn5} is equivalent to
\be \label{dn5 system}
  I_\mu \ddot{\rhob} + \D_\mu \rhob + \D_\mu Q_0(\rhob,\rhob) = 0.
\ee

Observe that $\D_\mu$ is upper triangular when $\mu = 0$. We can make a simple change
of variables that ``almost diagonalizes" the linear part of \eqref{dn5 system}; this will be advantageous
down the line.
Let\be\label{final cov}
\rhob:= T_\mu \thetab \quad \text{where} \quad
T_\mu:=\left[ \begin{array}{cc} 1 & -\mu A \delta \\ 0 & 1\end{array}\right].
\ee
Since $T_\mu$ is a small perturbation of the identity, we will continue to ascribe to $\theta_1$
and $\theta_2$ the physical meaning of $\rho_1$ and $\rho_2$; that is $\theta_1$ is ``heavy" and $\theta_2$ (which is in fact exactly $\rho_2$) is ``light."

With this, \eqref{dn5 system} becomes
\bes
 I_\mu T_\mu \ddot \thetab + \D_\mu T_\mu \thetab + \D_\mu Q_0(T_\mu \thetab,T_\mu \thetab) = 0.
\ees
Since $T_\mu$ is invertible for all $\mu \ge 0$ and $I_\mu$ is invertible for $\mu >0$,
the above is equivalent to
\be\label{dn6}
I_\mu \ddot \thetab + L_\mu  \thetab +L_\mu   Q_\mu (\thetab, \thetab) = 0
\ee
where\footnote{This product defining $L_\mu$ is rather painful to multiply out, so we omit showing the details; it can be computed formally by replacing $A$ with $\cos(k)$ and $\delta$ with $i \sin(k)$ and asking a computer algebra
system to carry out the product. Then you just replace all the cosines with $A$ and sines with $-i \delta$.
The reason that this works is that $A e^{ikx} = \cos(k) e^{ikx}$ and $\delta e^{ikx} = i \sin(k) e^{ikx}$. Which is to
say the product is easier on the Fourier side.} 
\be\label{Lm}
L_\mu : = I_\mu T_\mu^{-1} I_\mu^{-1}\D_\mu T_\mu := \left[ 
\begin{array}{cc}
 -2 \delta^2 (1-\mu A^2) & -2 \mu A \delta \left(1-2A^2+\mu A^2 \delta^2 \right)\\
 2 \mu A \delta & 2(1 +  \mu A^2 -  \mu^2 A^2 \delta^2)
\end{array}
\right]
\ee
and
$$
Q_\mu(\thetab,\grave{\thetab}) := T^{-1}_\mu Q_0(T_\mu \thetab,T_\mu \grave{\thetab}). 
$$
As we advertised above,  $$
L_0 =\left[\begin{array}{cc} -2 \delta^2 & 0 \\ 0 & 2 \end{array}\right]$$
is a diagonal operator and thus $L_\mu$ is nearly diagonal.

At last we make the traveling wave ansatz:
$$
\thetab(j,t) = \hb(j-ct)
$$
where $\hb:\R \to \R^2$ and $c \in \R$ is the wave speed.
With this we get the following equation for $\hb$:\footnote{As before, we think of $h_1(x)$ as being the heavy variable and $h_2(x)$ as being the light one.} 
\be\label{TWE5}
 \bunderbrace{c^2 I_\mu \hb'' + L_\mu  \hb +L_\mu   Q_\mu (\hb, \hb)}{\G(\hb,\mu) }= 0.
\ee
The primes denote differentiation with respect to $x$, the independent variable of $\hb$.
Of course the operators $A$ and $\delta$ (out of which are constructed $L_\mu$ and $Q_\mu$) act on functions of $x\in\R$ just as they do on functions of $j \in \Z$. 
Note that $Q_\mu$ is bilinear and symmetric in its arguments.

At 
$\mu =0$ we have $\rhob = \thetab$. Thus the line of reasoning that lead to \eqref{mu zero} tells us  that putting\footnote{We use $\ib:=\left(\begin{array}{c}1\\0 \end{array} \right)$ and $\jb:=\left(\begin{array}{c}0\\1 \end{array} \right)$ to denote the usual unit vectors in $\R^2$.}
\be\label{what sigma does}
\sigmab_c := \sigma_c \ib 
\quad \implies \quad
\G(\sigmab_c,0) = 0.
\ee

\subsection{Some symmetries of $\G$}
We now point out two  symmetries possessed by the mapping $\G(\hb,\mu)$. The first is that 
if $h_1(x)$ is an even function of $x$ and $h_2(x)$ is odd, then the components of $\G(\hb,\mu)$ are, respectively, even and odd. This is
a consequence of the following simple facts:
\begin{itemize}
\item Both $\partial_x$ and $\delta$ map even functions to odd ones and vice-versa.
\item The map $A$ takes odd functions to odd functions and even to even. 
\item If $g_1$ is even and $g_2$ is odd then $g_1^2$ and $g_2^2$ are even while $g_1 g_2$ is odd.
\end{itemize}
With these in hand, showing that $\G$ maintains  ``even $\times$ odd" symmetry  amounts to just scanning through its definition.
Henceforth we assume that the function $\hb$ and its descendents will have this $\text{even} \times \text{odd}$ symmetry.  

Moreover, $\delta$ annihilates constants, just as $\partial_x$ does. And, also just like
$\partial_x$, we have
$$
\int_{\R} \delta f(x) dx =0 \mand \int_{-L}^L \delta f(x) dx = 0.
$$
In the first integral we assume that $f(x)$ is going to zero quickly enough as $|x| \to \infty$ and in the second that $f$ is periodic with period $2L$.
Both integrals read as a sort of ``mean-zero" condition for $\delta f$ and so we say that $\delta f$ is a ``mean-zero function."
Also observe that each term in the first row of $L_\mu$ 
has at least one factor of $\delta$ exposed. Which is to say that the first component of $\G$ has no constant term, or equivalently,
that it is  mean-zero.

We summarize these symmetries in the informal statement
\be\label{G sym}
\G(\cdot,\mu) : \left\{\text{evens} \right\} \times  \left\{\text{odds} \right\}  \to 
 \left\{ \text{mean-zero evens} \right\} \times  \left\{ \text{odds} \right\} .
\ee
It is worth pointing out that $I_\mu \partial_x^2$, $L_\mu$ and $L_\mu Q_\mu$, each on their own, have this same property.
Now that we have our traveling wave equation spelled out, and we understand it well at $\mu=0$,
we take the next section to lay out our function spaces, key definitions as well as prove some rudimentary estimates.

\section{Functions spaces, notation and basic estimates}\label{preliminaries}

\subsection{Periodic functions}
We let $W^{s,p}_\per:=W^{s,p}(\T)$ be  the usual ``$s,p$" Sobolev
space of $2 \pi$-periodic functions. We denote $L^p_\per:=W^{0,p}_\per$ and
 $H^s_\per:=W^{s,2}_\per$.
Put
\be\label{per spaces}
E^s_\per:=H^s_{\per} \cap \left\{\text{even functions} \right\}\mand O^s_\per:=H^s_{\per} \cap \left\{\text{odd functions} \right\}.\ee
We will also make use of 
\be\label{per0 spaces}
E^s_{\per,0}:=E^s_{\per} \cap \left\{u(X):\int_{-\pi}^\pi u(X) dX = 0 \right\}.
\ee
That is to say, mean-zero even periodic functions.
By $C^s_\per$ we mean the space of $s-$times differentiable $2\pi$-periodic functions
and $C^\infty_\per$ is the space of smooth $2\pi$-periodic functions. 

\subsection{Functions on $\R$}
We let $W^{s,p}:=W^{s,p}(\R)$ be the usual ``$s,p$" Sobolev
space of functions defined on  $\R$. For $b \in \R$ put
\be\label{Wspb spaces}
W^{s,p}_{b} := \left\{ u \in L^2(\R) : \cosh^b(x) u(x) \in W^{s,p}(\R) \right\} .
\ee
These are Banach spaces with the naturally defined norm. 
If we say a function is ``exponentially localized" we mean that it is in one of these spaces with $b > 0$.

 Put $L^p_b:=W^{0,p}_b$, 
$H^s:=W^{s,2}$ and $H^s_b:=W^{s,2}_b$ and 
denote $\| \cdot \|_{s,b}:=\| \cdot \|_{H^s_b}$. We let 
\be\label{sb spaces}
E^s_b:=H^s_{b} \cap \left\{\text{even functions} \right\}\mand O^s_b:=H^s_{b} \cap \left\{\text{odd functions} 
\right\}.
\ee
For instance the monatomic wave profile $\sigma_c(x)$, by virtue of \eqref{sigma is localized}, is in $E^s_{b_c}$ for all $s$.
We will also make use of 
\be\label{sb0 spaces}
E^s_{b,0}:=E^s_{b} \cap \left\{u(x):\int_\R u(x) dx = 0 \right\}.
\ee
This is another space of mean-zero even functions.

\subsection{Big $\O$ and big $C$ notation.}
Many of our 
quantities will depend on the mass ratio $\mu$, the wavespeed $c$ and a decay rate $b$.
Of these, $\mu$ is the chief but $b$ will play an important role as well; we generally view $c$
as being fixed and as such we do not usually track dependence on it. 
In order to simplify the statements of many of our estimates we 
employ the following conventions. Any exceptions/restrictions will be clearly noted.
\begin{definition}\label{big C def}
Suppose that $q = q (\mu,b,c) \in \R_+$ and $f = f(\mu) \in \R_+$. 
\begin{itemize}
\item We write
\be\label{big C eqn 1}
q \le C f  \quad \text{or} \quad q \le \O(f)
\ee
if, for all  $|c| \in (c_0,c_1]$, there exists $\mu_* \in (0,1) $ and $C>0$ such that
$q(\mu,b,c) \le C f(\mu)$ for all $\mu \in  (0,\mu_*)$ and $b \in [0,b_c]$. 
That is to say, the constant $C$ in \eqref{big C eqn 1}
may depend on $c$ but not on $\mu$ or $b$.

\item We write
\be\label{big C eqn 2}
q \le C_b f \quad \text{or} \quad q \le \O_b(f)
\ee
if, for all  $|c| \in (c_0,c_1]$ and $b \in (0,b_c]$, there exists $\mu_{*}(b) \in (0,1) $ and $C_b>0$ such that
$q(\mu,b',c) \le C_b f(\mu)$ for all $\mu \in  (0,\mu_{*}(b)]$ and $b' \in [b,b_c]$. 
That is to say, the constant $C_b$ in \eqref{big C eqn 2}
may depend on $c$ and $b$ but not on $\mu$.

\end{itemize}
\end{definition}
The point is this: if an estimate depends in a bad way on the decay rate $b$ then we adorn it with the subscript $b$.
Moreover, our definition indicates that $C_b$ increases as $b\to 0^+$.

We will also occasionally write either $q = \O(1)$ or $q=  \O(\mu^p)$, as
as opposed to $q \le \O(1)$ or $q \le   \O(\mu^p)$. 
We use this to indicate that we can control $q$ from above and below by $C\mu^p$.
Specifically
\begin{definition}\label{big C def 2}
Suppose that $q = q (\mu,c) \in \R_+$ and $p \in \R$.
We write
\be\label{big C eqn 3}
q = \O(\mu^p) 
\ee
if  $|c| \in (c_0,c_1]$ implies there is $\mu_*\in (0,1) $ and $0<C_1<C_2$ such that
$$
C_1 \mu^p \le q(\mu,c) \le C_2 \mu^p
$$
for all $\mu \in  (0,\mu_*)$. 
\end{definition}

Lastly, we will encounter terms which are small ``beyond all orders of $\mu$." By this we mean the following.
\begin{definition}\label{mu infinity} Suppose $q = q(\mu,c) \in \R_+$. We write 
$$
q \le \O(\mu^\infty)
$$
if $q \le\O(\mu^p)$ for all $p \ge 1$.
\end{definition}

\subsection{General estimates for $W^{s,p}_b$}
We take for granted the  containment estimate:  
$$
\| f\|_{s,b} \le \| f \|_{s',b'} \quad \text{ when $s \le s'$, $b \le b'$}.
$$
Likewise, there is the generalization of the Sobolev embedding estimate:
$$\ds
\| f\|_{W^{s-1,\infty}_b} \le C \| f\|_{s,b}\quad \text{ when $s \ge 1$}.
$$

We have the following simple estimate:
\begin{multline*}
\| f g \|_{W^{0,p}_b} = \| \cosh^b(\cdot) f g\|_{L^p} = \| (\cosh^{b'}(\cdot) f)( \cosh^{b-b'}(\cdot) g)\|_{L^p}\\
\le \| \cosh^{b'}(\cdot) f\|_{L^\infty} \| \cosh^{b-b'}(\cdot) g\|_{L^p} = \| f \|_{W^{0,\infty}_{b'}} \| g \|_{W^{0,p}_{b-b'}}.
\end{multline*}
This in turn implies 
\be\label{borrow}
\| f g\|_{W^{s,p}_b} \le C \| f\|_{W^{s,\infty}_{b'}}\| g\|_{W^{s,p}_{b-b'}}
\ee
and, using the Sobolev estimate
\be\label{algebra}
\| fg \|_{s,b} \le C \| f\|_{s,b'}\|g\|_{s,b-b'}.
\ee
Note that \eqref{borrow} holds for $s \ge 0$ but \eqref{algebra} only when $s \ge 1$.
The important feature of these estimates is that the weight $b$ on the left hand side can be shared
between the two functions $f$ and $g$ more or less however we wish. It will even be necessary at several stages for us to have $b-b'<0$.
Note also that \eqref{algebra} implies that, if $b \ge 0$ and $s\ge1$, that $H^s_b$ is an algebra.

We also have the containment $L^q_b \subset L^p$ when $p < q$ and $b > 0$. These follow
from the following estimate involving H\"olders inequality
\be\label{lp emb}
\| f\|_{L^p}=\| \sech^b(\cdot) \cosh^b(\cdot) f\|_{L^p} \le \|\sech^b(\cdot)\|_{L^{q'}} \| f \|_{L^{q}_b}\le C_b\| f \|_{L^{q}_b}
\ee
where $\ds {1 \over q'} = {1 \over p} - {1 \over q}$.

\subsection{Simple operator estimates}
The following estimates follow from $u-$substitution and the definition of the shift operator $S^d$:
\be\label{S is bounded}
\| S^d f\|_{W^{s,p}_\per} 
=
\| f\|_{W^{s,p}_\per} \mand
\| S^d f\|_{W^{s,p}_b} 
\le  e^{|db|}
\| f\|_{W^{s,p}_b}.
\ee
That is to say $S^d$
is a bounded operators on all of the function spaces
we use here; we will treat them as such without comment.
These estimates imply
\be\label{Ad is bounded}
\| A\|_{\X \to \X} \le C
\mand 
\| \delta\|_{\X \to \X} \le C
\ee
where $\X$ is either $W^{s,p}_b$ or $W^{s,p}_\per$. In turn these, after a quick glance at the definitions of $L_\mu$ and $T_\mu$,
give
\be\label{LT is bounded}
\| L_\mu - L_0\|_{\X \to \X} \le C\mu \mand
\| T_\mu - {\bf id}\|_{\X \to \X}\le C\mu
\ee
where $\X$ is either $W^{s,p}_b \times W^{s,p}_b$ or $W^{s,p}_\per \times W^{s,p}_\per$.
Note that the estimates in \eqref{LT is bounded} hold for all $\mu \in (0,1)$.

\subsection{The bilinear map $L_\mu Q_\mu$.}
We have not written out $L_\mu Q_\mu(\hb,\hbg)$ in full detail; 
there are very many terms and ultimately not much would be learned.
But  we can give a useful and relatively simple collection of estimates for it. 
For functions $\hb:\R \to \R^n$, define the ``windowed absolute value" as
$$
|\hb(x)|_W:= \max_{d\in \Z,\ |d| \le 10} \left\{ |S^d\hb(x)|_{\R^n}\right\}.
$$
Here $| \cdot|_{\R^n}$ is the just the Euclidean norm on $\R^n$.
From the esimates for $S^d$ in \eqref{S is bounded}, we can conclude that
\be\label{window 1}
\left \| \left| \hb(\cdot) \right|_W \right \|_{\X} \le C\| \hb \|_{\X}
\ee
where $\X=L^2_b$ or $L^2_\per$.

Examining \eqref{here is Q} and \eqref{final cov}
gives
\be\label{LQ0}
L_0 Q_0(\hb,\hbg) = \left[ \begin{array}{c} -2 \delta^2 (h_1 \hg_1 + h_2 \hg_2) \\
2 (h_1 \hg_2 + h_2 \hg_1)
\end{array}\right].
\ee
Since $\delta$ is made out of the shifts $S^{\pm 1}$, this formula tells us that for all $x \in \R$
\be\label{LQ0i}
\left |L_0 Q_0(\hb,\hbg)(x) \cdot \ib\right| \le C\left( |h_1(x)|_W |\hg_1(x)|_W+|h_2(x)|_W |\hg_2(x)|_W\right)
\ee
and
\be\label{LQ0j}
\left |L_0 Q_0(\hb,\hbg)(x) \cdot \jb\right| \le C\left( |h_1(x)| |\hg_2(x)|+|h_2(x)| |\hg_1(x)|\right).
\ee

Similarly, for all $x \in \R$, we have
\be\label{LQ err}
\left | L_\mu Q_\mu(\hb,\hbg)(x)-L_0 Q_0(\hb,\hbg)(x) \right|_{\R^2} \le C\mu | \hb (x)|_W |\hbg(x)|_W.
\ee
Here is why. The estimates \eqref{LT is bounded}
and the definition of $Q_\mu$ in \eqref{here is Q}  demonstrate
that extracting $L_0 Q_0$ from $L_\mu Q_\mu$ leaves only terms with at least one
exposed power of $\mu$. The operator $Q_0(\hb,\hbg)$ is bilinear and this is why we have
the product of $| \hb (x)|_W$ and $|\hbg(x)|_W$ on the right. And we have the windowed absolute
value because $T_\mu$ and $L_\mu$ are constructed out of $A$ and $\delta$, which themselves
are made from $S^{\pm 1}$. Tracking through the definitions shows that the most shifts that could land
on one function is ten: two in $T_\mu$, two in $T_\mu^{-1}$ and six in $L_\mu$. This is why the window has a
radius of ten.

Putting  \eqref{window 1}, \eqref{LQ0i}, \eqref{LQ0j} and \eqref{LQ err} together with \eqref{borrow} and various Sobolev 
estimates yields the following collection of estimates:
\be\label{E1}\begin{split}
\| L_\mu Q_\mu(\hb,\grave{\hb}) \cdot \ib\|_{s,b} \le &C \left(
\| h_1\|_{W^{s,\infty}_{b'}} \| \grave{h}_1\|_{s,b-b'} + 
\| h_2\|_{W^{s,\infty}_{b''}} \| \grave{h}_2\|_{s,b-b''} \right)\\ + &C\mu \left( 
 \| h_1\|_{W^{s,\infty}_{b'''}} \| \grave{h}_2\|_{s,b-b'''} 
+  \| h_2\|_{W^{s,\infty}_{b''''}} \| \grave{h}_1\|_{s,b-b''''} 
\right)
\end{split}\ee
\be\label{E2}\begin{split}
\| L_\mu Q_\mu(\hb,\grave{\hb}) \cdot \jb\|_{s,b} \le &C \left( 
 \| h_1\|_{W^{s,\infty}_{b'}} \| \grave{h}_2\|_{s,b-b'} 
+  \| h_2\|_{W^{s,\infty}_{b''}} \| \grave{h}_1\|_{s,b-b''} 
 \right)\\ + &C\mu \left( 
 \| h_1\|_{W^{s,\infty}_{b'''}} \| \grave{h}_1\|_{s,b-b'''} + 
\| h_2\|_{W^{s,\infty}_{b''''}} \| \grave{h}_2\|_{s,b-b''''}
\right)
\end{split}\ee
\be\label{E3}\begin{split}
\| L_\mu Q_\mu(\hb,\grave{\hb}) \cdot \ib\|_{s,b} \le &C \left(
\| h_1\|_{s,b'} \| \grave{h}_1\|_{s,b-b'} + 
\| h_2\|_{{s,b''}} \| \grave{h}_2\|_{s,b-b''} \right)\\ + &C\mu \left( 
 \| h_1\|_{s,b'''} \| \grave{h}_2\|_{s,b-b'''} 
+  \| h_2\|_{{s,b''''}} \| \grave{h}_1\|_{s,b-b''''} 
\right)
\end{split}\ee
\be\label{E4}\begin{split}
\| L_\mu Q_\mu(\hb,\grave{\hb}) \cdot \jb\|_{s,b} \le &C \left(
 \| h_1\|_{s,b'} \| \grave{h}_2\|_{s,b-b'} 
+  \| h_2\|_{{s,b''}} \| \grave{h}_1\|_{s,b-b''} 
 \right)\\ + &C\mu \left( 
\| h_1\|_{s,b'''} \| \grave{h}_1\|_{s,b-b'''} + 
\| h_2\|_{{s,b''''}} \| \grave{h}_2\|_{s,b-b''''}
\right)
\end{split}\ee
\be\label{E5}\begin{split}
\| L_\mu Q_\mu(\hb,\grave{\hb}) \cdot \ib\|_{H^s_\per} \le &C \left(
\| h_1\|_{H^s_\per} \| \grave{h}_1\|_{H^s_\per} + 
\| h_2\|_{H^s_\per} \| \grave{h}_2\|_{H^s_\per} \right)\\ + &C\mu \left( 
 \| h_1\|_{H^s_\per} \| \grave{h}_2\|_{H^s_\per} 
+  \| h_2\|_{{H^s_\per}} \| \grave{h}_1\|_{H^s_\per} 
\right)
\end{split}\ee
and
\be\label{E6}\begin{split}
\| L_\mu Q_\mu(\hb,\grave{\hb}) \cdot \jb\|_{H^s_\per} \le &C \left(
 \| h_1\|_{H^s_\per} \| \grave{h}_2\|_{H^s_\per} 
+  \| h_2\|_{H^s_\per} \| \grave{h}_1\|_{H^s_\per} 
 \right)\\ + &C\mu \left( 
\| h_1\|_{H^s_\per} \| \grave{h}_1\|_{H^s_\per} + 
\| h_2\|_{{H^s_\per}} \| \grave{h}_2\|_{H^s_\per}
\right).
\end{split}\ee
In the above, $b',b'',b'''$ and $b''''$ are free to be anything.
 In \eqref{E3}-\eqref{E6} we need $s \ge 1$,
but $s \ge 0$ suffices for \eqref{E1} and \eqref{E2}.
All of the estimates \eqref{E1}-\eqref{E6}  hold for all $\mu \in (0,1)$.


\section{The strategy}\label{strategy}
Here we outline our approach to the proof of Theorem \ref{nontech main}.
The first thing we do is we put 
$
\hb = \sigmab_c + \xib
$
into \eqref{TWE5}. 
For the time being, we are thinking of $\xib$ as being small and {\it localized} function, for instance in $E^s_{b_c} \times O^s_{b_c}$.
(That is to say $\xib$, like $\hb$, has the even $\times$ odd symmetry.)
Using \eqref{what sigma does} and 
some algebra shows that
\be\label{first pass}
c^2 I_\mu \xib'' + L_0 \xib + 2L_0 Q_0(\sigmab_c,\xib) = R_\mu.
\ee
The left hand side consists of all the leading order terms, plus the singularly perturbed term $c^2 \mu \xi_2''$.
The right hand side $R_\mu$ is everything else.
Its exact form is not 
germaine at this point but it is made up of:
\begin{itemize}
\item terms which are linear in $\xib$ but have at least one prefactor of $\mu$
(for instance $\mu A^2\delta^2 \xi_1$),  \item $\O(\mu)$ residuals
(for instance $\mu A\delta \sigma_c(x)$) and 
\item terms which are quadratic in $\xib$.
\end{itemize}
 That is to say, the right hand side is ``small." 
 
 If we can invert the linear operator of $\xib$ on the left hand side we would
have our system written as a $\xib = N(\xib)$ where $N$ is an (ostensibly) small operator and then we could use a contraction mapping argument to solve this fixed point equation. A look at the page count of this article makes it clear that we could not get such a  strategy to work.

To see why, we write out \eqref{first pass} component-wise. The  first (or ``heavy") component reads
\be\label{first pass 1}\begin{split}
\bunderbrace{c^2 \xi_1'' - 2 \delta^2[(1+2\sigma_c) \xi_1]}{\H \xi_1} = R_{\mu,1}.
\end{split}\ee
The operator $\H$ is closely related to an operator that appears in the analysis of the  stability of monatomic FPUT solitary waves \cite{friesecke-pego1}.
The following result, which states that $\H$ is (more or less) invertible in the class of even functions, can be inferred from results there; we 
carry out the details in Appendix \ref{A appendix}.
\begin{proposition}\label{A props} 
The following holds for
$|c| \in (c_0,c_1]$\footnote{The constants $c_1$ and $b_c$ in this proposition may be smaller
than their counterparts with the same names in Theorem \ref{sigma exists}, but since that theorem remains
true with the smaller values here, we act here (and throughout) as if the constants were the same to begin with.}. The map $\H$ is a homeomorphism of $E^{s+2}_{b}$ and $E^s_{b,0}$ for any $s \ge 0$ and $b \in (0,b_c]$. Specifically we have
\be\label{A estimate}
\| \H^{-1}  \|_{H^{s}_{b,0} \to H^{s+2}_b} \le C_b. 
\ee
\end{proposition}
Note that the result states that the range of $\H$ is the mean-zero even functions as opposed to all even functions.
The symmetry properties described in \eqref{G sym} imply that $R_{\mu,1}$ is in fact an even mean-zero function.
Thus we can write \eqref{first pass 1} as $\xi_1 = \H^{-1} R_{\mu,1}$ and the right hand side of this can be shown
to meet the hypotheses of the contraction mapping theorem in the localized spaces. That is, the heavy component of \eqref{first pass}
poses no problem.

Writing out the second (``light") component of \eqref{first pass} gives
\be\label{first pass 2}
\bunderbrace{c^2 \mu \xi_2'' + 2(1 + 2\sigma_c) \xi_2}{\U_\mu \xi_2} = R_{\mu,2}.
\ee
This operator is a garden-variety second order Schr\"odinger operator\footnote{Note that since  $\mu \to 0^+$, 
we are considering this operator in the ``semi-classical limit."} with potential functiom $2(1 + 2\sigma_c(x))$.
Since we are working with $x \in \R$, standard undergraduate differential equations theory tells us that there are two linearly
independent solutions of $\U_\mu \mathfrak{z} = 0$. Call them $\mathfrak{z}_\muz(x)$ and $\mathfrak{z}_\muo(x)$. 
Since $\sigma_c(x)$ is an even function, we can arrange it so that $\mathfrak{z}_\muz(x)$ is even and $\mathfrak{z}_\muo(x)$
is odd.
And since $\sigma_c(x)$ is positive and converges to zero at infinity 
we can infer that these functions converge, as $|x| \to \infty$, to solutions of $c^2\mu \mathfrak{z}'' + 2 \mathfrak{z}= 0$. Which is
to say they are asymptotic as $|x| \to \infty$ to a linear combination of $\sin(\tilde{\omega}_\mu x)$ and $\cos(\tilde{\omega}_\mu x)$, with $\tilde{\omega}_\mu : = \ds \sqrt{2 /c^2 \mu}$.
Thus, since $\mu$ is small, these have very high frequencies.

The functions $\mathfrak{z}_\muz$ and $\mathfrak{z}_\muo$ are  the Jost solutions\footnote{Or, at least, they are linear combinations of them, see \cite{reed-simon3}.} for the operator  $\U_\mu$. Note that since these functions do not converge to zero
at infinity they are not in the localized spaces $H^s_b$. And we also know from ODE theory that all solutions of $\U_\mu z = 0$
will be linear combinations of $\mathfrak{z}_\muz$ and $\mathfrak{z}_\muo$. All of this implies that the only function $z(x) \in H^s_b$ which solves $\U_\mu z = 0$ is $z(x) \equiv 0$. In this way we can conclude $\U_\mu$, viewed as a map from $H^{s+2}_b$ to $H^s_b$, is injective, since it has a trivial kernel. 

We remark now that an alternate way to prove
the injectivity of $\U_\mu$ on the localized spaces is by way of the following coercive estimate: if $f \in H^{2}_b$
then
\be\label{U coerce}
\| f\|_{0,b} \le C_b \mu^{-1/2}\|\U_\mu f\|_{0,b}.
\ee
We prove a similar estimate for a closely related operator below in Appendix \ref{B appendix}.
 The point is that this estimate also implies 
that the kernel of $\U_\mu$ in $H^s_b$ spaces is trivial and thus $\U_\mu$ is injective.

However $\U_\mu$ is not surjective on those spaces.
Now suppose that we have
a solution $\xi_2$ of \eqref{first pass 2} in $O^s_b$. The embedding \eqref{lp emb} implies that $\xi_2$ is in $L^1$.
So if we multiply \eqref{first pass 2} by ${\mathfrak{z}}_\muo(x)$ then integrate over $\R$, the integrals will converge. We get
$$
\int_\R [\U_\mu \xi_2(x)] {\mathfrak{z}}_\muo(x) dx = \int_\R \ R_{\mu,2} {\mathfrak{z}}_\muo(x)dx.
$$
The operator $\U_\mu$ is symmetric with respect to the $L^2$ inner-product, thanks to integration by parts. Thus, since $\U_\mu {\mathfrak{z}}_\muo =0$, we have
\be\label{NC}
\int_\R R_{\mu,2} {\mathfrak{z}}_\muo(x) dx = 0.
\ee
That is to say, we have a third equation\footnote{Of course, we could repeat 
this argument with $\mathfrak{z}_\muo$ replaced by $\mathfrak{z}_\muz$. But \eqref{G sym}
tells us that $R_{\mu,2}$ is an odd function of $x$ and so  the resulting integral condition,
$\ds \int_\R R_{\mu,2} {\mathfrak{z}}_\muz(x) dx = 0$,  is met automatically since the integrand is odd.} we need to solve in addition to \eqref{first pass 1} and \eqref{first pass 2}.
But we only have two unknowns, $\xi_1$ and $\xi_2$.

\begin{remark}\label{AY}
In \cite{VSWP}, the authors study this same problem but with the the restoring force taken to be that of the Toda lattice ($F_{s} = e^{-r}-1$)
as opposed to the simple force  ($F_s=-r-r^2$) we use here. The approach taken there follows the formalism of asymptotics for ``fast-slow" dynamical systems. Through the right lens what they find there  parallels what we have here. In particular they arrive at an equation equivalent to our \eqref{first pass 2}.
The Toda lattice is integrable and in that case the function $\sigma_c(x)$
is known explicitly in terms of elementary functions, as is the leading order part of $R_{\mu,2}$.
 Moreover, they find explicit formulas for their Jost solutions 
in terms of hypergeometric functions. 

The point is that the (ostensibly) leading order part of 
their analog of the solvability condition
\eqref{NC}
is totally explicit and depends, ultimately,
only on $\mu$. It can even be evaluated
using residue calculus, though the resulting formula is not so simple to understand. 
Numerical computation of this formula demonstrates that  the leading order part of the solvability
condition is met at a countably infinite sequence of values of $\mu$ converging to zero. 

Which is to say that perhaps there do exist genuinely localized
traveling waves for the problem at (or near) those special values of $\mu$. 
Our feeling is that if this is true, the method we deploy here is not sufficient
to prove it and so we take the cautious route and say only that it is possible.
\end{remark}

In \cite{beale2}, Beale found a similar phenomenon in the traveling wave equations for the capillary-gravity problem. In that case, the necessary condition is somewhat simpler and 
says that (roughly translating his work
into the language of our problem)
$\ds \int_\R R_{\mu,2} \sin( \tilde{\omega}_\mu x)dx = 0$. That is to say, the Fourier transform of his right hand side  has to be equal to zero at a particular frequency.
Nevertheless there is enough commonality with his work (and its successors \cite{amick-toland} \cite{FW}) that we are able to use his method, though there are some substantive complications.

The key idea of his method is to replace $\xib$ with $\etab+a \pb_\mu^a$ where $\etab$ is localized and 
$a \pb_\mu^a$ is an exact spatially periodic solution\footnote{Naturally a big part of our task is to show that such solutions exist; see Section \ref{periodic solutions}.} of \eqref{TWE5}
with amplitude $a$.
The frequency of this solution is very close to the asymptotic frequency $\tilde{\omega}_\mu$ of ${\mathfrak{z}}_\muo(x)$. The amplitude $a$ becomes our third variable.

Plugging ``Beale's ansatz," $\hb = \sigmab_c + \etab+a \pb_\mu^a$ into \eqref{TWE5}
and \eqref{NC}
gets us to
\begin{equation}\label{fail}
\H \eta_1 = \tilde{R}_{\mu,1}, \quad \U_\mu \eta_2 = \tilde{R}_{\mu,2} \mand 
\bunderbrace{\int_\R \tilde{R}_{\mu,2}{\mathfrak{z}}_\muo(x) dx}{\varsigma_\mu(a;\etab)}  =0.
\ee
Here $\tilde{R}_\mu=(\tilde{R}_{\mu,1},\tilde{R}_{\mu,2})$ is some complicated collection of terms much like those in $R_\mu$
from before, but now includes additional terms involving the periodic part $a\pb_\mu^a$.
The  equation $\varsigma_\mu(a;\etab) = 0$ is to be viewed as ``the equation for $a$."

A variation on the Riemann-Lebesgue lemma 
 can be used to show that $\varsigma_\mu(0;0) \le \O(\mu^{\infty})$.
So long as $\partial_a \varsigma_\mu(0;0) \ne 0$, the inverse function
theorem shows
 that we can solve $\varsigma_\mu(a;\etab)=0$ for $a$ given $\etab$.
Which in turn gives us reason to believe that we can solve $\U_\mu \eta_2 = \tilde{R}_{\mu,2}$ for $\eta_2$. From there 
we would have everything. 

This all worked in \cite{beale2} \cite{amick-toland} and \cite{FW}. But here something is different: 
$\partial_a \varsigma_\mu(0;0)$ is identically zero at a countable collection of values of $\mu$
converging to zero. 
Moreover, at values of $\mu$ ``away" from the zeros\footnote{
This is the origin of the set $M_c$ in Theorem \ref{nontech main}.}, we find that $\partial_a \varsigma_\mu(0;0) = \O(\mu^{1/2}).$ While this is small, it is still big enough to push through the inverse function theorem argument. But the size of a solution $a$ we get from the  inverse function theorem  is roughly the same size as $1/\partial_a \varsigma_\mu(0;0) = \O(\mu^{-1/2})$.

Plugging the selected value of $a$ into $R_{\mu,2}$ winds up generating terms 
which are linear in $\etab$ but come with a prefactor of $\mu^{1/2}$ instead of the $\O(\mu)$
prefactor we had originally in $R_{\mu,2}$.
Which is to say the linear parts in the right hand side of the light equation of \eqref{fail} are
still small, just not as small as we thought at the outset.
And here is the killer: the coercive estimate \eqref{U coerce} tells us that the size of the inverse of $\U_\mu$ on its range is $\O(\mu^{-1/2})$. Which means that the light equation, after inversion of $\U_\mu$, looks like $$
\eta_2 = \U_\mu^{-1} \tilde{R}_{\mu,2} = B_\mu \eta_2 + \textrm{residual and nonlinear terms}
$$
where $B_\mu$ is an $\O(1)$ linear operator. 

Unless we were so lucky as to have the operator norm of 
$B_\mu$ be strictly less than one (and we cannot get that by making $\mu$ small), 
it is not obvious how to solve this equation for $\eta_2$.
And we are all but certain that $B_\mu$, while $\O(1)$, is rather large in absolute terms. 
And there are other problematic terms lurking in $\tilde{R}_\mu$ as well. The general idea is this: {\it any $\O(\mu)$ term 
on the right hand side of the second equation in \eqref{fail} winds up being an $\O(1)$ term after solving for the amplitude $a$ and inverting $\U_\mu$.}

So we need to do something to get rid of all those deadly $\O(\mu)$ terms. We started the above line of reasoning at \eqref{first pass} and in that equation
we kept all the $\O(1)$ linear terms on the left (except the singular term) and put all $\O(\mu)$ terms on the right, along
with nonlinear terms. Our remedy is a rather brute force one:
we will  keep all  $\O(1)$ and most of the $\O(\mu)$ linear terms on the left and thus the right hand side will
no longer have the problematic $\O(\mu)$ linear parts; see Section \ref{beales ansatz}.
Once this is done, our general strategy remains as above, but to execute
it
we must
understand the $\O(\mu)$ linear perturbations of $\H$ and $\U_\mu$. The former is no problem, but the latter is quite complicated since it is where the singular term lives and, moreover, the perturbation is nonlocal;
see Appendices \ref{A appendix} and \ref{B appendix}. 
We also need to make a ``refined leading order limit" to get rid of $\O(\mu)$ residuals from the right hand side. 
That is to say, we need to modify $\sigmab_c$ slightly  to eliminate some problematic terms.
This is what we do next.

\section{Refined leading order limit.}\label{refined}

The estimates \eqref{LT is bounded} show that
$$
\| \G(\sigmab_c,\mu) \|_{H^s_b \times H^s_b} \le C\mu.
$$
Recall that  the traveling wave equation \eqref{TWE5} is $\G(\hb,\mu) = 0$.
The quantity  $\G(\sigmab_c,\mu)$ is the residual, or rather, the amount by which the leading
order term fails to solve the system at $\mu > 0$. 
We can shrink the size of the residual by slightly modifying $\sigmab_c$.

To this end, 
we let
$$
\G_{mod}(\hb,\mu):= c^2 I_0 \hb'' + L_\mu  \hb +L_\mu   Q_\mu (\hb, \hb). 
$$
All we have done here is remove from the singularly perturbed part from $\G(\hb,\mu)$.
We have $\G_{mod}(\sigmab_c,0) =\G(\sigmab_c,0)= 0.$ Computing the linearization of $\G_{mod}$ at $(\sigmab_c,0)$ we find
$$
D_{\hb} \G_{mod}(\sigmab_c,0) \xib = c^2 I_0 \xib'' + L_0 \xib + 2 L_0 Q_0(\sigmab_c,\xib)
= \left[ \begin{array} {cc} \H & 0 \\ 0 & 2 + 4\sigma_c \end{array}
\right] \xib
$$
where $\H$ is taken as above in \eqref{first pass 1}.

We saw in Proposition \ref{A props} that $\H$ was invertible 
and in Theorem \ref{sigma exists} that $\sigma_c(x)>0$. Therefore we can conclude
$D_{\hb} \G_{mod}(\sigmab_c,0)$ is an invertible map from $E^{s+2}_{b_c} \times O^{s}_{b_c}$ to $E^s_{b_c,0} \times O^s_{b_c}$. Though we do not show the details,
$G_{mod}(\hb,\mu)$ is a smooth mapping from $E^{s+2}_b \times O^{s}_b$ to $E^s_{b,0} \times O^s_b$. And so we can use the implicit function theorem to conclude that
there exists a smooth map $\mu \mapsto \sigmab_{c,\mu} \in E^{s+2}_{b_c} \times O^{s}_{b_c}$ for which $\G_{mod}(\sigmab_{c,\mu},\mu) = 0$ so long as $\mu$ is close enough to zero. The implicit function theorem also shows that $
\| \sigmab_{c,\mu} - \sigmab_c \|_{s,b_c}\le \O(\mu).
$

In summary we have
\begin{lemma}\label{refined lemma}
For all $|c| \in (c_0,c_1]$ there exists $\mu_{\xi} \in (0,1)$ 
for which
 $\mu \in (0,\mu_\xi)$ implies the existence of
$\ds \xib_\mu  \in \bigcap_{s \ge 0} \left (E^{s}_{b_c} \times O^s_{b_c}\right)$ such that $\sigmab_{c,\mu}:=\sigmab_c + \mu \xib_\mu$ satisfies
$$
\G_{mod}(\sigmab_{c,\mu},\mu) = 0
$$
and 
\be\label{xi estimate}
\| \xi_{\mu,1}\|_{s,b} 
+\| \xi_{\mu,2}\|_{s,b} 
\le C.
\ee
\end{lemma}

The important consequence of the above is that
\be\label{J0}
\G(\sigmab_{c,\mu},\mu) = \G(\sigmab_{c,\mu},\mu) - \G_{mod}(\sigmab_{c,\mu},\mu)= c^2 (I_\mu-I_0) \sigmab_{c,\mu}'' = c^2 \mu^2 \xi_{\mu,2}'' \jb.
\ee
That is to say, $\sigmab_{c,\mu}$ solves the first component \eqref{TWE5} {\it exactly} and it solves it to $\O(\mu^2)$ in the second component; the residual is thus $\O(\mu^2)$ which is good enough for what follows to work.
In our nanopteron solutions, $\sigmab_{c,\mu}$ will be the main piece of the localized part. 

\section{Periodic solutions}\label{periodic solutions}
In this section we state a theorem about the existence of spatially periodic traveling wave solutions
for \eqref{dn5}. After the statement we will provide an overview of its proof. The proof itself, which is
not short, is in Appendix \ref{per appendix}.

\begin{theorem} \label{periodic solutions exist}For all $|c|>c_0$ there exists
$\mu_\per > 0$ and $a_\per>0$  such that for all $\mu \in (0,\mu_\per)$
there exist $|\upsilon_\mu| \le 1$ and maps
\be\begin{split}
\omega^a_\mu&: [-a_\per,a_\per] \longrightarrow \R\\
\psi_{\mu,1}^a&: [-a_\per,a_\per] \longrightarrow C^\infty_\per\cap\left\{\text{even functions}\right\}\\
\psi_{\mu,2}^a&: [-a_\per,a_\per] \longrightarrow C^\infty_\per\cap\left\{\text{odd functions}\right\}
\end{split}\ee
with the following properties.
\begin{itemize}
\item Putting
\be\label{this is p}\hb(x)=
a\pb^a_\mu(x):=\left(   \begin{array}{c}
\mu \varphi_{\mu,1}^a(x)\\\varphi_{\mu,2}^a(x) \end{array} \right):=
a\left(   \begin{array}{c}
\mu \upsilon_\mu \cos(\omega_\mu^a x)\\
\sin(\omega_\mu^a x) 
\end{array}
\right)+a\left( \begin{array}{c}
 \mu \psi^a_{\mu,1}(\omega_\mu^a x)\\
 \psi^a_{\mu,2}(\omega_\mu^a x)
\end{array}
\right)
\ee
solves \eqref{TWE5} for all $|a|\le a_\per$.
That is
\be\label{varphi solves}
\G(a \pb_\mu^a,\mu) = 0.
\ee
\item $\omega_\mu^0 = \omega_\mu$ 
where $\omega_\mu = \O(\mu^{-1/2})$. The mapping $\mu \mapsto \omega_\mu$ is smooth with respect $\mu$.
\item $\psi^0_{\mu,1} = \psi^0_{\mu,2} = 0$.
\item For all $s \ge0$, there exists $C>0$ such that for all $|a|,|\ga|\le a_\per$ we have
\be\label{periodic lip}
\left \vert  \omega_\mu^a -\omega_\mu^\ga\right \vert +
 \left\|  \psi^a_{\mu,1} - \psi^\ga_{\mu,1} \right\|_{W^{s,\infty}_\per} +
 \left\|  \psi^a_{\mu,2} - \psi^\ga_{\mu,2} \right\|_{W^{s,\infty}_\per} 
 \le C|a-\ga|.
\ee
\end{itemize}
\end{theorem}

Here is the formal argument for the proof; it will highlight from where we get the ``critical frequency"
$\omega_\mu$ and also the leading order part of $\pb_\mu^a$. The key is, of course, linear theory.
We want to find
small amplitude periodic solutions of $\G(\hb,\mu)$ for a fixed value of $\mu$. We make the substitution
$\hb(x) = a \pb(x)$ where we are thinking of $a\ne0$ as being small and $\pb$ as a periodic function
of an as yet unspecified frequency but (more or less) unit amplitude. If we put this into \eqref{TWE5} we find that,
after canceling one factor of $a$ from all terms:
$$
\bunderbrace{c^2 I_\mu \pb'' + L_\mu \pb}{\B_\mu \pb} + a L_\mu Q_\mu(\pb,\pb) = 0.
$$
The obvious starting point is to put $a = 0$; naturally this implies that 
$\B_\mu \pb = 0$. So the question becomes: are there periodic solutions of the linearized problem?
The answer is yes, but only for certain frequencies. 
We make the guess that 
$\pb(x) = \ub e^{i \omega x} + {\text{c.c.}}
$
for an undetermined vector $\ub \in \C^2$ and frequency $\omega$. 

We know that
 $\partial_x e^{i \omega x} = i\omega e^{i \omega x}$ and $S^d e^{i\omega x} = e^{i\omega d} e^{i \omega x}$.
Thus we discover that
$$
\B_\mu (\ub e^{i \omega x}) = (\tilde{\B}_\mu(\om) \ub)  e^{i\om x}
$$
where
$$
\tilde{\B}_\mu(\om):= -c^2\om^2  I_\mu+ \tilde{L}_\mu(\omega)
$$
and 
\be\label{Lmu symbol}
 \tilde{L}_\mu (\om) :=
 \left[ \begin{array}{cc} 2  \sin^2(\om)(1 - \mu \cos^2(\om) )
 &
 -\mu i  \sin(2\om) (1 - 2 \cos^2(\om) - {(\mu/ 4)}  \sin^2(2\om))\\
 i \mu \sin(2\om) &
 2 \left(1+\mu \cos^2(\om) + {(\mu^2/4)}\sin^2(2\om)  \right)
 \end{array}
 \right].
\ee
Finding nontrivial periodic solutions of the linear part is equivalent to finding $\omega$
so that $\tilde{\B}_\mu(\omega)$ has a nontrivial kernel. 
Which is to say when
\be\label{det eqn}
\det\left(-c^2 \om^2 I_\mu+ \tilde{L}_\mu(\omega)\right) = 0.
\ee
Some factoring and quite a few trigonometric identities reveal that if 
\be\label{crit freq}
c^2 \mu \omega^2 =\bunderbrace{1+\mu + \sqrt{(1+\mu)^2-4 \mu \sin^2(\omega)}}{\lambda_\mu^+(\om)}
\ee
then we have \eqref{det eqn}.

In Figure \ref{critfreq} we sketch $\lambda_\mu^+(\omega)$
and $c^2 \mu \omega^2$ vs $\omega$.  
\begin{figure}
\centering
    \includegraphics[width=6.5in]{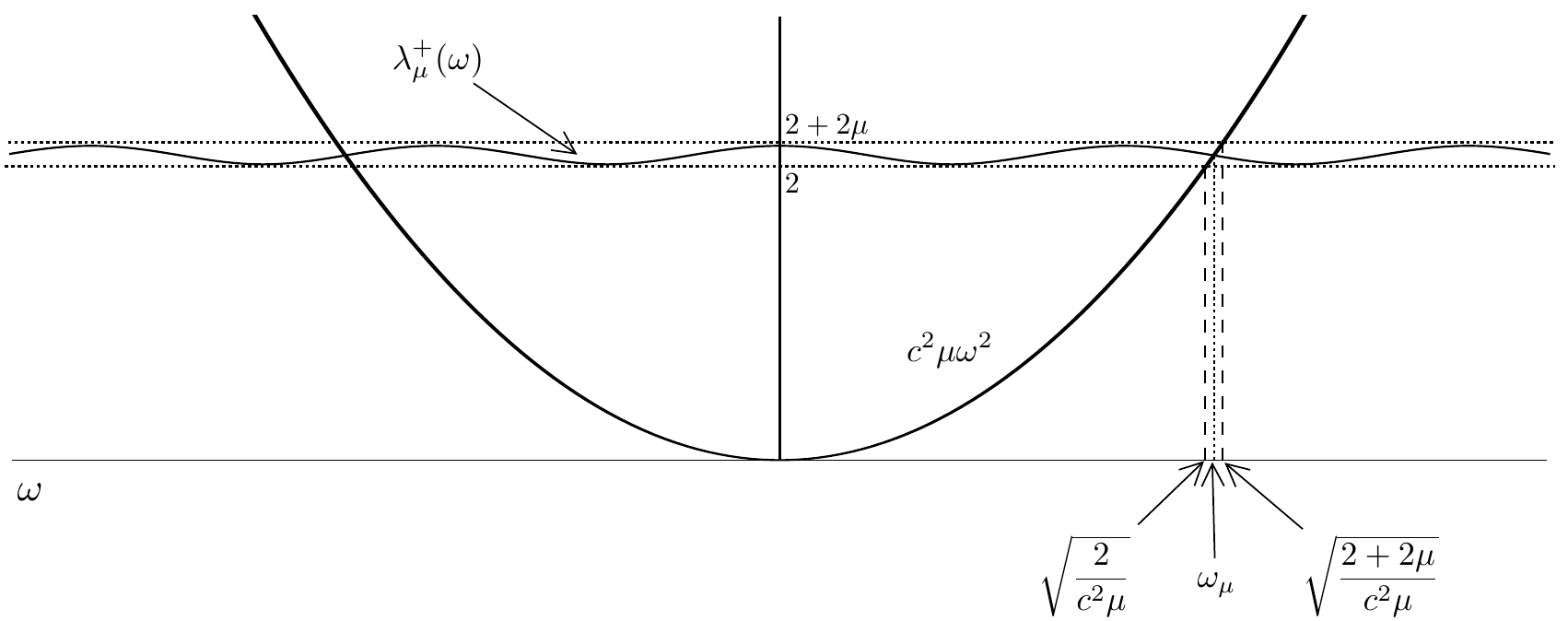}
 \caption{  \it Sketch of the graphs of  $\lambda_\mu^+(\omega)$
and $c^2 \mu \omega^2$ vs $\omega$; their intersection
determines the critical frequency $\omega_\mu$.
 \label{critfreq}}
\end{figure}
This sketch is representative of the general situation when
$c>c_0$ and $\mu$ sufficiently close to zero. Specifically, we have
\begin{lemma}\label{omega lemma}
For all $|c|>c_0$ there exists $\mu_\omega\in (0,1)$ such that
for all  $\mu \in (0,\mu_\omega)$
there is a unique nonnegative 
number $\omega_\mu$ for which putting $\omega = \pm \omega_\mu$ solves \eqref{crit freq}.
Moreover we have 
\be\label{omega est}
\sqrt{2 \over c^2 {\mu}} \le \omega_\mu \le \sqrt{2+2\mu \over c^2 {\mu}} .
\ee
That is to say, $\om_\mu = \O(\mu^{-1/2})$. Lastly, the map $\mu \mapsto \omega_\mu$ is $C^\infty$.
\end{lemma}

One establishes this rigorously using differential calculus, but the picture tells the main story so we omit the proof.
It is not so different than the proof of Lemma 2.1 {\it (vi)} in \cite{FW}.
It is worth pointing out the the asymptotic frequency, $\tilde{\omega}_\mu$, of the Jost solutions described in Section \ref{strategy} is the left hand endpoint of \eqref{omega est}.

When $\omega = \omega_\mu$, we can find two linearly independent solutions of
$\B_\mu (\om_\mu) \pb = 0$ and, more importantly, we can make a linear combination of these that 
has the even $\times$ odd symmetry we require. It is
$$
\pb_\mu^0(x) := \nub_\mu(\omega_\mu x):=\left(   \begin{array}{c}
\mu \upsilon_\mu \cos(\omega_\mu x)\\
\sin(\omega_\mu x) 
\end{array}
\right).
$$
Here
$\upsilon_\mu \le \O(1)$ is given by the rather dreadful formula \eqref{this is upsilon} below. 

In this way we see that $a \pb_\mu^0(x)$ is a good first guess for a small amplitude
spatially periodic solution of \eqref{TWE5}. To get a rigorous result we employ a Liapunov-Schmidt decomposition
which is more or less just a quantitative version of the method of ``bifurcation from a simple eigenvalue" developed in
\cite{crandall-rabinowitz} and \cite{zeidler1}. The main difficulty is not in getting existence of a solution, but rather getting  the $\mu$-uniform estimates in Theorem
\ref{periodic solutions}. The interested reader can
at this point jump to Appendix \ref{per appendix} to see the gory bits, but for now we return to the search for nanopterons.

\section{Nanopteron ansatz and governing equations}\label{beales ansatz}
We briefly recapitulate where are in our search for traveling wave solutions of \eqref{dn5}.
We have a refined leading order 
localized solution $\sigmab_{c,\mu}$ for which we know that $\|\G(\sigmab_{c,\mu},\mu)\|_{E^s_b \times O^s_b}$
is  $ \O(\mu^2)$.
It is the sum of the $\O(1)$ monatomic solitary wave $\sigmab_{c}$ and an $\O(\mu)$ correction.
And we have, for each $\mu>0$, a family of high frequency periodic solutions parameterized by their amplitude $a$ which we denote by
$a \pb_\mu^a(x)$. Specifically, $\G(a \pb_\mu^a,\mu) =0$.

The solutions we seek are the sum of $\sigmab_{c,\mu}$, $a \pb_\mu^a$ and a small,  localized remainder. Specifically
we make the ansatz
\be\label{ba}
\gb = \sigmab_{c,\mu} + a \pb^a_\mu + \etab
\ee
and plug this
into \eqref{TWE5}. To be clear, our unknowns are $\etab\in E^s_b \times O^s_b$ and the amplitude $a$.

Some algebra leads us from \eqref{TWE5} to the following system
\bes
c^2 I_\mu \etab'' + L_\mu \etab + 2 L_\mu Q_\mu(\sigmab_{c,\mu},\etab)=J_0+J_3 +J_4 + J_5
\ees
with
\be\label{J0345}
\begin{split}
J_0:=&-\G(\sigmab_{c,\mu},\mu)\\
J_3:=&-2a L_\mu Q_\mu(\sigmab_{c,\mu},\pb_\mu^a)\\
J_4:=& -2aL_\mu Q_\mu(\pb_\mu^a,\etab) \\
\mand J_5:=&- L_\mu Q_\mu(\etab,\etab).
\end{split}\ee
The defintions for $I_\mu$, $L_\mu$, $L_\mu Q_\mu$ and $\G$ are found in Section \ref{equation of motion}.

Note that
\begin{itemize}
\item
$J_0$ is the residual from the refined leading order limit (see \eqref{J0}) and thus is $\O(\mu^2)$.
\item
The term $J_4$ is nonlinear because it consists of terms that look like ``$a \cdot \etab$."
\item $J_5$
is quadratic in $\etab$.  \end{itemize}
Thus $J_0$, $J_4$ and $J_5$ can viewed as being small.
The term $J_3$ encapsulates the interaction between the periodic part and the leading order localized part: it will be
at the center of much of our analysis. 

It will be advantageous to move the ``off-diagonal" parts of the left hand side over to the right.\footnote{We did discuss in Section \ref{strategy} how we wanted
keep the $\O(\mu)$ linear parts on the left; for technical reasons these $\O(\mu)$ off-diagonal parts do not create the problems we described there.}
So we denote the off-diagonal part of $L_\mu$ as:
$$
\mu \Theta_\mu:=\left[
\begin{array}{cc}
0 & \mu \Theta_{\mu,1} \\
\mu \Theta_{\mu,2} & 0
\end{array}
\right]:=
\left[ \begin{array}{cc}
0 & -2 \mu A \delta (1- 2 A^2 + \mu A^2) \\
2 \mu A \delta & 0 
\end{array}
\right].
$$
We know from \eqref{Ad is bounded} that $A$ and $\delta$ are $\O(1)$ operators on $W^{s,p}_b$ spaces and thus
we have
\be\label{Theta is bounded}
\| \Theta_{\mu,1}\|_{W^{s,p}_b \to W^{s,p}_b} \le C\mand
\| \Theta_{\mu,2}\|_{W^{s,p}_b \to W^{s,p}_b} \le C.
\ee
Which is to say that $\mu \Theta_\mu$ is in fact $\O(\mu)$.

The diagonal part of $2 L_\mu Q_\mu(\sigmab_{c,\mu},{\bf f})$ is
\be\label{Sigmas}
\Sigma_\mu {\bf f}:= 
[\bunderbrace{2L_\mu Q_\mu(\sigmab_{c,\mu},f_1 \ib)\cdot \ib}{\Sigma_{\mu,1} f_1}]\ib 
+ [\bunderbrace{2L_\mu Q_\mu(\sigmab_{c,\mu},f_2 \jb)\cdot \jb}{\Sigma_{\mu,2}f_2}]\jb .
\ee
The corresponding off-diagonal part is
$
\ds\mu \Omega_\mu {\bf f}
:=
2 L_\mu Q_\mu(\sigmab_{c,\mu},{\bf f}) - \Sigma_\mu {\bf f}.
$
More explicitly, this is
\be\label{Omegas}
\mu \Omega_\mu {\bf f}= 
[\bunderbrace{2L_\mu Q_\mu(\sigmab_{c,\mu},f_2 \jb)\cdot \ib}{\mu \Omega_{\mu,1} f_2}]\ib 
+ [\bunderbrace{2L_\mu Q_\mu(\sigmab_{c,\mu},f_1 \ib)\cdot \jb}{\mu \Omega_{\mu,2}f_1}]\jb .
\ee

In Lemma \ref{refined} we saw that
 $\sigmab_{c,\mu} = \sigma_c \ib + \mu \xib_\mu \in E^s_{b_c} \times O^s_{b_c}$. 
Thus if we use the bilinear estimate \eqref{E1} on $\Sigma_{\mu,1}$ we get
$$
\| \Sigma_{\mu,1} f\|_{s,b} \le C
\|\sigma_c+\mu \xi_{\mu,1}\|_{s,b_c}\|f\|_{W^{s,\infty}_{b-b_c}}
+ C\mu^2 
\|\xi_{\mu,2}\|_{s,b_c}\|f\|_{W^{s,\infty}_{b-b_c}}.
$$
Using the estimates for $\sigma_c$ in \eqref{sigma is localized} and those 
 for $\xi_{\mu,1}$ and $\xi_{\mu,2}$ in \eqref{xi estimate}
then gives
$
\| \Sigma_{\mu,1} f\|_{s,b} \le C
\|f\|_{W^{s,\infty}_{b-b_c}}.
$
The same sort of reasoning using \eqref{E1}-\eqref{E4} leads to
\be\begin{split}\label{sigma estimate}
&\| \Sigma_{\mu,1} f\|_{s,b} + \| \Sigma_{\mu,2} f\|_{s,b} +\| \Omega_{\mu,1} f\|_{s,b} + \| \Omega_{\mu,2} f\|_{s,b} \le C\|f\|_{s,b-b_c}\\
\mand &\| \Sigma_{\mu,1} f\|_{s,b} + \| \Sigma_{\mu,2} f\|_{s,b}+\| \Omega_{\mu,1} f\|_{s,b} + \| \Omega_{\mu,2} f\|_{s,b}  \le C\|f\|_{W^{s,\infty}_{b-b_c}} .
\end{split}\ee
Note that these estimates impute to $\Sigma_{\mu,1}$, $\Sigma_{\mu,2}$, 
$\Omega_{\mu,1}$ and $\Omega_{\mu,2}$ a ``localizing" property:
$\Sigma_{\mu,1} f$ and its friends decay more rapidly at infinity than $f$ does. This is important!

With these definitions we have
\be\label{second pass}
\bunderbrace{c^2 I_\mu \etab'' + (L_\mu-\mu\Theta_\mu) \etab +\Sigma_\mu \etab}{\text{these terms are diagonal}}=J_0+J_2 +J_3 +J_4 + J_5
\ee
with
\be\label{J2}
J_2:=-\mu\Theta_\mu \etab - \mu\Omega_\mu\etab.
\ee

Written out, the first component of \eqref{second pass} reads
\be\label{heavy equation v1}
 \bunderbrace{c^2 \eta_1''  -2 \delta^2 (1-\mu A^2)\eta_1
+ \Sigma_{\mu,1} \eta_1}{\H_\mu \eta_1} = j_2 + j_3 + j_4 + j_5.
\ee
Here, we use the notation
\be\label{little j}
j_k:=J_k \cdot \ib
\ee
to denote the first components of the $J_k$ maps.
Note that the first component
of $J_0$ is identically zero as can be seen in \eqref{J0}. 
This equation is our ``higher order" version of the heavy equation \eqref{first pass 1}.

A direct computation shows that $\H_0 = \H$ from \eqref{first pass 1} and 
the estimates \eqref{Ad is bounded} and \eqref{LT is bounded} give us
$$\| \H_\mu - \H_0\|_{H^s_{b} \to H^s_{b}} \le C\mu.$$
A Neumann series argument  allows us to conclude that the results of Proposition~\ref{A props} 
can be extended to $\H_\mu$. That is to say we have

\begin{lemma}\label{Amu props} 
For all $|c| \in (c_0,c_1]$ and $b \in (0,b_c]$
there exists $\mu_{\H}(b) \in (0,1)$  for which $\mu \in (0,\mu_{\H}(b))$  
implies that 
the map $\H_\mu$ is a homeomorphism of $E^{s+2}_b$ and $E^s_{b,0}$. Moreover we have
 \be\label{Amu estimate}
\| \H_\mu^{-1} \|_{E^{s}_{b,0} \to E^{s+2}_b} \le C_b.
\ee
In addition, $\mu_\H(b)$ is nondecreasing as a function of $b$.
\end{lemma}
Noting that 
$j_2,$ $j_3$, $j_4$ and $j_5$ are even and mean-zero by virtue of the symmetries 
described at \eqref{G sym}, we are free to use Lemma \ref{Amu props}
 to rewrite \eqref{heavy equation v1} as 
\be\label{heavy equation v2}
\eta_1 = \bunderbrace{\H_\mu^{-1} \left( j_2 + j_3 + j_4 + j_5\right)}{N_1^\mu(\eta_1,\eta_2,a)}.
\ee

The second component of \eqref{second pass} reads
\be\label{light equation v0}
c^2 \mu \eta_2'' +
2(1 + \mu A^2 - \mu^2 A^2 \delta^2)
\eta_2 + \Sigma_{\mu,2} \eta_2 = l_0+l_2+l_3+l_4+l_5.
\ee
Here, 
\be\label{little l}
l_k:=J_k \cdot \jb
\ee
are the second components of the $J_k$ maps.

We need to make one more change to the left hand side of this equation. 
This change is made for a rather subtle reason that will not
pay off for quite some time. Let 
\be\label{this is tau}
\tau_\mu:={1 \over 2 \mu^2} \left( c^2 \mu \omega_\mu^2 - 2 - 2 \mu \cos^2(\omega_\mu)\right)
\ee
where $\omega_\mu$ is the periodic solutions' critical frequency. Remembering that $\omega_\mu$
is the solution of \eqref{crit freq}, Taylor's
theorem shows that \be\label{tau is small}
|\tau_\mu| \le C.
\ee

Then we put
\be\label{this is l1}
l_1:=\mu^2 (2\tau_\mu +2A^2 \delta^2) \eta_2
\ee
and see that \eqref{light equation v0} is equivalent to
\be\label{light equation v1}
\bunderbrace{c^2 \mu \eta_2'' +
2(1 + \mu A^2 + \mu^2 \tau_\mu)
\eta_2 + \Sigma_{\mu,2} \eta_2}{\L_\mu \eta_2} = l_0+l_1+l_2+l_3+l_4+l_5.
\ee
This equation is our ``higher order" analog of the light equation \eqref{first pass 2}.

To summarize, our goal is now to solve
\be\label{getting getting}
\eta_1 = N_1^\mu(\eta_1,\eta_2,a) \mand \L_\mu\eta_2 = l_0+l_1+l_2+l_3+l_4+l_5.
\ee
This seems like two equations with two unknowns, but recall
that buried in the $j$'s and $l$'s are functions which depend on the third variable, the periodic amplitude $a$.
We discuss how to get a third equation in the next section, which is focused on the properties
of the light operator $\L_\mu$.

\section{The light operator $\L_\mu$ and the final system.}\label{light operator}
\subsection{Properties of $\L_\mu$.}
The first thing to observe is that $\L_\mu$ is a nonlocal perturbation of the Schr\"odinger map $\U_\mu$
from Section \ref{strategy}.
As with its forebear, $\L_\mu$ is injective but not surjective.
We describe now the key properties of $\L_\mu$ in a series of lemmas. The proofs of these
lemmas are carried out in  Appendix \ref{B appendix}.

The first result we have for $\L_\mu$ is a collection of coercive estimates.
\begin{lemma} \label{coerce lemma} 
For all $|c| \in (c_0,c_1]$ and $b \in (0,b_c]$
there exists $\mu_{\L}(b) \in (0,1)$ for which $\mu \in (0,\mu_{\L}(b))$ and $f \in O^{s+2}_b$ imply
\be\label{B coerce}
\| f \|_{s+k,b} \le {C_b \mu^{-(k+1)/2}} \| \L_\mu f\|_{s,b}.
\ee
In the above we require
 $ -1 \le k \le 2$ and $s+k \ge 0$.
 In addition, $\mu_\L(b)$ is nondecreasing as a function of $b$.
\end{lemma}

This lemma implies that, when viewed as an operator from $O^{s+2}_{b}$ to $O^{s}_{b}$,
$\L_\mu$ is injective,
since it tells us that if $\L_\mu g = 0$ and $g \in O^{s+2}_b$, then $g(x) \equiv 0$.
The next result tells us that $\L_\mu$ is not surjective.

\begin{lemma} \label{range lemma} For all $|c| \in (c_0,c_1]$ 
there exists $\mu_\gamma \in (0,1)$ such that $\mu \in (0,\mu_\gamma)$  
implies the existence of a 
 nonzero, smooth, bounded, odd function $\gamma_\mu(x)$ for which
\be\label{B solvability}
f \in O^{s+2}_{b},\quad  g \in O^s_{b}\mand \L_\mu f = g\implies\bunderbrace{\int_\R g(x) \gamma_\mu(x) dx}{\iota_\mu[g]} =0.
\ee
In the above we require $b \in (0,b_c]$.
Moreover, for $ b \in (0,b_c]$, if $\mu \in (0,\min\{\mu_\gamma,\mu_\L(b)\})$ 
we have the reverse implication:
\be\label{B solvability suff}
g \in O^s_{b} \mand \iota_\mu[g] = 0 \implies \text{there exists unique } f\in O^{s+2}_{b}\text{ such that } \L_\mu f = g.
\ee
In this case we write ``$f = \L_\mu^{-1} g$."
\end{lemma}

The coercive estimate \eqref{B coerce} implies that, when defined, the size of $\L_\mu^{-1}$ (viewed as a map from $O^s_b$ to itself) is $\O_b(\mu^{-1/2})$.

The function $\gamma_\mu(x)$, which defines the range of $\L_\mu$, is 
utterly central to our analysis.
It is very much 
 like the Jost solution ${\mathfrak{z}}_\muo(x)$ described in Section \ref{strategy}.
  For instance, it is asymptotic at infinity to a sinusoid. 
    Importantly, we have 
 quantitative estimates on the amplitude, asymptotic frequency and phase shift of this function.
\begin{lemma}\label{gamma props}
For all $|c|\in(c_0,c_1]$ the function $\gamma_\mu(x)$ defined in Lemma \ref{range lemma} has the following features.
\begin{itemize}
\item There are constants $0<C_1<C_2$ and a map
\be\label{vartheta is 1}
\begin{split}
(0,\mu_\gamma)&\longrightarrow [C_1,C_2]\\
\mu& \longmapsto \vartheta_\mu^\infty
\end{split}
\ee
such that
\be\label{gamma limit}\begin{split}
&\lim_{x \to \infty}
|\gamma_\mu(x)   -\sin(\omega_\mu(x+\vartheta^\infty_\mu))|  \\= &\lim_{x \to\infty}|\gamma'_\mu(x) -  \omega_\mu \cos(\omega_\mu(x+\vartheta^\infty_\mu))|\\ =& 0 
\end{split}
\ee
when $\mu \in  (0,\mu_\gamma)$.
\item There exists an open set $M_c \subset \R_+$ for which $0 \in \overline{M}_c$ and 
\be
\label{Mc def}
\mu \in M_c \implies |\sin(\om_\mu \vartheta^\infty_\mu)| > {1 /2}.
\ee
\item For $s =0$ or $s=1$ we have
\be\label{gamma est}
\| \gamma^{(s)}_\mu \|_{L^\infty} \le C \mu^{-s/2}.
\ee
\item For all $s \ge 0$ and $b \in (0,b_c]$ we have
\be\label{stationary phase}
\left \vert \iota_\mu[g]\right \vert \le C_b \mu^{s/2} \|g\|_{s,b}
\ee
for all $\mu \in (0,\mu_\gamma)$.
\end{itemize}

\end{lemma}

Observe that \eqref{gamma limit} states that $\gamma_\mu(x)$ is asymptotic as $x \to \infty$ to a sinusoid
which has frequency $\omega_\mu$, the critical frequency for the periodic solutions. It turns out that
this is quite an important feature of $\gamma_\mu(x)$ and  is why we had to make the ``$\tau_\mu$" adjustment to the left hand side of the light equation back at \eqref{this is tau}.

\subsection{Amplitude selection and the final system}
Now that we have a better understanding of $\L_\mu$, we 
can finally close our underdetermined system \eqref{getting getting}. The solvability
condition \eqref{B solvability} in Lemma \ref{range lemma} tells us that  solving the second equation in \eqref{getting getting}
requires that we have
\be\label{solve v1}
\iota_\mu [ l_0 + l_1 + l_2 + l_3 + l_4+l_5] =0.
\ee
This is our 
third equation and in particular we can view it as an equation for the amplitude~$a$.

We could, following the example set in \cite{beale2} or \cite{amick-toland} and desribed in Section \ref{strategy}, use the inverse function theorem to solve \eqref{solve v1} for $a$ given $\etab$. Success in that venture requires
$$
\bunderbrace{-\partial_a\big\vert_{\etab=0,a=0} \left(\iota_\mu [ l_0 + l_1 + l_2 + l_3 + l_4+l_5] \right)}{\kappa_\mu}\ne0.
$$
The negative sign is just there for convenience later. 
An examination of $l_0$, $l_1$, $l_2$ and $l_5$ (defined in Section \ref{beales ansatz}) shows that these do not depend on $a$ and thus do not contribute to $\kappa_\mu$. Likewise, $l_4$ is, roughly speaking, $``a \cdot \etab"$ and so it also does not contribute. The only piece that does is $l_3$. Carrying out the  differentiation shows
\be\label{kappa}
\kappa_\mu=-\partial_a\big\vert_{\etab=0,a=0} (\iota_\mu[l_3]) = \iota_\mu[\chi_\mu]
\ee
where
\be\label{chi}
\chi_\mu(x) := -\partial_a\big\vert_{\etab=0,a=0} \left( l_{3} \right)= \Sigma_{\mu ,2} \varphi_{\mu,2}^0 + \mu^2 \Omega_{\mu,2} \varphi_{\mu,1}^0.
\ee

Remarkably, we can find an exact formula for the leading order term of $\kappa_\mu$. 
Here is the result, which we verify in Appendix \ref{kappa appendix}.
\begin{lemma} \label{kappa lemma} For all $|c| \in (c_0,c_1]$ there exist $\mu_\kappa \in (0,1)$ for which
\be\label{kappa est}
|\kappa_\mu - 2 c^2 \mu \omega_\mu\sin(\omega_\mu \vartheta^\infty_\mu)| \le C\mu
\ee
holds when $\mu \in (0,\mu_\kappa)$.
\end{lemma}

This, in combination with \eqref{Mc def},
implies
\be
\mu \in M_c \cap(0,\mu_\kappa)\implies 
C^{-1} \mu^{1/2} \le |\kappa_\mu| \le C \mu^{1/2}
\ee
for some constant $C>1$ which does not depend on $\mu$. 
Therefore, when $\mu \in M_c \cap(0,\mu_\kappa)$ we can invoke the inverse function theorem to solve \eqref{solve v1} for $a$ given $\etab$.
Nonetheless we do not directly employ the inverse function theorem here. We require some uniformity with respect to $\mu$  in the resulting solution map that 
we would not find in an ``off the shelf" version of that theorem. 
Instead we proceed along the lines of the proof the inverse function theorem.

Recalling how we defined $\chi_\mu$ in \eqref{chi}, we write $l_3$ as the part  which is linear in $a$
plus a remainder:
\be\label{this is l31}
l_3 = -a \chi_\mu+l_{31}.
\ee
We leave $l_{31}$ implicitly defined.
It  is small in the sense that it is more or less quadratic in $a$.
Since we know now that $\kappa_\mu$ is non-zero for $\mu \in M_c$ we see that \eqref{solve v1} can be rewritten as
\be\label{a eqn}
a =\bunderbrace{ { 1 \over \kappa_\mu}\iota_\mu \left[ l_0 + l_1 + l_2 + l_{31} + l_4 + l_5 \right] }{N_3^\mu(\eta_1,\eta_2,a)}.
\ee
This is the ``equation for $a$."

Next we define, for $f \in O^s_b$:
\be\label{this is Pmu}
\P_\mu f:= f - {1 \over \kappa_\mu} \iota_\mu[f] \chi_\mu.
\ee
This is a projection onto the range of $\L_\mu$ since, by
construction, we
have
 \be\label{construction}
 \iota_\mu[ \P_\mu f] = 0\mand \P_\mu \chi_\mu = 0.
 \ee
 Which is to say that an invocation of Lemma \ref{range lemma} allows us 
to solve $\L_\mu f = \P_\mu g$ for $f$ given $g \in O^s_{b}$. We denote the solution by
$$
f = \L_\mu^{-1} \P_\mu g.
$$
The following lemma, proved in Appendix \ref{kappa appendix}, gives us estimates on the size $\L^{-1}_\mu \P_\mu$.
\begin{lemma}\label{L inv lemma}
For all $|c| \in (c_0,c_1]$ and
$b \in (0,b_c]$, if  $\mu \in (0,\min\left\{\mu_\kappa,\mu_\L(b)\right\}) \cap M_c$ 
we have
\be\label{B inv estimate}
\| \L_\mu^{-1} \P_\mu\|_{O^s_{b} \to O^{s+k}_{b}} \le C_b \mu^{-(2+k)/2} . 
\ee
In the above we require $|k|\le2$ and $s+k\ge 0$.
\end{lemma}
Putting $k = 0$ in the above gives the norm of $\L_\mu^{-1} \P_\mu$ (viewed as a map from $O^s_b$ to itself) as being $\O_b(\mu^{-1})$.
This is a factor of $\mu^{-1/2}$ larger  than the  size of $\L_\mu^{-1}$ implied by the coercive 
estimate \eqref{B coerce}. That is to say, the projection $\P_\mu$ is a large operator; this is the same phenomenon descibed
in Section \ref{strategy} wherein ``solving for $a$" renders certain terms much larger than they at first appeared.

Given \eqref{this is l31} and \eqref{construction}, we see that
$$
\P_\mu (l_0+l_1 + l_2 + l_3 + l_4 + l_5) = \P_\mu  (l_0 +l_1 + l_2 + l_{31} + l_4 + l_5).
$$
Also note that
 \eqref{a eqn}, since it is derived from \eqref{solve v1}, is equivalent to
  $$\P_\mu( l_0+l_1 + l_2 + l_3 + l_4 + l_5) = l_0+l_1+l_2+l_3 + l_4+ l_5.$$

And so if we revisit the second equation in our system \eqref{getting getting} we see that
$$
\L_\mu \eta_2 =  l_0+l_1+l_2+l_3 + l_4+ l_5 = \P_\mu( l_0+l_1 + l_2 + l_{31} + l_4 + l_5)
$$
or rather 
\be\label{N2}
\eta_2 = \bunderbrace{\L_\mu^{-1} \P_\mu( l_0+l_1 + l_2 + l_{31} + l_4 + l_5)}{N_2^\mu(\eta_1,\eta_2,a)}.
\ee
 
Summing up: we will have found a solution of \eqref{TWE5} of the form \eqref{ba},
and thus a nanopteron solution of \eqref{dn5},  if we can find a solution of
\be\label{NE4}\begin{split}
\eta_1=N_1^\mu(\eta_1,\eta_2,a), \quad
\eta_2=N_2^\mu(\eta_1,\eta_2,a) \mand
a = N_3^\mu(\eta_1,\eta_2,a) 
\end{split}
\ee
where $N_1^\mu$ is given in \eqref{heavy equation v2}, $N_2^\mu$ in \eqref{N2} and $N_3^\mu$ in \eqref{a eqn}.
When convenient we compress the above
as
$$
(\etab,a) = N^\mu(\etab,a)
$$
where $N^\mu := (N_1^\mu,N_2^\mu,N_3^\mu)$.

\section{Existence  of solutions and properties thereof.}\label{enchilada}
Note that \eqref{NE4} has the form of a fixed point equation,
and we will look
for the solution $(\etab_\mu,a_\mu)$
in
\be\label{X1}
X_1:=E^2_{b_*}\times  O^1_{b_*} \times \R.
\ee
Here we have taken
 \be\label{bstar}
 \ b_*:={b_c /2}.
 \ee
The mismatch in regularity in the defintion of $X_1$ is done for technical reasons. Ultimately our solutions will be smooth by virtue of a bootstrapping argument.
Note that $X_1$ is a Hilbert space. 

We also define
$$
X_0:=E^2_{b_*/2} \times O^0_{b_* / 2} \times \R
$$
and, for $s\ge1$:
$$
X_s:=E^{s+1}_{b_*}\times  O^s_{b_*} \times \R. 
$$
These are also Hilbert spaces.
For  ordered triples $\r=(\r_1,\r_2,\r_3) \in \R_+^3$ we put
$$
U^s_{\mu,\r}:=\left\{ 
(\etab,a) \in X_s : \| \eta_1\|_{s+1,b_*} \le \r_1 \mu^3,\ \| \eta_2\|_{s,b_*}\le \r_2 \mu^2\ \text{and } |a| \le \r_3 \mu^{(6+s)/2}
\right\}.
$$
These are ``balls" in $X_s$ with $\mu$-dependendent radii; note that the larger $s$ is, the smaller is the $a$-part.

The following lemma contains all the information we need to prove our main result.
\begin{lemma} \label{mover}
For all $|c| \in (c_0,c_1]$ there exists $\mu_* \in (0,1)$ for which the following hold.
\begin{itemize} \item For all $s \ge 1$ and $\mu\in (0,\mu_*)\cap M_c$ we have
\be\label{boot}
(\etab,a) \in E^{s+1}_{b_*} \times O^s_{b_*} \times [-a_\per,a_\per] \
\implies N^\mu(\etab,a) \in E^{s+3}_{b_*} \times O^{s+2}_{b_*} \times \R.
\ee
\item
For all $s \ge 1$ and $\r \in \R_+^3$ (with $\r_3 < a_\per$)
there exists $\tilde{\r} \in \R_+^3$ such that for all $\mu \in (0,\mu_*)\cap M_c$ we have
\be\label{really small}
(\etab,a) \in U^s_{\mu,\r} \implies N^\mu(\etab,a) \in  U^{s+1}_{\mu,\tilde{\r}}.
\ee
\item There exists $\r_* \in \R_+^3$ such that such that for all $\mu \in (0,\mu_*)\cap M_c$ we have
\be\label{map to self}
(\etab,a) \in U^1_{\mu,\r_*} \implies N^\mu(\etab,a) \in  U^{1}_{\mu,\r_*}
\ee
and
\be\label{L2 contract}
(\etab,a),(\grave{\etab},\ga) \in U^1_{\mu,\r_*} \implies \|N^\mu(\etab,a) - N^\mu(\grave{\etab},\ga)\|_{X_0} \le {1 \over 4} \| (\etab,a) - (\grave{\etab},\ga)\|_{X_0}.
\ee
\end{itemize}

\end{lemma}
The proof of this is quite involved and is found in Appendix \ref{estimate appendix}. What this lemma is saying is that $N^\mu$
 is something akin to a contraction on $X_1$. It proves
\begin{theorem}\label{main result} 
For $c \in (c_0,c_1]$ and $\mu \in (0,\mu_*) \cap M_c$ there exists unique $(\etab_\mu,a_\mu) \in U^1_{\mu,\r_*} \cap \left( C^\infty(\R) \times C^\infty(\R) \times \R\right)$ for which $(\etab_\mu,a_\mu) = N^\mu(\etab_\mu,a_\mu)$.
Moreover,
 for all $s \ge 0$ there exists 
$C_s>0$ such that
\be\label{final estimates}
\| \eta_{\mu,1}\|_{s,b_*} \le C_s \mu^3,\quad
\| \eta_{\mu,2}\|_{s,b_*} \le C_s \mu^2 \mand
|a_\mu| \le C_s \mu^{s}.
\ee
The estimate for $|a_\mu|$ implies that it is small beyond all order of $\mu$.
\end{theorem}
Given this, we get our main result, Theorem \ref{nontech main}, by putting
$$
\left(
\begin{array}{c}
\Upsilon_{c,\mu,1}\\
\Upsilon_{c,\mu,2}
\end{array}
\right):=T_\mu \left( \mu \xib_\mu + \etab_\mu\right) \mand
\left(\begin{array}{c}
\Phi_{c,\mu,1}\\
\Phi_{c,\mu,2}
\end{array}
\right):=T_\mu\left(a_\mu \pb_\mu^{a_\mu}\right).
$$
Since $T_\mu$ (defined in \eqref{final cov}) is nearly the identity (see \eqref{LT is bounded}), the estimates in Theorem \ref{nontech main}
follow immediately. So let us proceed.

\subsection{The proof of Theorem \ref{main result}}
Fix $\mu \in (0,\mu_*) \cap M_c$.
Let $\etab_0=0$ and $a_0=0$. For $n \ge 0$ put
\be\label{iterate}
(\etab_{n+1},a_{n+1}) = N^\mu(\etab_n,a_n).
\ee
Since $(0,0) \in U^1_{\mu,\r_*}$, \eqref{map to self} implies that 
$$
\left\{ (\etab_n,a_n) \right\}_{n=0}^\infty \subset U^1_{\mu,\r_*}.
$$
$U^1_{\mu,\r_*}$ is  bounded in $X_1$, which is a Hilbert space. Therefore there exists a weakly convergent subsequence of $\left\{ (\etab_n,a_n) \right\}_{n=0}^\infty$.
Call the limit $(\etab_\mu,a_\mu) \in X_1$. Since the norm is lower-semicontinuous for weak limits, we have $(\etab_\mu,a_\mu) \in U^1_{\mu,\r_*}$.

Next,
\eqref{L2 contract} implies that $\| (\etab_{n+1},a_{n+1}) - (\etab_{n},a_n)\|_{X_0} \le 4^{-n}\| (\etab_{1},a_{1})\|_{X_0}$ which in turn implies that the sequence is Cauchy in $X_0$
and thus convergent. Since $X_1 \subset X_0$, the sequence's limit in $X_0$ must agree with the weak limit
$(\etab_\mu,a_\mu) \in X_1$. 

Moreover,  \eqref{L2 contract} implies
$$
\|N^\mu(\etab_\mu,a_\mu) - N^\mu(\etab_n,a_n)\|_{X_0} \le 4^{-1}\|(\etab_\mu,a_\mu) - (\etab_n,a_n)\|_{X_0}.
$$
Since the right hand side converges to zero as $n \to \infty$, we see that 
the iteration \eqref{iterate} implies
$$
(\etab_\mu,a_\mu) = N^\mu(\etab_\mu,a_\mu)
$$
where by $``="$ we mean equality in $X_0$. But since the left hand side is in $X_1$, the equality is in fact equality in this space. Thus we have our solution.
Uniqueness within $U^1_{\mu,\r_*}$ follows from \eqref{L2 contract}.
The smoothing properties of $N^\mu$ in \eqref{boot} immediately imply that $\etab_\mu$, since it is (part of) a fixed point, is smooth. The estimates implicit in \eqref{really small} give the estimates in \eqref{final estimates}.

\appendix
\section{Proof of Proposition \ref{A props}---properties of $\H$}\label{A appendix}
First we factor $\H$ as
$$
\H \xi = (c^2 \partial_x^2 - 2 \delta^2) \bunderbrace{\left[\xi - {4 \delta^2 \over c^2\partial_x^2 - 2 \delta^2} (\sigma_c \xi)\right]}{G_c \xi}
$$
where we interpret $\ds 4 \delta^2/( c^2\partial_x^2 - 2 \delta^2)$ as a Fourier multiplier operator, as in \cite{friesecke-pego1}.

A rescaled version of the operator $G_c$ appears\footnote{Their operator is called
``$(I-P^{(\ep)} DN^{(\ep)}(\phi))$" and appears at their equation $(4.6)$.} in \cite{friesecke-pego1}.
Specifically, their operator is equivalent to ours after the following rescaling:
\be\label{fp rescale}
c^2 = c_0^2 + \ep^2, \quad g(X) = \xi(X/\ep) \mand \varsigma_\ep(X) = \ep^{-2} \sigma_{c}(X/\ep).
\ee
Here $\ep > 0$ is small. In this case we find that $G_c \xi(x) = P_\ep g(X)$ where
$$
 P_\ep g := g - {4 \ep^2 \delta^2[\ep] \over (c_0^2+\ep^2)\ep^2 \partial_X^2 - 2 \delta^2[\ep]} 
\left( \varsigma_\ep g
\right).
$$
In the above $\ds \delta[\ep] := {1 \over 2}\left(S^\ep - S^{-\ep}\right).$

In \cite{friesecke-pego1} they show that
\be\label{fp est 2}
\left\|{4 \ep^2 \delta^2[\ep] \over (c_0^2+\ep^2)\ep^2 \partial_X^2 - 2 \delta^2[\ep]} 
- {\alpha_1 \over 1-\beta_1 \partial_X^2}\right \|_{H^s \to H^s} \le C \ep^2.
\ee
The  constants $\alpha_1\ne 0 $ and $\beta_1>0$ can be determined exactly, but are not important for our purposes here.
This is shown on the Fourier side by a careful analysis of the associated symbols and their poles.
A very similar operator is studied in \cite{FW} (specifically in Lemma A.12) and therein 
the authors show how to take the ideas of \cite{friesecke-pego1}
and extend estimates like \eqref{fp est 2} from $H^s$ to $H^s_q$ where $q>0$. Again, it all takes place
at the level of symbols.
In this way, the pole analysis on the symbols carried out in \cite{friesecke-pego1} is sufficient to show that
there exists $q > 0$ such that 
\be\label{fp est}
\left\|{4 \ep^2 \delta^2[\ep] \over (c_0^2+\ep^2)\ep^2 \partial_X^2 - 2 \delta^2[\ep]} 
- {\alpha_1 \over 1-\beta_1 \partial_X^2}\right \|_{H^s_{q} \to H^s_{q}} \le C \ep^2.
\ee

Likewise in \cite{friesecke-pego1} they show that
$$
\|\varsigma_\ep(\cdot)- \alpha_2 \sech^2(\beta_2 \cdot)\|_{H^1}  \le C \ep^2.
$$
The nonzero constants $\alpha_2$ and $\beta_2$ can be determined exactly, but we do not need them now.
One can extract from \cite{friesecke-pego1} that  this result can be extended to higher regularity and weighted spaces. Specifically we have
$$
\|\varsigma_\ep(\cdot)- \alpha_2 \sech^2(\beta_2 \cdot)\|_{H^s_q} \le C \ep^2.
$$

All of this together implies that
$$
\left\| P_\ep - \left[1 - {\alpha_1 \over 1 - \beta_1 \partial_X^2} \left( \alpha_2 \sech^2(\beta_2 \cdot) \cdot \right)\right]\right\|_{H^s_{q} \to H^s_{q}} \le C \ep^2.
$$
The operator in the square brackets is also analyzed in \cite{friesecke-pego1} and is shown to be a homeomorphism of $E^1$
and itself. Again, this can be extended to $E^s_q$. And so a Neumann series argument allows us to conclude that $P_\ep$ is also a homeomorphism of $E^s_q$ and itself. Unraveling the scalings shows that $G_c$
is a homeomorphism of $E^s_{b_c}$ and itself
and
$$
\| G_c\|_{E^{s}_b \to E^{s}_b} \le C.
$$

The map $c^2 \partial_x^2 - 2 \delta^2$ maps $E^{s+2}_{b}$ to $E^{s}_{b,0}$.  Morepver it is invertible when $b>0$. 
Here are the main ideas.  If $u \in E^{s}_{b,0}$ then  its Fourier transform $\Fo[u](k)$ can be analytically extended from $k \in \R$ to the horizonal strip $\left\{|\Im k| < b\right\} \subset \C$.  Since $u$ is a
mean-zero function we have $\Fo[u](0)$. And since since the Fourier transform of an even function is again an even function, we have $\partial_k \Fo[u] (0)= 0$. Thus $\Fo[u](k)$ has a zero of order at least two at $k = 0$.
Viewed as a Fourier multiplier operator, $c^2 \partial_x^2 - 2 \delta^2$ has symbol $\tilde{w}(k):=-c^2k^2 + 2\sin^2(k)$.
When $c>\sqrt{2}=c_0$ one can show (and in fact this is shown in \cite{friesecke-pego1}) that the only zero of $\tilde{w}(k)$ 
in the horizonal strip $\left\{|\Im k| < b_c\right\} \subset \C$ occurs at $k = 0$ and is of order two. Thus $\Fo[u](k)/\tilde{w}(k)$ has a removable singularity at $k = 0$ and is analytic in $\left\{|\Im k| < b_c\right\}$. 
Which implies, by way of the Paley-Wiener theorem (as described in Lemma 3 in \cite{beale1}) that we have $(c^2 \partial_x^2 - 2\delta^2)^{-1} u \in E^{s}_b$ for $b  \in (0,b_c]$ and moreover
$$
\|(c^2 \partial_x^2 - 2\delta^2)^{-1}\|_{E^{s}_{b,0} \to E^{s+2}_b} \le C_b.
$$

Thus the factorization of $\H$ implies that it is the product of two invertible operators. The usual algebra completes the proof.

\section{Proof of Theorem \ref{periodic solutions}---periodic solutions exist}\label{per appendix}
This proof proceeds along the lines of the proof of the Crandall-Rabinowitz-Zeidler ``bifurcation from a simple eigenvalue" theorem.
The main difficulty is carrying out the requisite estimates in a way that is uniform in the mass ratio $\mu$.

The first step is to put
$$
\hb(x) := I^\mu \varphib(\omega x)
$$
where $I^\mu:=\diag(\mu,1)$ and $\varphib(X)$
is $2 \pi$-periodic. 
The frequency $\omega$ is left undetermined; it is going to become one of our unknowns.

Putting this into \eqref{TWE5} gives the following equation for $\varphib$:
\be\label{PTWE}
{c^2 \omega^2 \mu \partial_X^2 \varphib +L_\mu[\omega] I^\mu \varphib}+ {L_\mu[\omega] Q_\mu[\omega](I^\mu \varphib,I^\mu \varphib)}= 0 .
\ee
The operators $L_\mu[\omega]$ and $Q_\mu[\omega]$ are formed from $L_\mu$ and $Q_\mu$ by replacing $A$ and $\delta$ in their definitions by
$$
A[\omega]:={1 \over 2} \left(S^\omega + S^{-\omega} \right) \mand 
\delta[\omega]:={1 \over 2} \left(S^\omega - S^{-\omega} \right)
$$
respectively.

We will look for solutions $\varphib$ in 
$$
\Xc^s:=E^s_{\per,0} \times O^s_\per
$$
The spaces $\Xc^s$ are Hilbert spaces with the usual inner product and norm, which we abbreviate as
$\langle \ub, {\bf v}\rangle_{s} := \langle\ub,{\bf{v}} \rangle_{\Xc^s}$ and
$\| \ub \|_s := \|\ub\|_{\Xc^s}$. We abuse notation and use $\| f\|_s := \| f\|_{H^s_\per}$ as well.

The following lemma contain all the information about $L_\mu[\omega]$ we need for the proof.
\begin{lemma}\label{linear stuff} For all $|c|>c_0$ there exists $\mu_\per >0$ such the following hold for all $\mu \in (0,\mu_\per)$.
\begin{itemize}
\item 
For all $\omega \in \R$ and $s \ge 0$, $L_\mu[\omega]$ is a bounded operator from 
$
\Xc^s 
$
to
itself. 

\item 
For any $\omega_1,\omega_2 \in \R$ we have
\be\label{lip1}
\| [(L_\mu [\omega_1] -L_\mu[\omega_2]) \ub] \cdot \ib \|_{s} \le C|\omega_1-\omega_2|\|\ub\|_{s+1}
\ee
and
\be\label{lip2}
\| [(L_\mu [\omega_1] -L_\mu[\omega_2]) \ub]   \cdot \jb \|_{s} \le C\mu|\omega_1 - \omega_2|\|\ub\|_{s+1}.
\ee
\item There exists $\omega_\mu$ 
and
$|\upsilon_\mu| \le 1$ for which we have
$$
\bunderbrace{ c^2 \omega_\mu^2 \mu \partial_X^2 \varphib + L_\mu[\omega_\mu] I^\mu \varphib}{\Gamma_\mu \varphib} = 0 
\mand \varphib \in \Xc^s$$
if and only if $\varphib(X) = a \nub_\mu(X)$ where
$$
\nub_\mu(X):= \left[\begin{array}{c} \upsilon_\mu \cos(X) \\ \sin(X) \end{array}\right]
$$
and $a \in \R$ is arbitrary.
(That is to say $\ker\Gamma_\mu  = \spn \nub_\mu$.)
\item There exists $|z_\mu| \le 1$ such that
\be\label{eqn}
\Gamma_\mu \varphib = \gb
\ee
has a solution $\varphib \in \Xc^{s+2}$ for a given $\gb \in \Xc^s$
if and only if
\be\label{solve it}
\langle \gb ,\nub^*\rangle_{0} = 0
\ee
where
$$
\nub^*_\mu(X):= \left[\begin{array}{c} z_\mu \cos(X) \\ \sin(X) \end{array}\right].
$$
Also $\Gamma_\mu^\dagger \nub_\mu^* = 0$.
(That is to say $\range \Gamma_\mu = [\ker \Gamma_\mu^\dagger]^\perp = [\spn \nub^*_\mu]^\perp$.)

\item Moreover, if \eqref{eqn} holds and we demand that $\varphib$ also meets \eqref{solve it}
then
\be\label{per coerce}
\|\varphib\|_{s+2} \le C\|\gb\|_{s}
\ee
where the constant $C>0$ does not depend on $\mu$, $\varphib$, $\gb$.

\end{itemize}
\end{lemma}

\begin{proof} The proof of this lemma is not much more than a lengthy exercise in Fourier series. 
Specifically, for $u \in L^2_\per$ put
$$
 \hat{u}(k) := {1 \over 2 \pi} \int_{ -\pi}^\pi e^{- i kX} u(X) dX.
$$
In which case the Fourier inversion formula reads
$$
u(X) = \sum_{k \in Z} \hat{u}(k) e^{ik X}.
$$
Note that since we are working in the space $E^s_{\per,0} \times O^s_\per$, the $k =0$ modes
are identically zero for all the functions we consider.

From the discussion discussion in Section \ref{periodic solutions},
we have
$$
\hat{L_\mu[\omega] \varphib }(k) = \tilde{L}_\mu (\omega k) \hat{\varphib}(k)
$$
where $\tilde{L}_\mu(\omega)$ is given by \eqref{Lmu symbol}.
The matrix $\tilde{L}_\mu(\omega)$ is uniformly bounded in $\omega$ and its first row  vanishes at $\omega=0$. 
Membership in $E^s_{\per,0}$ requires that
$\hat{\ub}(0)\cdot \ib =0$. In this way we see that $L_\mu$ is bounded from $\Xc^s$ to itself.

The following estimate for the shift map is well-known, and its proof can be found in (for instance) \cite{FW}:
$$
\| (S^{\omega_1} - S^{\omega_2}) u\|_{s} \le  |\omega_1 - \omega_2|\| u\|_{s+1}.
$$
This estimate implies similar estimates for $A[\omega]$  and $\delta[\omega]$.  
With these, proving \eqref{lip1} and \eqref{lip2} amounts to just keeping track of $\mu$ 
in the definition of $L_\mu[\omega]$.

Now we need to identify the bifurcation frequency
so that  $\Gamma_\mu$ has the properties ascribed to it.
This is more or less exactly the linear calculation carried out in Section \ref{periodic solutions}.
Since $\hat{\partial_X u}(k) = i k \hat{u}(k)$ we see that $\Gamma_\mu \nub = 0$ if
and only if
$$
\tilde{L}_\mu(\omega_\mu k)I^\mu \hat{\nub}(k)= c^2\mu \omega_\mu^2 k^2 {\hat{\nub}}(k).
$$
So we see that we need to take $\omega_\mu$ and $k$ so that $c^2 \mu \omega^2_\mu k^2$ is an eigenvalue of $\tilde{L}(\omega_\mu k)I^\mu$. 
For simplicity we take $ k = \pm1$.

The eigenvalues of $\tilde{L}_\mu(\om)I^\mu$ are
$$
\lambda_\mu^\pm(\om) := 1+ \mu \pm \sqrt{ (1+\mu)^2 - 4\mu \sin^2(\om)}.
$$
We need to understand when we can solve
\be\label{key eqn}
\lambda_\mu^-(\om) = c^2 \mu \om^2\quad \text{or} \quad \lambda_\mu^+(\om) = c^2\mu \om^2
\ee
for $\om$.

%
%

We claim that so long $c^2 >2$ then 
$$
\ds
\lambda^-_\mu(\om)= c^2 \mu \om^2 \iff \om=0.
$$
Towards this end, computation of the Maclaurin series of $\lambda^-_\mu(\om)$  gives
$$
\lambda^-_\mu(\om)= {2\mu \over 1+\mu} \om^2 + \O(\om^4).
$$
For $\mu \ge0$ and $c^2 > 2$ we have $c^2 > 2/(1+\mu)$ which 
implies that only $\om=0$ solves \eqref{key eqn} in a neighborhood of the origin. 
For $\om$ away from the origin,
we note that $\lambda_\mu^-(\om)$ is a uniformly bounded function of $\om$ whereas $c^2 \mu \om^2$ diverges as $\om \to \infty$. We omit the particulars,
but this implies there are no solutions far from the origin.
Since we are working in $\Xc^s= E^s_{\per, 0} \times O^s_\per$, there are no $k = 0$  modes in the Fourier expansion of $\ub$. Thus we cannot use $\lambda^-_\mu$ 
as the eigenvalue in \eqref{key eqn}. 

So we will use $\lambda_\mu^+$. 
The second equation in \eqref{key eqn} is precisely \eqref{crit freq} from
Section \ref{periodic solutions}.
Lemma \ref{omega lemma} in that Section
tells us what we need to know about its solutions.
Specifically there exists a unique nonnegative solution $\omega_\mu$ and that this solution is $\O(\mu^{-1/2})$.
%

Thus we now consider the 
the eigenvector of $\tilde{L}_\mu(\omega_\mu)I^\mu$ associated with $\lambda_\mu^+(\omega_\mu)$. It is 
$\bf{v}_\mu:=
\left[ \begin{array}{c}
i{\upsilon}_\mu \\
1
\end{array}
\right]
$
where 
\be\label{this is upsilon}
{\upsilon}_\mu : = -{2\mu    \cos(\omega_\mu) \sin(\omega_\mu) (1 - 2 \cos^2(\omega_\mu) - \mu \cos^2(\omega_\mu) \sin^2(\omega_\mu)) \over \lambda^+_\mu(\omega_\mu) - 2\mu   \sin^2(\omega_\mu)(1 - \mu \cos^2(\omega_\mu) )}.
\ee
Examination of the formula for $\lambda_\mu^+(\om)$ shows that it lies in 
$2 \le \lambda_\mu^+(\om)  \le 2+\mu$. (See Figure~\ref{critfreq}).
This tells us that $|{\upsilon}_\mu| \le C\mu$ for $\mu$ sufficiently close to zero and, in particular, can be made less than one.

All of the above facts are on the ``frequency side" but we can reassemble things on the ``space side" to find that 
if we put 
$$
\nub_\mu (X):= \Im \left( {\bf v}_\mu e^{i X}\right) = \left[ \begin{array}{c} \upsilon_\mu \cos(X) \\ \sin(X) \end{array}\right]
$$
then $\nub_\mu \in \Xc^s$ and $\Gamma_\mu \nub_\mu = 0$. The uniqueness of $\omega_\mu$ implies that $\nub_\mu$ is the only such solution in $\Xc^s$, up to scalar multiples. A  parallel line of reasoning, together with the usual Fredholm theory, shows that 
there exists $\nub_\mu^*$, with the stated estimates, 
such that $\range \Gamma_\mu = [\ker \Gamma_\mu^\dagger]^\perp = [\spn \nub^*_\mu]^\perp$. 

Now suppose that $\Gamma_\mu \ub = \gb$. So, for all $k \in \Z / \left\{ 0 \right\}$ we have
$$
\left(-c^2 \omega_\mu^2 \mu k^2 + \tilde{L}_\mu(\omega_\mu k)I^\mu \right) \hat{\ub}(k) = \hat{\gb}(k).
$$
Or rather
$$
\left(1 - {1 \over c^2 \omega_\mu^2 \mu k^2}\tilde{L}_\mu(\omega_\mu k)I^\mu \right) \hat{\ub}(k) =  -{1 \over c^2 \omega_\mu^2 \mu k^2}\hat{\gb}(k).
$$

For $\mu$ sufficiently close to zero, inspection of the entries of $\tilde{L}_\mu(\om)$ shows that $$\| \tilde{L}_\mu(\om)I^\mu\| \le 3.$$ To be clear, we mean the
standard matrix norm here.
And thus the estimate \eqref{omega est} implies $$
\left\| {1 \over c^2 \omega_\mu^2 \mu k^2}\tilde{L}_\mu(\omega_\mu k)I^\mu\right\| \le {3 \over 2 k^2}.
$$
This quantity is less than one for all $|k| \ge 2$, uniformly in $\mu$. Thus $\left(1 -  {1 \over c^2 \omega_\mu^2 \mu k^2}\tilde{L}_\mu(\omega_\mu k)I^\mu\right)^{-1}$ exists via the Neumann series and has matrix norm uniformly bounded in both $k$ and $\mu$. In this way we see that
$$
|k|\ge2 \implies | \hat{\ub}(k) |\le Ck^{-2} |\hat{\gb}(k)|.
$$
Lastly, for $k = \pm1$, it is a tedious but mundane exercise in two-by-two matrices to see that $[-c^2 \omega_\mu^2\mu + L_\mu(\omega_\mu)I^\mu] \hat{\ub(\pm1)} = \hat{\gb}(\pm1)$ plus the condition \eqref{solve it} implies $|\hat{\ub}(\pm 1)| \le C|\hat{\gb}(\pm1)|$. Thus we use Plancheral's theorem and we see that
$$
\| \ub\|_{s+2} \le C\|\gb\|_{s}.
$$
This completes the proof of Lemma \ref{linear stuff}.

%

\end{proof}

%


Now that we understand $L_\mu$, we put
$$
\omega = \omega_\mu + \xi \mand \varphib = a \nub_\mu + a \psib
$$
where 
$
\psib 
$
meets condition \eqref{solve it}. Plugging this into \eqref{PTWE} returns, after some algebra,
\be\label{hwg}
\Gamma_\mu \psib=R_0(\xi) + R_1(\xi)+R_2(\xi,\psib)+a R_3(\xi,\psib)
\ee
where
\bes\begin{split}
R_0(\xi)&:= - 2 c^2 \mu \omega_\mu \xi \partial_X^2 \nub_\mu\\
R_1(\xi)&:=- c^2 \mu \xi^2 \partial_X^2 \nub_\mu -\left( L_\mu[\omega_\mu+\xi] -L_\mu[\omega_\mu]\right)I^\mu \nub_\mu\\
R_2(\xi,\psib)&:=-{ c^2 \mu (2\omega_\mu\xi +\xi^2)\partial_X^2\psib -\left(L_\mu[\omega_\mu+\xi] -L_\mu[\omega_\mu]\right)I^\mu \psib}\\
R_3(\xi,\psib)&:=-{L_\mu[\omega_\mu+\xi] Q_\mu[\omega_\mu+\xi](I^\mu (\nub_\mu+\psib),I^\mu (\nub_\mu+\psib))} .
\end{split}\ees

We perform a Liapunov-Schmidt decomposition \cite{zeidler2}. Let
$$
\pi_\mu \ub := {\langle\ub,\nub_\mu^*  \rangle_0 \over \langle\nub_\mu,\nub_\mu^*  \rangle_0 }\mand 
\Pi_\mu \ub:=[\pi_\mu \ub] \nub_\mu. 
$$
The following properties are easily checked:
$$
\pi_\mu \Pi_\mu = \pi_\mu,\quad
\Pi_\mu^2 = \Pi_\mu, \quad 
\Pi_\mu \nub_\mu = \nub_\mu \mand \Pi_\mu\Gamma_\mu = \Gamma_\mu\Pi_\mu = 0.
$$
Note the the solvability condition \eqref{solve it} is equivalent to saying $\pi_\mu \gb = 0$. Thus we can always solve equations of the form
$\Gamma_\mu \varphib = (1-\Pi_\mu) \gb$. Such solutions are not unique since $\Gamma_\mu$ has a non-trivial kernel. By demanding that the solution has $\pi_\mu \varphib = 0$, the solution becomes unique; 
we denote it by $\varphib =: \Gamma_\mu^{-1} (1- \Pi_\mu) \gb$.

Applying $\pi_\mu$ to \eqref{hwg} gives, after rearranging terms:
\be\label{xi eqn}
\xi = \bunderbrace{{1 \over 2 c^2 \mu \omega_\mu \langle\nub_\mu,\partial_X^2\nub_\mu^* \rangle_0} \pi_\mu \left(R_1(\xi) + R_2(\xi,\psib) + aR_3(\xi,\psib)  \right)}{\Xi_\mu(\xi,\psib,a)}.
\ee
Applying $1-\Pi_\mu$ to \eqref{hwg} gives, after rearranging terms:
\be\label{psi eqn}
\psib = \bunderbrace{\Gamma_\mu^{-1}(1 - \Pi_\mu) \left(R_0(\xi)+R_1(\xi) + R_2(\xi,\psib) + aR_3(\xi,\psib)  \right)}{\Psi_\mu(\xi,\psib,a)}.
\ee

Now we begin to estimate the terms.
First we have using the estimates on $\nub_\mu$, $\nub_\mu^*$ and the Cauchy-Schwartz inequality
$$
\left \vert \pi_\mu \ub \right \vert \le C \mu \| \ub\cdot \ib\|_{0} + C \| \ub\cdot \jb\|_{0}.
$$
In turn this gives, for any $s \ge 0$:
$$
\| \Pi_\mu \ub \|_{s} \le  C \mu \| \ub\cdot \ib\|_{0} + C \| \ub\cdot \jb\|_{0}
\mand
\| (1-\Pi_\mu) \ub \|_{s} \le  C \| \ub\|_{s}.
$$
The latter of these, in tandem with the estimate \eqref{per coerce} show that
$$
\| \Gamma_\mu^{-1} (1- \Pi_\mu) \ub \|_{s+2} \le C \| \ub \|_s
$$
Also, the estimates for $\omega_\mu$, $\nub_\mu$ and $\nub_\mu^*$ imply that
\be\label{bifurcate}
\left \vert {1 \over 2 c^2 \mu \omega_\mu \langle\nub_\mu,\partial_X^2\nub_\mu^* \rangle_0}  \right \vert \le C\mu^{-1/2}.
\ee

We turn our attention to estimating the $R_0, R_1$ and $R_2$. The following estimates
hold when $s \ge 0$ and when $|\xi|,|\gxi|,\|\psib\|_s,\| \gpsib\|_s \le 1$. 
Each of these estimates follows in routine way from the estimates in Lemma \ref{linear stuff}.
First, $R_0(0) = 0$ and
$$
\| R_0(\xi) - R_0(\gxi)\|_{s} \le C\mu^{1/2}| \xi - \gxi|.
$$
Next, $R_1(0) = 0$ and
$$
\| R_1(\xi) - R_1(\gxi)\|_{s} \le C\mu| \xi - \gxi|.
$$
Then we have $R_2(0,0) = 0$ and
$$
\| R_2(\xi,\psib) - R_2(\gxi,\gpsib)\|_{s}\le C\mu^{1/2}\left(
\| \psib\|_{s+2}
+\| \gpsib\|_{s+2}\right) |\xi - \gxi|
+
C\mu^{1/2}\left(
|\xi|+|\gxi|
\right)
\| \psib-\gpsib\|_{s+2}.
$$

Finally we look at $R_3$. In some sense this is the most complicated since it has the most terms. On the other hand,
the justification of the estimates  $\eqref{E5}-\eqref{E6}$ for $L_\mu Q_\mu$ in Section \ref{equation of motion} contains the fundamental ideas and 
there are no subtleties. In the interest of brevity, we merely
report the outcome of the estimates. The estimates are
\bes\begin{split}
\| R_3(0,0)\cdot \ib \|_{s} \le &C\\
\| R_3(0,0)\cdot \jb \|_{s} \le &C\mu\\
\| (R_3(\xi,\psib) - R_3(\gxi,\gpsib))\cdot \ib\|_{s} \le &
C\left(1+\|\psib\|_{s+2} + \|\gpsib\|_{s+2} \right)|\xi - \gxi|\\
+ &C \left(1+|\xi|+|\gxi| \right) \|\psib-\gpsib\|_{s+2}\\
\| (R_3(\xi,\psib) - R_3(\gxi,\gpsib))\cdot \jb\|_{s} \le 
&C\mu\left(1+\|\psib\|_{s+2} + \|\gpsib\|_{s+2} \right)|\xi - \gxi|\\
+ &C \mu\left(1+|\xi|+|\gxi| \right) \|\psib-\gpsib\|_{s+2}.
\end{split}\ees

The preceding estimates for $R_0,\ R_1,\ R_2$ and $ R_3$ and \eqref{bifurcate} 
imply
$$
\left \vert \Xi_\mu(0,0,a) \right \vert +
\left \| \Psi_\mu(0,0,a) \right\|_{2}
\le C\mu^{1/2}|a|
$$
and
\begin{multline*}
\left \vert \Xi(\xi,\psib,a) -  \Xi(\gxi,\gpsib,a)\right \vert +
\| \Psi_\mu(\xi,\psib,a) -  \Psi_\mu(\gxi,\gpsib,a)\|_2 \\
\le 
C \left( \mu^{1/2}+|\xi|+|\gxi|+\| \psib\|_{2}
+\| \gpsib\|_{2} + |a|\right)\left( |\xi - \gxi|+
\| \psib-\gpsib\|_{2} \right).
\end{multline*}
Thus there exists $\rho_*>0$, $a_\per>0$ and $\mu_\per>0$ such that for all $\mu \in (0,\mu_\per)$
and $|a|\le a_\per$ 
the mapping $(\xi,\psib) \to (\Xi_\mu(\xi,\psib,a),\Psi_\mu(\xi,\psib,a))$ take the ball of radius $\rho_*$ in $\R \times \Xc^2$
to itself and is a contraction on that set with contraction constant less that $1/2$. And so, for each $\mu \in (0,\mu_\per)$ and $|a|\le a_\per$ there is a unique fixed point $(\xi_\mu^a,\psib_\mu^a)$ of the map $(\Xi_\mu,\Psi_\mu)$.
When $a = 0$ this fixed point is trivial. 

Next we notice that, for $|a|,|\ga| \le a_\per$ we have
\bes\begin{split}
&\|(\xi_\mu^a,\psib^a_\mu) - 
(\xi_\mu^\ga,\psib^\ga_\mu) \|_{\R \times \Xc^2}\\
\le 
&\|(\Xi_\mu(\xi_\mu^a,\psib_\mu^a,a),\Psi_\mu(\xi_\mu^a,\psib_\mu^a,a))-
(\Xi_\mu(\xi_\mu^a,\psib_\mu^a,\ga),\Psi_\mu(\xi_\mu^a,\psib_\mu^a,\ga))\|_{\R \times \Xc^2}\\
+&
\|(\Xi_\mu(\xi_\mu^a,\psib_\mu^a,\ga),\Psi_\mu(\xi_\mu^a,\psib_\mu^a,\ga))-
(\Xi_\mu(\xi_\mu^\ga,\psib_\mu^\ga,\ga),\Psi_\mu(\xi_\mu^\ga,\psib_\mu^\ga,\ga))\|_{\R \times \Xc^2}.
\end{split}\ees
For the second line in the above we use the contraction estimate and get
\bes\begin{split}
&\|(\xi_\mu^a,\psib^a_\mu) - 
(\xi_\mu^\ga,\psib^\ga_\mu) \|_{\R \times \Xc^2}\\
\le 
&2\|(\Xi_\mu(\xi_\mu^a,\psib_\mu^a,a),\Psi_\mu(\xi_\mu^a,\psib_\mu^a,a))-
(\Xi_\mu(\xi_\mu^a,\psib_\mu^a,\ga),\Psi_\mu(\xi_\mu^a,\psib_\mu^a,\ga))\|_{\R \times \Xc^2}.
\end{split}\ees
Tracing through the definitions shows that $\Xi_\mu$ and $\Psi_\mu$ depend on $a$ only through the term $R_3$
and do so linearly. Thus we can use the estimates for $R_3$ from above to get
$$
\|(\xi_\mu^a,\psib^a_\mu) - 
(\xi_\mu^\ga,\psib^\ga_\mu) \|_{\R \times \Xc^2}\\\le 
C|a-\ga|.
$$
We can  bootstrap this last estimate to
$$\|(\xi_\mu^a,\psib^a_\mu) - 
(\xi_\mu^\ga,\psib^\ga_\mu) \|_{\R \times \Xc^s}\le 
C|a-\ga|
$$
for any $s \ge 0$.  

These complete the proof of Theorem \ref{periodic solutions exist}. Summing up, if we put
$$
\omega_\mu^a:=\omega_\mu + \xi_\mu^a\mand
\pb_\mu^a(x) := \bunderbrace{I^\mu \nub_\mu(\omega_\mu^a x) + I^\mu \psib_\mu^a(\omega_\mu^a x)}{\left( \begin{array}{c} \mu \varphi_{\mu,1}^a(x) \\ \varphi_{\mu,2}^a(x)\end{array} \right)}
$$
then $a \pb_\mu^a(x)$ solves \eqref{TWE5} for all $|a|\le a_\per$ and $\mu \in (0,\mu_\per)$.

\subsection{Some final estimates}

The estimates in Theorem \ref{periodic solutions exist} are stated in terms 
of $\omega_\mu^a$, $\nub_\mu(X)$ and $\psib_\mu^a(X)$, but we will be needing estimates on $\pb_\mu^a(x)$
and so we close out 
 this appendix with some additional estimates for that function.
That is to say estimates for the periodic solutions in the ``original coordinates."
We use the following lemma:
\begin{lemma} Suppose that $P(X) \in C^{\infty}_\per$. For $|a|\le a_\per$ let $P^a(x):= P(\omega_\mu^a x)$.
Then, there exists $C>0$ such that for all $|a|\le a_\per$ and $\mu \in (0,\mu_\per)$ we have
$$
\|  P^a
 \|_{W^{s,\infty}} \le C\|P\|_{C^{s+1}_\per} \mu^{-s/2}.
$$
Also, for all $b < 0$ there exists $C_b>0$ (which diverges as $b \to 0^-$) such that
$$
\| P^a - P^\ga\|_{W^{s,\infty}_b} \le C_b \|P\|_{C^{s+1}_\per}\mu^{-s/2}|a-\ga|.
$$
\end{lemma}
\begin{proof} The first estimate follows from the chain rule and estimates for $\omega_\mu^a$ in Theorem \ref{periodic solutions}. We do not provide the details.
As for the second estimate, the key quantity to estimate on the left hand side is
$$
|(\omega_\mu^a)^s P^{(s)}(\omega_\mu^a x) -  (\omega_\mu^\ga)^s P^{(s)}(\omega_\mu^\ga x) |.
$$
The triangle inequality give
\begin{multline*}
|(\omega_\mu^a)^s P^{(s)}(\omega_\mu^a x) -  (\omega_\mu^\ga)^s P^{(s)}(\omega_\mu^\ga x) |
\\\le |(\omega_\mu^a)^s - (\omega_\mu^\ga)^s||P^{(s)}(\omega_\mu^a x)|
+ |\omega_\mu^\ga|^s |P^{(s)}(\omega_\mu^a x) - P^{(s)}(\omega_\mu^\ga x)|.
\end{multline*}
Factoring and the estimates for $\omega_\mu^a$ in Theorem \ref{periodic solutions} give:
$$
 |(\omega_\mu^a)^s - (\omega_\mu^\ga)^s| \le C\mu^{-(s-1)/2}|\omega_\mu^a -\omega_\mu^\ga| \le C\mu^{-(s-1)/2}|a-\ga|.
$$
Also the Fundametnal Theorem of Calculus (FTOC) tells us
$$
 |P^{(s)}(\omega_\mu^a x) - P^{(s)}(\omega_\mu^\ga x)| \le \| P \|_{C^{s+1}_\per} |\omega_\mu^a-\omega_\mu^\ga||x|.
$$
And so
$$
|(\omega_\mu^a)^s P^{(s)}(\omega_\mu^a x) -  (\omega_\mu^\ga)^s P^{(s)}(\omega_\mu^\ga x) |
\le C\mu^{-s/2}|a-\ga|\| P\|_{C^{s+1}_\per}|x|.
$$
Since $\cosh^b(x)|x| \in L^\infty(\R)$ so long as $b<0$, we can use these to get the second estimate
in the lemma.
\end{proof}

Applying this lemma gives us
\be\label{varphi estimate}
\| \varphi^a_{\mu,1}\|_{W^{s,\infty}} +\| \varphi^a_{\mu,2}\|_{W^{s,\infty}} \le C \mu^{-s/2}.
\ee
And, for $b < 0$:
\be\label{varphi lipschitz}
\| \varphi_{\mu,1}^a - \varphi_{\mu,1}^\ga\|_{W^{s,\infty}_b}+\| \varphi_{\mu,2}^a - \varphi_{\mu,2}^\ga\|_{W^{s,\infty}_b} \le C_b\mu^{-s/2}|a-\ga|.
\ee

\section{Proof of Lemmas \ref{coerce lemma}-\ref{gamma props}---properties of $\L_\mu$.}\label{B appendix}

Let
$$
\S_\mu \eta := c^2 \mu \eta'' + c^2 \mu \omega_\mu^2 \eta+ 4 \sigma_c(x) \eta.
$$
This Schr\"odinger operator\footnote{It is almost the same as the Schr\"odinger operator $\U_\mu$ described in Section \ref{strategy} at \eqref{first pass 2} and it has the same sort of properties as $\U_\mu$.
The only difference is that the constant term ``$2$" in $\U_\mu$ has been replaced by ``$c^2 \mu \omega_\mu^2$"
in $\S_\mu$. The estimates in \eqref{omega est} tell us that these two constants are within $\O(\mu)$ of one another,
so this is a minor change. The reason we do this is so that the Jost solutions for $\S_\mu$
will be asymptotic as $|x| \to \infty$ to a sinusoid with  the critical frequency $\omega_\mu$ of the periodic solutions from Theorem \ref{periodic solutions exist}.}
will be the frame on which we build $\L_\mu$.
The estimates \eqref{LT is bounded}, 
\eqref{E1}-\eqref{E4}
and
\eqref{sigma estimate}
 can be used to show that
\be\label{almost S} 
\| (\S_\mu - \L_\mu) \eta\|_{s,b} \le C\mu \|\eta\|_{s,b}
\ee
and so we see that $\L_\mu$ is a small perturbation of $\S_\mu$. 
Nonetheless, observe that we are interested in the approximation when $\mu$ is small and 
as such the singular nature of $\S_\mu$ and $\L_\mu$ create a number of technical difficulties that are not resolvable using standard perturbation methods.

\subsection{The coercive estimates.}
First we prove an {\it a priori} estimate  for solutions of 
\be\label{nonhom}
\S_\mu f = c^2 \mu f'' + c^2 \mu \omega_\mu^2 f+ 4 \sigma_c(x) f = g
\ee
by means of an energy argument. To begin, we assume $f$ and $g$ are compactly supported, smooth and odd functions. 

Multiplying \eqref{nonhom} by
$\ds
{\sinh(2bx) f'(x)/ \left(
c^2 \mu \omega_\mu^2 + 4 \sigma_c(x) \right)}
$
and integrating on $\R$ returns
$$
\int_\R \left(
{c^2 \mu \sinh(2bx) \over 
c^2 \mu \omega_\mu^2 + 4 \sigma_c(x)} f''(x) f'(x) 
+ \sinh(2bx)f'(x)f(x)
\right)dx = \int_\R {\sinh(2bx) \over 
c^2 \mu \omega_\mu^2 + 4 \sigma_c(x)}f'(x)g(x) dx.
$$
Performing the usual integration by parts sudoku on the left converts this to
\begin{multline}\label{last eqn}
\int_\R 
{c^2 \mu  \cosh(2bx) \over 
c^2 \mu \omega_\mu^2 + 4 \sigma_c(x)} [f'(x)]^2
dx
 -
\int_\R {2c^2 \mu \sinh(2bx)  \sigma_c'(x) \over
b(c^2 \mu \omega_\mu^2 + 4 \sigma_c(x))^2} [f'(x)]^2 dx\\
+
\int_\R2b \cosh(2bx) [f(x)]^2 dx 
= -{1 \over b}\int_\R {\sinh(2bx) \over 
c^2 \mu \omega_\mu^2 + 4 \sigma_c(x)}f'(x)g(x) dx.
\end{multline}
We also know from Theorem \ref{sigma exists} that $\sigma_c'(x)$
is negative for $x > 0$ and is positive for $x < 0$. This implies that $ \sinh(2bx)\sigma'_c(x) \le 0$ for all $x$.
Which means that the middle term on the left of \eqref{last eqn} is clearly non-negative. Thus we can omit it from \eqref{last eqn} to get
\begin{multline}
\label{back one}\int_\R 
 \left({c^2 \mu \cosh(2bx) \over 
c^2 \mu \omega_\mu^2 + 4 \sigma_c(x)}\right) [f'(x)]^2
dx\\
+
\int_\R \cosh(2bx) [f(x)]^2 dx 
\le -{1\over b}\int_\R {\sinh(2bx) \over 
c^2 \mu \omega_\mu^2 + 4 \sigma_c(x)}f'(x)g(x) dx.
\end{multline}
Since $\omega_\mu = \O(\mu^{-1/2})$ (from estimate \eqref{omega est}), the positivity of $\sigma_c$ (from Theorem \ref{sigma exists}) implies that
\be\label{equiv}
C^{-1} \le {1\over c^2 \mu \omega_\mu^2 + 4 \sigma_c(x)} \le C
\ee
for some $C > 1$  and all $x \in \R$.
Also we have a constant $C>1$ for which
\be\label{oeq}
C^{-1} \cosh^{2b}(x) \le \cosh(2bx) \le C \cosh^{2b}(x) \mand |\sinh(2bx)| \le C \cosh^{2b}(x)
\ee
holds for all $x \in \R$ and $b \in [0,1]$.

Putting the above together with \eqref{back one} and using Cauchy-Schwarz gets us
$$
\mu \| f' \|^2_{0,b} + \| f\|^2_{0,b} \le C b^{-1} \| f'\|_{0,b} \|g\|_{0,b} \le {\mu \over 2} \| f'\|^2_{0,b} + {1 \over 2\mu}C^2b^{-2}\|g\|^2_{0,b}.
$$
The last inequality above is just ``Cauchy's inequality with parameter." 
Bringing the first term on the right over to the left and adjusting some constants gives us
$$
\mu \| f' \|^2_{0,b} + \| f\|^2_{0,b} \le C_b \mu^{-1} \|g\|^2_{0,b}.
$$
And so we have shown that
$
 \| f\|_{0,b} \le Cb^{-1} \mu^{-1/2} \|g\|_{0,b}$ and $\| f' \|_{0,b} \le C b^{-1} \mu^{-1} \|g \|_{0,b}$.
 And these in combination with the equation \eqref{nonhom} imply that 
 $\| f'' \|_{0,b} \le C b^{-1} \mu^{-3/2} \|g \|_{0,b}$.

In the above we assumed that $f$ and $g$ were smooth, odd and compactly supported, but a standard density argument implies
that the same result holds 
$f \in O^2_b$ and $g \in O^0_b$. Specifically we have, for $k = 0,1,2$
\be\label{coerce 0}
f \in O^2_{b},\quad g \in O^0_b \mand \S_\mu f = g \implies \| f \|_{k,b} \le C_b \mu^{(k+1)/2} \| g \|_{0,b}.
\ee

Here is an improvement of this estimate that we need.  It says that if 
we are content to meausure $f$ in a less regular space than we measure $g$ in, that we 
can claw back some powers of $\mu$ in the estimate.   On the right hand side of \eqref{back one} we integrate by parts to get
$$
\mu \| f' \|^2_{0,b} + \| f\|^2_{0,b} \le C b^{-1} \left \vert \int_\R f(x) {d \over dx} \left( {\sinh(bx) g(x) \over c^2 \mu \omega_\mu^2 + 4 \sigma_c(x)}\right) dx\right \vert.
$$
On the left we used \eqref{equiv} and \eqref{oeq} as above.
Cauchy-Schwarz on the right, plus \eqref{equiv} and \eqref{oeq}, give us
$$
\mu \| f' \|^2_{0,b} + \| f\|^2_{0,b} \le Cb^{-1} \| f \|_{0,b} \| g\|_{1,b} \le {1 \over 2} \|f \|_{0,b}^2 + {1 \over 2}{C^2 b^{-2}} \|g\|_{1,b}^2.
$$
We have again used Cauchy's inequality in the last step.
This implies
$$
\mu \| f' \|^2_{0,b} + \| f\|^2_{0,b} \le C_b \| g\|^2_{1,b}.
$$
And so we get
\be\label{coerce smooth}
f \in O^2_{b},\quad g \in O^1_b  \mand
\S_\mu f = g \implies
 \| f\|_{0,b} \le C_b \| g \|_{1,b}.
 \ee
Note that unlike \eqref{coerce 0} there are no negative powers of $\mu$ on the right hand side of this.

Next we differentiate \eqref{nonhom} $s-$times and we get
$$
\S_\mu f^{(s)} = g^{(s)} + \left( \S_\mu f^{(s)} - (\S_\mu f)^{(s)}\right).
$$
It is straightforward to show that $\|\left( \S_\mu f^{(s)} - (\S_\mu f)^{(s)}\right)\|_{0,b} \le C \| f\|_{s-1,b}$.
An induction argument takes this last estimate, \eqref{coerce 0} and \eqref{coerce smooth} and yields,
for $k = -1,0,1,2$ and $s+k \ge 0$:
\be\label{coerce gen}
f \in O^{s+2}_{b},\quad g \in O^s_b  \mand
\S_\mu f = g \implies
 \| f\|_{s+k,b} \le C_b \mu^{-(k+1)/2} \| g \|_{s,b}.
 \ee
 
If  $\S_\mu f  = g_1 + g_2$ then we can deploy different choices for
$s$ and $k$ for the different $g_j$. For instance, if $g_1 \in O^s_b$ and $g_2 \in O^{s+k}_b$ we would have
\be\label{coerce gen 2}
\| f \|_{s+k,b} \le C_b\left(\mu^{-(k+1)/2} \|g_1\|_{s,b}+ \mu^{-1/2} \|g_1\|_{s+k,b}\right).
\ee

We can now prove Lemma \ref{coerce lemma}.

\begin{proof} (Lemma \ref{coerce lemma})
 If 
$\L_\mu f = g$ then $\S_\mu f= g + (\S_\mu - \L_\mu) f$. Using \eqref{coerce gen 2} tells us that
 $$
  \| f\|_{s+k,b} \le C_b \mu^{-(k+1)/2} \| g \|_{s,b} + C_b \mu^{-1/2} \| (\S_\mu - \L_\mu) f\|_{s+k,b}.
 $$
Then we use \eqref{almost S} on the final term to get
 $$
  \| f\|_{s+k,b} \le C_b\mu^{-(k+1)/2} \| g \|_{s,b} + C_b\mu^{1/2} \| f\|_{s+k,b}.
 $$
 Now fix $b \in (0,b_c]$.
 By making $\mu$ sufficiently close to zero (call the threshold $\mu_\L(b)$) we have
 $C_b\mu^{1/2} \le 1/2$, in which case the last inequality implies
  $
  \| f\|_{s+k,b} \le C_b \mu^{-(k+1)/2} \| g \|_{s,b}.
 $
 This gives us the estimates in Lemma \ref{coerce lemma}.  \end{proof}

\subsection{Jost solutions of $\S_\mu$.}
To prove Lemmas \ref{range lemma} and \ref{gamma props}
we first need a good understanding of 
Jost solutions for the Schr\"odinger operator $\S_\mu$. This is to say, nontrivial solutions of 
\be\label{jost ivp}
\S_\mu \zeta = c^2 \mu \zeta'' + c^2 \mu \omega_\mu^2 \zeta + 4\sigma_c(x) \zeta =0.
\ee
The function $\gamma_\mu(x)$ whose properties are described in Lemmas \ref{range
lemma} and \ref{gamma props} will be shown
to be a perturbation of these Jost solutions in the next subsection.

Since \eqref{jost ivp} is a second order linear differential equation, we know there are two linearly independent solutions
 to this. 
Since $\sigma_c(x)$ is an even function, we may assume that one of these is odd (which we call $\zeta_\muo$)
and the other which is even (which we call $\zeta_\muz$).
To be clear, let $\zeta_{\mu,0}$ be the solution of $\S_\mu \zeta =0$ with initial conditions
\be\label{muz ic}
\zeta_\muz(0) = 1 \mand \zeta'_\muz(0) = 0.
\ee
Let $\zeta_\muo$ be the solution of $\S_\mu \zeta =0$ with initial conditions
\be\label{muo ic}
\zeta_\muo(0) = 0 \mand  \zeta'_\muo(0) = \omega_\mu.
\ee
We need some precise information about these functions and their behavior as $\mu \to 0^+$.
Most of what happens here is classical, but we do not know of a reference which contains the 
collection of results we need. Here is that collection:
\begin{lemma} \label{collection}
There exists $\mu_\zeta>0$ such that for $\mu \in (0,\mu_\zeta)$ and $j = 0,1$
there exist smooth functions $r_{\mu,j}(x)$ and $\phi_{\mu, j}(x)$ such that putting
\be\label{polar}\begin{split}
\zeta_{\mu,j}(x)&=r_{\mu,j}(x) \sin(\omega_\mu (x + \phi_{\mu,j}(x)))\\ \mand \zeta_{\mu,j}'(x)&=\omega_\mu r_\muj(x) \cos( \omega_\mu(x+\phi_{\mu,j}(x)))
\end{split}
\ee
solves $\S_\mu \zeta_{\muj}=0$ with initial conditions \eqref{muz ic} (for $j = 0$) and \eqref{muo ic} (for $j = 1$).

Moreover $r_\muj(x)$, $\phi_\muj(x)$ and $\zeta_\muj(x)$ have the following properties.
\begin{itemize}
\item We have
 \be\label{gen r est} 
\ds\|r_\muj\|_{W^{s,\infty}}\le C \mu^{-s/2}.\ee
\item
There are constants $0 < C_1 < C_2$ such that 
\be\label{r is 1}
C_1 < r_\muj(x) < C_2\ee for all $\mu \in (0,\mu_\zeta)$.
\item There exists $r_\muj^\infty$ such that 
\be\label{r converges}\sup_{ x \ge 0} e^{b_c x} |r_\muj(x) - r_\muj^\infty| = 0.\ee
\item  We have
\be\label{gen phi est}
\|\phi_\muj\|_{W^{1,\infty}} \le C \text{ and, for $s \ge 2$, } \|\phi_\muj\|_{W^{s,\infty}} \le C\mu^{-(s-1)/2}.
\ee
\item There exists $\phi_\muj^\infty$ such that 
\be\label{phi converges}
\sup_{ x \ge 0} e^{b_c x} |\phi_\muj(x) - \phi_\muj^\infty| = 0.
\ee 
\item There are constants $0 < C_1 < C_2$ such that 
\be\label{phi is 1}
C_1 < \phi_\muj^\infty < C_2\ee
 for all $\mu \in (0,\mu_\zeta)$.
 \item The maps $\mu \mapsto r_\muj^\infty$ and $\mu \mapsto \phi_\muj^\infty$ are continuous from $ (0,\mu_\zeta)$ into $L^\infty$.
 \item We have \be\label{overall phi est}
\| \zeta_\muj^{(s)}\|_{L^\infty} \le C \mu^{-s/2}.
\ee
\end{itemize}

%
%

\end{lemma}


\begin{proof}
Note that \eqref{polar} is essentially a polar coordinate decomposition of $\zeta_\muj(x)$.
Putting \eqref{polar} into $\S_\mu \zeta_\muj = 0$ gives the following system for $r_\muj(x)$ and $\phi_\muj(x)$:
\be\label{rp eqn}\begin{split}
r_\mu'(x) &= -{2 \over c^2\mu \omega_\mu} \sigma_c(x) r(x) \sin(2\omega_\mu(x+\phi_\muj(x))) \\
\phi_\muj'(x) &= {4 \over c^2\mu \omega_\mu^2} \sigma_c(x) \sin^2(\omega_\mu(x+\phi_\muj(x))).
\end{split}
\ee
If we are interested in $\zeta_{\muo}$ the we have $r_\muo(0) = 1$ and $\phi_\muo(0) = 0$. If we are
interested in $\zeta_\muz$ then we put $r_\muz(0) = 1$ and $\phi_\muz(0) = \pi/2 \omega_\mu$.
We will now focus on the what happens for $\zeta_\muo$; the other case is only different in minor details.

Before getting into the estimates, the exponential decay of $\sigma_c(x)$ (Theoren \ref{sigma exists}) implies that solutions of \eqref{rp eqn} will remain
bounded, and in fact converge, as $x \to \infty$. This, together with the fact that solutions of systems of differential equations
depend continuously on parameters, is enough to conclude that the maps $\mu \mapsto r_\muo \in L^\infty$ and $\mu \mapsto \phi_\muo \in L^\infty$ depend
continuously on $\mu$. We spare the details.

Now look at the second equation in \eqref{rp eqn}.
Since $\sigma_c(x)$ is positive we see immediately that $\phi_\muo(x)$ is an increasing function of $x$. 
This and the exponential decay of $\sigma_c(x)$ imply $\phi_\muo(x)$ will converge
 to some positive value $\phi_\muo^\infty$ as $x \to \infty$. The FTOC then 
 gives us the relation
\be\label{this is pmu}
\phi_\muo^\infty = {4 \over c^2\mu \omega_\mu^2}\int_{\R_+} \sigma_c(x) \sin^2(\omega_\mu(x+\phi_\muo(x)))dx.
\ee
Since $\mu \omega_\mu^2 =\O(1)$ we see that the above implies
\bes
\phi_\muo^\infty \le {4 \over c^2\mu \omega_\mu^2}\int_{\R_+} \sigma_c(x) dx < C_2
\ees
where $C_2>0$ is independent of $\mu$.

We claim that $\phi_\muo^\infty$ is also bounded below by a positive constant which is independent of $\mu$. Since $\phi_\muo(x)$ is increasing with $\phi_\muo(0) = 0$, it follows that $Y(x) := x+\phi_\muo(x)$ is 
an invertible map from $\R_+$ to itself and that $Y(x) \ge x$ for all $x \ge 0$. Thus we can make the change of variables $y = Y(x)$ in \eqref{this is pmu} to get
\be\label{cov phi}
\phi_\muo^\infty =  {4 \over c^2\mu \omega_\mu^2}\int_{\R_+} {\sigma_c(X(y))  \over 1 + \phi_\muo'(X(y))}\sin^2(\om_\mu y) dy
\ee
where $X(y)$ is the inverse of $Y(x)$. 

Since $Y(x) \ge x$ when $x \ge 0$ and $Y(0) = 0$ we deduce that $0 \le X(y) \le y$ for all $y \ge 0$. And since $\sigma_c$ is a decreasing function on $\R_+$
we have $\sigma_c(X(y)) \ge \sigma_c(y)$.
Next note that \eqref{rp eqn} implies $\phi_\muo' (x)\le\ds {4 \over c^2\mu \omega_\mu^2} \sigma_c(0)$ for all $x$ and thus
$
\ds(1+ \phi_\muo'(X(y))) \le 1 + {4 \over c^2\mu \omega_\mu^2}\sigma_c(0)$ for all $y$. Putting these into \eqref{cov phi} results in
\be\label{cov phi 2}
\phi_\muo^\infty \ge  {4 \over c^2\mu \omega_\mu^2 + 4 \sigma_c(0) }\int_{\R_+} \sigma_c(y) \sin^2(\om_\mu y) dy.
\ee

Since $\omega_\mu \to \infty$ as $\mu \to 0^+$, the Riemann-Lebesgue lemma implies that
$$\ds \lim_{\mu \to 0^+}
\int_\R {\sigma_c(x)  }\sin^2(\omega_\mu y) dy = {1 \over 2} \int_\R {\sigma_c(x)  }dy.
$$
The right hand side above is bounded below by a constant independent of $\mu$. Thus, for $\mu>0$ small enough we have
$
\phi_\muo^\infty> C_1> 0.
$
Therefore
$
0 < C_1 < \phi_\muo^\infty < C_2
$
for all $\mu$ sufficiently close to zero, which is \eqref{phi is 1}.

Using \eqref{this is pmu} and the fact that $\mu \omega_\mu^2 = \O(1)$ we see that
$\ds
|\phi_\muj(x) - \phi_\muj^\infty| \le C \int_x^\infty \sigma_c(x) dx.
$
We know from Theorem \ref{sigma exists} that $\sigma_c(x) \le C e^{-b_c |x|}$. Therefore
$
|\phi_\muj(x) - \phi_\muj^\infty| \le Ce^{-b_c x}
$
for $x > 0$. This is \eqref{phi converges}. To get \eqref{gen phi est} we use the second equation in \eqref{rp eqn} and bootstrap.


Now we need to estimate $r_\muo(x)$. We have, from the first equation in \eqref{rp eqn}, the formula
$$
r_\muo(x) = \exp \left(-{2 \over c^2 \mu \omega_\mu}  \int_0^x \sigma_c(y) \sin(2 \omega_\mu(y+\phi_\muo(y))) dy \right).
$$
The exponential decay of $\sigma_c(x)$ implies that $r_\muo(x)$ converges to a constant, denoted $r_\muo^\infty$, as $x \to \infty$. But notice that $1/\mu \omega_\mu = \O(\mu^{-1/2})$ and
so it not clear that this constant can be controlled in  way which is independent of $\mu$.  To get estimates for $r_\muo(x)$ we instead estimate
\be\label{this is E}
E_1(x) := c^2\mu [\zeta_\muo'(x)]^2 +(c^2\mu \omega^2_\mu + 4 \sigma_c(x)) [\zeta_\muo(x)]^2.
\ee
The change of variables \eqref{polar} implies that
\be\label{this is rr}
r_\muo^2(x) = \omega_\mu^{-2} [\zeta'_\muo(x)] + [\zeta_\muo(x)]^2.
\ee
Since $\omega_\mu = \O(\mu^{-1/2})$ and $\sigma_c(x)$ is positive and bounded above it straightforward to conclude that are constants $0<C_1 < C_2$, independent of $\mu$, such that \be\label{r is E}
C_1 r_\muo(x) <\sqrt{E_1(x)} < C_2 r_\muo(x)\ee holds for all $x$.

Differentiation of $E_1(x)$ with respect to $x$ followed by using \eqref{jost ivp} gives
\be\label{energy 1}
E_1' = 4\sigma'_c \zeta_\muo^2.
\ee
The definition of $E_1$ and the fact that $\mu \omega_\mu^2 = \O(1)$ implies that there is a ($\mu$-independent) constant $C>0$ such that
$$
\zeta_\muo^2(x) \le C E_1.
$$
This, together with the fact that $\sigma'_c(x) \le 0$ for $x\ge 0$ implies, by way of \eqref{energy 1}, that
$$
C \sigma_c' E_1 \le E_1' \le 0
$$
for all $x \ge 0$.
Or rather (since $E_1 \ge 0$)
$$
C \sigma'_c \le E_1'/E_1 \le 0.
$$
Integrating the above from $0$ to $x\ge0$ gives (since $E_1(0) = \mu \omega_\mu^2$):
$$
C (\sigma_c(x) - \sigma_c(0)) \le \ln(E_1(x)/\mu \omega_\mu^2) \le 0.
$$
Thus
$$
\mu \omega_\mu^2 e^{ C (\sigma_c(x) - \sigma_c(0))}  \le E_1(x) \le \mu \omega_\mu^2.
$$
Since $\mu \omega_\mu^2 = \O(1)$ this implies that, for all $x \in \R_+$, we have constants $0 < C_1 < C_2$ (independent of $\mu$) such that
$
C_1 < E_1(x) < C_2.
$
Thus we have
$\ds
0 < C_1 < r_\muo(x) < C_2.
$
This is \eqref{r is 1}.

Using this, the estimate \eqref{gen phi est} and the first equation in \eqref{rp eqn} gives \eqref{gen r est} by way of bootstrapping.
Note that \eqref{gen phi est} and \eqref{gen r est} yield \eqref{overall phi est}.

Now we prove \eqref{r converges}. If we integrate \eqref{energy 1} on $\R_+$ we find that $E_1(x)$ converges as $x \to \infty$ (since $\sigma_c'(x)$
decays exponentially) to some limit $E_1^\infty$. And so we have, using the estimate for $\sigma_c$ in Theorem \ref{sigma exists},
\be\label{E111}
|E_1(x)-E_1^\infty| \le C\int_x^\infty \left |\sigma_c'(y)\right| dy \le C e^{-b_c x}
\ee
when $x \ge 0$. 
Next notice that the equation \eqref{this is E} and \eqref{this is rr} tell us that
$$
E_1(x) - c^2 \mu \omega_\mu^2 r_\muo^2(x) = 4 \sigma_c(x) [\zeta_\muo(x)]^2.
$$
Using \eqref{overall phi est} and the exponential decay of $\sigma_c(x)$ here tell us that
\be\label{e1est}
|E_1(x) - c^2 \mu \omega_\mu^2 r_\muo^2(x)| \le C e^{-b_c x}
\ee
for all $x \ge 0$. This implies
\be\label{relate}
E_1^\infty = c^2 \mu \omega_\mu^2 [r^\infty_\muo]^2.
\ee

Using \eqref{r is 1} give
\bes\begin{split}
|r_\muo(x) - r_\muo^\infty| = |r_\muo(x) + r_\muo^\infty|^{-1}|r^2_\muo(x) - [r_\muo^\infty]^2| \le C|r^2_\muo(x) - [r_\muo^\infty]^2|.
\end{split}\ees
Then we use \eqref{relate} and the fact that $\mu \omega_\mu^2 = \O(1)$ to get
\bes\begin{split}
|r_\muo(x) - r_\muo^\infty| \le  C|c^2\mu \omega_\mu^2 r^2_\muo(x) - E_1^\infty|.
\end{split}\ees
The triangle inequality give
\bes\begin{split}
|r_\muo(x) - r_\muo^\infty| \le  C|c^2\mu \omega_\mu^2 r^2_\muo(x) -E_1(x)| + C|E_1(x)- E_1^\infty|.
\end{split}\ees
Then \eqref{E111} and \eqref{e1est} give $|r_\muo(x) - r_\muo^\infty| < C e^{-b_c x}$ when $x\ge0$. This is \eqref{r converges}.

\end{proof}

\subsection{Jost solutions of $\L_\mu^*$}

In this subsection we show the existence of a nontrivial, odd, smooth bounded function $\gamma_\mu(x)$
for which $\L_\mu^* \gamma_\mu = 0$. It is this function $\gamma_\mu$ which is described in Lemmas \ref{range lemma} and \ref{gamma props}.
By $\L_\mu^*$ we mean the $L^2 \times L^2$ adjoint of $\L_\mu$, specifically\footnote{Here is the formula for $\Sigma_{\mu,2}^*$:
$$
\Sigma_{\mu,2}^* f = 2 T_\mu^* Q_0(T_\mu \sigmab_{\mu,c},(T_\mu^*)^{-1}L_\mu^* (f \jb))\cdot \jb.
$$
The $L^2 \times L^2$ adjoints of $T_\mu$ and $L_\mu$ are easily computed from their definitions in Section~\ref{equation of motion} and the observation that $A$
is symmetric and $\delta$ is anti-symmetric with respect to the $L^2$ inner-product.}
$$
\L_\mu^* f  =  c^2 \mu f'' + 2(1 + \mu A^2 + \mu^2 \tau_\mu)f + \Sigma_{\mu,2}^* f.
$$
If we can find such a $\gamma_\mu$, then for any $f \in O^{s+2}_b$ (with $b>0$) the adjoint property tells us that
$$
\L_\mu f= g \implies \int_\R \gamma_\mu(x) g(x) dx =0.
$$
This is the first conclusion in Lemma \ref{range lemma}.
It turns out that the function $\gamma_\mu$ we seek is an $\O(\mu^{1/4})$ perturbation\footnote{In fact,
it is an $\O(\mu^{1/2} |\ln(\mu)|)$ perturbation, but getting this sharper
bound is more than we need here and including it would lengthen this already lengthy appendix.}
 of $\zeta_{\muo}$ (from the previous subsection) in the $L^\infty$ norm.  
 
\subsubsection{Decomposition of $\L_\mu$} 
 To show this we first decompose $\L_\mu^*$ as
 $$
\L_\mu^* = \S_\mu +\mu \Delta_\mu + \mu K_\mu^*
$$
where
\be\label{K and Delta}
K^*_\mu : = \mu^{-1}  (\Sigma^*_{\mu,2}  - \Sigma^*_{0,2} )
\mand
\mu \Delta_\mu := 2(1 + \mu A^2 + \mu^2 \tau_\mu) -c^2 \mu \omega_\mu^2.
\ee
In the above we used that fact that $\Sigma^*_{0,2} f = 4 \sigma_c(x) f$. 

The line of reasoning that led to the estimates for $\Sigma_{\mu,2}$ at \eqref{sigma estimate} can be repeated to show that
 \be\label{localizing}
\| K^*_\mu f\|_{W^{k,p}_{b+b_c}} \le C \| f\|_{W^{k,p}_{b}}.
\ee
Note that this estimate implies that $K^*_\mu$ is a localizing operator.

Also, the definition of $\tau_\mu$ at \eqref{this is tau} and the estimate \eqref{LT is bounded} tell us
\be\label{Delta est}
\|\Delta_\mu f\|_{W^{k,p}_{b}}\le C\|f\|_{W^{k,p}_b}.
\ee
Moreover, the definitions of $\tau_\mu$  and $\omega_\mu$ (in \eqref{crit freq}) imply that
\be\label{why Delta}
 \Delta_\mu \cos(\omega_\mu x) =  \Delta_\mu \sin(\omega_\mu x) = 0.
\ee

If  $\L_\mu^* \gamma_\mu = 0$ then clearly
$
\ds \S_\mu \gamma_\mu = - \mu \Delta_\mu \gamma_\mu - \mu K^*_\mu \gamma_\mu.
$
Since $\gamma_\mu$ is odd we may normalize it so that $\gamma_\mu'(0) = \omega_\mu$. In which case the variation of parameters formula
tells us that
\be\label{fp}\begin{split}
\gamma_\mu(x) -  \zeta_\muo(x) =  & -\mu \omega_\mu \zeta_\muo(x) \int_0^x( \Delta_\mu \gamma_\mu (y) + K_\mu^*\gamma_\mu(y)) \zeta_\muz(y) dy \\
&+\mu \omega_\mu \zeta_\muz(x) \int_0^x ( \Delta_\mu \gamma_\mu (y) + K_\mu^*\gamma_\mu(y))\zeta_\muo(y) dy .
\end{split}\ee

So for functions $f(x)$ put
\begin{equation}\label{V}\begin{split}
V_\mu f (x) : =  &-\mu \omega_\mu \zeta_\muo(x) \int_0^x( \Delta_\mu f (y) + K_\mu^* f(y)) \zeta_\muz(y) dy \\
&+\mu \omega_\mu \zeta_\muz(x) \int_0^x ( \Delta_\mu f(y) + K_\mu^*f(y))\zeta_\muo(y) dy .
\end{split}
\end{equation}
Letting $u_\mu := \gamma_\mu - \zeta_\muo$ we see that \eqref{fp} is equivalent to 
\be\label{f eqn}
(1- V_\mu)u_\mu = V_\mu \zeta_\muo.
\ee
Our goal will be to show that the operator norm of $V_\mu$ is less than one for $\mu$ small enough. 
Thus by the Neumann series $(1-V_\mu)$ is invertible
and we can solve \eqref{f eqn}.

\subsubsection{Asymptotically sinusoidal functions}
However, we will not work in $L^\infty$ but rather in
\be\label{AS}\begin{split}
AS^s_b:=& \left( W^{s,\infty}_b\cap\left\{ \text{odds} \right\} \right)\oplus\spn \left\{ \sin(\omega_\mu x) \right\} \oplus \spn \left\{ i_1(x) \cos(\omega_\mu x)\right\}.
\end{split}\ee
In the above, $i_1(x)$ is a smooth, odd, non-decreasing function such that
\be\label{i1 def}
i_1(x)=\begin{cases} 
0 &\text{when } 0\le x\le1/2\\
1 &\text{when } x \ge 2.
\end{cases}
\ee
%

The functions in $AS^s_b$  are odd functions which are asymptotically sinusoidal, hence the 
name of the space. It is obvious that $AS^s_b$ is a subspace of $L^\infty$.
So long as $b >0$,  $AS^s_b$ is the direct sum of three subspaces of $L^\infty$ whose intersections are trivial. Thus, if $f \in AS^s_b$ there exist
unique $\ell_f \in W^{s,\infty}_b \cap \left\{ \text{odds}\right\}$ and $\alpha_f ,\beta_f \in \R$ such that
\be\label{f decomp}
f(x) = \ell_f(x) + \alpha_f \sin(\omega_\mu x) + \beta_f i_1(x) \cos(\omega_\mu x).
\ee
Moreover,
$AS^s_b$ is a Banach space with norm:
$$
\| f\|_{AS^s_b} := \|\ell_f\|_{W^{s,\infty}_b} + \omega_\mu^s\sqrt{\alpha_f^2+ \beta^2_f}.
$$
The $\omega_\mu^s$ factor is there to capture the oscillatory part's derivative, though is, strictly speaking, not really
needed.

We need a few estimates. First we have the easy embedding estimate 
\be\label{easy emb}
\| f\|_{W^{s,\infty}} \le C
\| f\|_{AS^s_b}.
\ee Then we have the ``ideal" estimate which states that if $f \in AS^s_b$ and $g \in W^{s,\infty}_b$ then $fg \in W^{s,\infty}_b$
with the estimate
\be\label{ideal}
\| fg\|_{W^{s,\infty}_b} \le C \| f\|_{AS^s_b} \| g\|_{W^{s,\infty}_b} .
\ee
Here is why:
$$
f(x)g(x) = g(x) \ell_f(x) + g(x) \left( \alpha_ f \sin(\omega_\mu x) + \beta_f i_1(x) \cos(\omega_\mu x)\right).
$$
Both terms on the right hand side are in $W^{s,\infty}_b$ and the rest is bookkeeping.

\subsubsection{The size of $\zeta_\muo$}
The Jost solution $\zeta_\muo$ is, unsurprisingly, in $AS^s_{b}$.
In particular we have this estimate for all $b \in [0,b_c]$:
\be\label{phi is big}
\|\zeta_{\muo}\|_{AS^s_{b}} \le C \mu^{-s/2} \mu^{-b/2b_c}.
\ee
Note that the larger $b$ is, the larger the right hand side is. 
This is because it takes a while for $\zeta_\muo$ to settle into its asymptotic state, and thus the 
exponential weight takes a greater toll.
Here is the computation.

Adding a lot of zeroes to $\zeta_\muo(x)$ gives us
\be\label{zeta decomp}
\zeta_\muo(x) = {\ell_1 + \ell_2 + \ell_3 +\ell_3}+ \alpha_{\zeta_{\muo}} \sin(\omega_\mu x) +  \beta_{\zeta_{\muo}} i_1(x) \cos(\om_\mu x)
\ee
where
\be\label{zeta decomp defs}\begin{split}
\ell_1:= & (r_\muo(x) - r_\muo^\infty)\sin(\omega_\mu (x+\phi_\muo(x)))\\
\ell_2:=&r_\muo^\infty \sin(\omega_\mu x) [\cos(\omega_\mu\phi_\muo(x))-\cos(\omega_\mu \phi_\muo^\infty)]\\\
\ell_3:=&r_\muo^\infty  \cos(\omega_\mu x)(1-|i_1(x)|) \sin(\omega_\mu \phi_\muo(x)) \\
\ell_4:=&r_\muo^\infty  \cos(\omega_\mu x)i_1(x) [\sin(\omega_\mu \phi_\muo(|x|))-\sin(\omega_\mu \phi_\mu^\infty)] \\
 \alpha_{\zeta_{\muo}}:=&r_\muo^\infty  \cos(\omega_\mu\phi_\muo^\infty)\\
 \beta_{\zeta_{\muo}}:=&r_\muo^\infty \sin(\omega_\mu \phi_\muo^\infty).
\end{split}
\ee
In the above we have used the polar decomposition \eqref{polar}, the addition of angles formula
and the fact that if $o(x)$ is odd then $o(x) = \sgn(x) o(|x|)$.

The functions  $\ell_1$, $\ell_2$, $\ell_3$ and $\ell_4$ are in $W^{s,\infty}_{b_c}$, as we show in a moment.
If so then we have
$$
\| \zeta_{\muo}\|_{AS^s_b} \le \| \ell_1\|_{W^{s,\infty}_b}+\| \ell_2\|_{W^{s,\infty}_b}+\| \ell_3\|_{W^{s,\infty}_b}+\| \ell_4\|_{W^{s,\infty}_b}+
\omega_\mu^{s}r_\muo^\infty
$$
We have used the fact that $\sqrt{\alpha^2_{\zeta_{\muo}}+\beta^2_{\zeta_{\muo}}} = r_\muo^\infty$.

The estimates in \eqref{gen r est} and  \eqref{r converges} imply that $\|\ell_1\|_{W^{s,\infty}_b} \le C\mu^{-s/2}$ when $b \in [0,b_c]$.
Next, $|i_1(x)| - 1$ is compactly supported and smooth.
This, with \eqref{gen phi est}, give us $\| \ell_3 \|_{W^{s,\infty}_{b}} \le C\mu^{-s/2}$ for any $b \in [0,b_c]$.


 Now we address $\ell_4$; to estimate $\ell_2$ is similar.
 First notice that
 \be\label{ell41}
\| \ell_4\|_{L^\infty} \le C
\ee
since $|r_\muo^\infty|= \O(1)$.
Next, since $\sin(x)$ is globally Lipschitz, we have
$$
|\ell_4(x)| \le |r_\muo^\infty\omega_\mu||\phi_\muo(|x|)-\phi_\muo^\infty|.
$$
Since $\omega_\mu = \O(\mu^{-1/2})$ and $|r_\muo^\infty| = \O(1)$  we have
$$
|\ell_4(x)| \le C \mu^{-1/2} |\phi_\muo(x)-\phi_\muo^\infty|.
$$
Then \eqref{phi converges} tells us that
\be\label{ell42}
\|\ell_4\|_{L^\infty_{b_c}} \le C\mu^{-1/2}.
\ee
Interpolating between this estimate and \eqref{ell41} implies that
\be\label{ell43}
\|\ell_4\|_{L^\infty_{b}} \le C \mu^{-b/2b_c}.
\ee 
when $b \in [0,b_c]$.
The same sort of reasoning gets us
\be\label{ells}
\|\ell_2\|_{W^{s,\infty}_{b}}+\|\ell_4\|_{W^{s,\infty}_{b}} \le C\mu^{-s/2} \mu^{-b/2b_c}.
\ee 

Putting all of these estimates together gives \eqref{phi is big}.
A parallel argument shows  we have the same sort of estimate for $\zeta_\muz(x)$.

\subsubsection{Localization}
Next we claim that both $\Delta_\mu$ and $K_\mu^*$ map $AS^s_b \to W^{s,\infty}_b$. 
Which is to say that these operators localize asymptotically sinusoidal functions.
Specifically we have the
estimates
\be\label{localizing 2}
\| \Delta_\mu f\|_{W^{s,\infty}_b} \le C\mu^{-s/2} \| f\|_{AS^s_b}
\mand
\| K_\mu^* f\|_{W^{s,\infty}_b} \le C \mu^{-s/2} \| f\|_{AS^s_b}.
\ee
The estimate for $K_\mu^*$ is just a direct consequence of \eqref{localizing} and the details
are uninteresting. 

The 
estimate for $\Delta_\mu$ follows from \eqref{why Delta}.
Since $\Delta_\mu \cos(\omega_\mu x) = 0$ we have
$$\Delta_\mu (i_1(x) \cos(\omega_\mu x)) = \Delta_\mu(i_1(x) \cos(\omega_\mu x)) - i_1(x) \Delta_\mu(\cos(\omega_\mu x)).$$ Using the definition of $\Delta_\mu$ 
at \eqref{K and Delta}
converts this to
$$
\Delta_\mu (i_1(x) \cos(\omega_\mu x))  = 2 A^2(i_1(x) \cos(\omega_\mu x)) - 2i_1(x) A^2 (\cos(\omega_\mu x)).
$$
Applying the definition of $A$ to this gives
\bes\begin{split}
\Delta_\mu (i_1(x) \cos(\omega_\mu x))  =&
{1 \over 2} (i_1(x+2)-i_1(x)) \cos(\omega_\mu(x+2)) \\
+&{1 \over 2} (i_1(x-2)-i_1(x)) \cos(\omega_\mu(x-2)) .
\end{split}\ees
One can check from the definition that $(i_1(x+2)-i_1(x))$ has support in $[-5,5]$
as does $(i_1(x-2)-i_1(x))$.  In $L^\infty$ these functions are no bigger than one. Which is to say that $\Delta_\mu(i_\mu(x) \cos(\omega_\mu(x))$ is compactly supported with $\O(1)$ magnitude. 
Thus
$$
\| \Delta_\mu (i_1(\cdot) \cos(\omega_\mu \cdot))\|_{W^{s,\infty}_b} \le C \mu^{-s/2}.
$$
We also know from \eqref{Delta est} that $\|\Delta_\mu \ell \|_{W^{s,\infty}_b} \le \| \ell\|_{W^{s,\infty}_b}$
and from \eqref{why Delta} that $\Delta_\mu \sin(\om_\mu x) = 0$.
Thus if $f \in AS^s_b$ we have, from \eqref{f decomp},
$$
\Delta_\mu f = \Delta_\mu \ell_f + \beta_{f}\Delta_\mu (i_1(x) \cos(\omega_\mu x))
$$
and the estimate for
$\Delta_\mu$ in \eqref{localizing 2}
follows in the obvious way from the preceding estimates.

%

\subsubsection{$V_\mu$ is small}
Now let us estimate $V_\mu f$. 
Assume $f \in AS^0_{b_c/4}$; we need to estimate $V_\mu f$ in this same space.
We show how to estimate the first term in \eqref{V}, 
$$
V^1_\mu f(x):=-\mu \omega_\mu \zeta_\muo(x) \int_0^x (\Delta_\mu f(y) + K_\mu^* f(y)) \zeta_\muz(y) dy.
$$
The second term, denoted $V^2_\mu f(x)$, is no different.

From calculus, we have
\be\label{V1}
V^1_\mu f(x)=-\mu \omega_\mu \zeta_\muo(x) Z_1^\infty 
+
\mu \omega_\mu \zeta_\muo(x)Z_1(x)
\ee
where
\begin{multline}\label{Zs}
Z_1^\infty = \int_0^\infty (\Delta_\mu f(y) + K_\mu^* f(y)) \zeta_\muz(y) dy\\
\mand
Z_1(x):=\int_x^\infty (\Delta_\mu f(y) + K_\mu^* f(y)) \zeta_\muz(y) dy.
\end{multline}

Recalling from \eqref{lp emb} that
$\| f\|_{L^1} \le C_b \| f\|_{L^\infty_b}$, the localizing property \eqref{localizing 2}
tells us
$$
|Z_1^\infty| \le \| \Delta_\mu f\|_{L^1} + \| K_\mu^* f\|_{L^1} \le 
C  \| \Delta_\mu f\|_{L^\infty_{b_c/4}} +C \| K_\mu^* f\|_{L^\infty_{b_c/4}} \le C\| f\|_{AS^0_{b_c/4}}.
$$
And so
using $\om_\mu = \O(\mu^{-1/2})$ gives us
$$
\| \mu \omega_\mu \zeta_\muo Z_1^\infty \|_{AS^0_{b_c/4}}\le C \mu^{1/2} \|\zeta_\muo\|_{AS^0_{b_c/4}} |Z_1^\infty| \le C \mu^{1/2} \|\zeta_\muo\|_{AS^0_{b_c/4}} \| f\|_{AS^0_{b_c/4}}.
$$
Then we use \eqref{phi is big} with $b = b_c/4$ to get
$$
\| \mu \omega_\mu \zeta_\muo Z_1^\infty \|_{AS_{b_c/4}}\le C\mu^{3/8} \| f\|_{AS_{b_c/4}}.
$$
Thus the first term in \eqref{V1} is estimated.

Next, if $g \in L^\infty_b$ then one has
$$\ds \int_0^\infty e^{bx}\left \vert \int_x^\infty g(y) dy \right \vert dx 
\le 
 \int_0^\infty e^{bx} \int_x^\infty e^{-by}\|g\|_{L^\infty_b} dy dx 
\le C_b \| g\|_{L^\infty_b}.
$$
Using this, together with the
 boundedness of $\zeta_\muz$ from \eqref{gen phi est}, gives
$$
\left \|Z_1 \right\|_{L^\infty_{b_c/4}} \le C\| \Delta_\mu f\|_{L^\infty_{b_c/4}} +\| K_\mu^* f\|_{L^\infty_{b_c/4}} .
$$
And so
using \eqref{ideal} and $\om_\mu = \O(\mu^{-1/2})$ gives
\bes\begin{split}
\left\|\mu \om_\mu \zeta_\muo Z_1
\right \|_{L^\infty_{b_c/4}} 
\le C\mu^{1/2}\|\zeta_\muo\|_{AS^0_{b_c/4}} \left(\| \Delta_\mu f\|_{L^\infty_{b_c/4}} +\| K_\mu^* f\|_{L^\infty_{b_c/4}} \right).
\end{split}
\ees
Using \eqref{localizing 2} and \eqref{phi is big} with $b = b_c/4$ converts the above to
\bes\begin{split}
\left\|\mu \om_\mu \zeta_\muo Z_1 \right \|_{L^\infty_{b_c/4}} 
\le C\mu^{3/8} \|f\|_{AS^0_{b_c/4}}.
\end{split}
\ees
%
Thus we have an estimate for the second term in \eqref{V1}.

We have therefore shown that $\|V^1_\mu f\|_{AS^0_{b_c/4}} \le C\mu^{3/8} \|f\|_{AS_{b_c/4}}$. The other term in $V_\mu$ is handled in like fashion and we have
\be\label{V contract}
\|V_\mu f\|_{AS^0_{b_c/4}} \le C\mu^{3/8} \|f\|_{AS^0_{b_c/4}}
\ee

\subsubsection{Inversion and estimates}
For $\mu$ small enough, $(1-V_\mu)$ is invertible on $AS^0_{b_c/4}$ by the Neumann series. And so we have a solution
of \eqref{f eqn}
$$
u_\mu = (1-V_\mu)^{-1} V_\mu \zeta_\muo.
$$
which has the following estimates (from \eqref{phi is big}):
\be\label{u est}
\|u_\mu\|_{AS^0_{b_c/4}} \le C\mu^{3/8} \| \zeta_\muo\|_{AS^0_{b_c/4}} \le C\mu^{1/4}.
\ee
Also, \eqref{easy emb} gives
$$\|u_\mu\|_{L^\infty} \le \| u_\mu\|_{AS^0_{b_c/4}} \le C\mu^{1/4}.$$

It is easy enough to conclude that $u_\mu$ is a smooth function of $x$ by a bootstrapping argument
and in this way get \eqref{gamma est}.
This, however, does not imply that $u_\mu \in AS^s_{b_c/4}$ for all $s\ge0$.\footnote{For instance $\sin(\cosh(x))\sech(x) \in AS^0_1$ and is  smooth,  but it is not in $AS^1_1$.}
So, to get \eqref{gamma limit} in Lemma \ref{gamma props}, we need to show that $u_\mu \in AS_{b_c/4}^1$.

 The FTOC shows that
$$
{d \over dx} V_\mu f(x) = \tilde{V}_\mu f(x) 
$$
where
\begin{equation*}\begin{split}
\tilde{V}_\mu f (x) : =  &-\mu \omega_\mu \zeta'_\muo(x) \int_0^x( \Delta_\mu f (y) + K_\mu^* f(y)) \zeta_\muz(y) dy \\
&+\mu \omega_\mu \zeta'_\muz(x) \int_0^x ( \Delta_\mu f(y) + K_\mu^*f(y))\zeta_\muo(y) dy .
\end{split}
\end{equation*}
This operator is nearly identical to $V_\mu$, the only difference being that the  prefactor functions $\zeta_\muj(x)$ have been differentiated. Repetition of the the same  steps  that got us \eqref{V contract} gets us
\be\label{VV contract}
\|\tilde{V}_\mu f\|_{ES^0_{b_c/4}} \le C\mu^{-1/8} \|f\|_{AS^0_{b_c/4}}.
\ee
Here, $ES_b^s$ is analogous to $AS_b^s$ except it consists of asymptotically sinusoidal {\it even} functions. We forgo the specifics. Unsurprisingly, if $u \in AS^1_b$ then $u' \in ES^0_b$.

We know that $(1-V_\mu) u_\mu = V_\mu \zeta_\muo$ and so
$$
u_\mu' =\tilde{V}_\mu u_\mu + \tilde{V}_\mu \zeta_\muo.
$$
The estimate for $\tilde{V}_\mu$ then tells us that $\| u_\mu'\|_{ES^0_{b_c/4}} \le C\mu^{-1/8}\left( \|u_\mu\|_{AS^0_{b_c/4}}
+ \|\zeta_\muo\|_{AS_{b_c/4}^0}\right)\le C.$ 

So we know now that $\gamma_\mu = \zeta_\muo + u_\mu \in AS^1_{b_c/4}$.  The estimate in \eqref{gamma est}
follows from those for $\zeta_\muo$, $u_\mu$ and \eqref{easy emb}.
Also, we know there exists $\ell_{\gamma_\mu} \in W^{1,\infty}_{b_c/4}$ and constants $\alpha_{\gamma_\mu}, \beta_{\gamma_\mu}$ such that
\be\label{gamma decomp}
\gamma_\mu(x) = \ell_{\gamma_\mu}(x) + \alpha_{\gamma_\mu} \sin(\om_\mu x) +
\beta_{\gamma_\mu} i_1(x) \cos(\om_\mu x).
\ee
Which means that 
\be\label{doo}\begin{split}
&\lim_{x \to \infty}|\gamma_\mu(x) - \alpha_{\gamma_\mu} \sin(\om_\mu x) - 
\beta_{\gamma_\mu} \cos(\om_\mu x)| \\=& 
\lim_{x \to \infty}|\gamma'_\mu(x) - \om_\mu \alpha_{\gamma_\mu} \cos(\om_\mu x) - 
\om_\mu\beta_{\gamma_\mu} \sin(\om_\mu x)|\\
=&0.
\end{split}
\ee

Furthermore, 
because of \eqref{u est} and \eqref{zeta decomp}, we have
\be\label{same at inf}
\sqrt{(\alpha_{\gamma_\mu} - \alpha_{\zeta_\muo})^2 +
(\beta_{\gamma_\mu} - \beta_{\zeta_\muo})^2}
 \le C \mu^{1/4}.
\ee
We know, from \eqref{r converges} and \eqref{phi converges}, that $\zeta_\muo(x)$
converges as $x \to \infty$ to $r_\muo^\infty\sin(\om_\mu(x+\phi_\muo^\infty))$ with $r_\muo^\infty$
and $\phi_\muo^\infty$ both $\O(1)$.
With this, an exercise
in trigonometry shows that
$$
\alpha_{\gamma_\mu} \sin(\om_\mu x)  +
\beta_{\gamma_\mu} \cos(\om_\mu x)=\varrho_\mu^\infty \sin(\om_\mu (x + \vartheta_\mu^\infty))
$$
for some constants $\varrho_\mu^\infty$ and $\vartheta_\mu^\infty$ which satisfy:
$$
|\varrho^\infty_{\mu} - r_\muo^\infty| \le C\mu^{1/4}
$$
and
\be\label{vartheta est}
|\vartheta_\mu^\infty - \phi_\muo^\infty| \le C\mu^{3/4}.
\ee
Since $r_\muo^\infty$ and $\phi_\muo^\infty$ are $\O(1)$, we see that $\varrho_\mu^\infty$ and $\vartheta_\muo^\infty$ are likewise $\O(1)$. 
Thus we can renormalize $\gamma_\mu$ so that $\varrho_\mu^\infty = 1$ exactly and not change
anything of substance. In this way, \eqref{doo} gives us the map $\vartheta_\mu^\infty$ described in Lemma \ref{gamma props} and the estimates \eqref{vartheta is 1} and \eqref{gamma limit}.

\subsection{The set $M_c$}
Now we establish the existence of the set $M_c$ described in Lemma \ref{gamma props}.
Let
$$
\tilde{M}_c:=\left\{ \mu \in (0,\mu_\zeta): \sin(\om_\mu \phi_\muo^\infty)>3/4\right\}.
$$
First, $\om_\mu$ and $\phi_\muo^\infty$ are continuous functions of $\mu$ and thus $\sin(\om_\mu \phi_\muo^\infty)$
is likewise continuous. Since $\tilde{M}_c$ is the preimage of an open set it is open. Moreover,
we know from \eqref{phi is 1} that $\phi_\muo^\infty = \O(1)$ and from \eqref{omega est} that $\om_\mu = \O(\mu^{-1/2})$. Thus $\om_\mu \phi_\muo^\infty = \O(\mu^{-1/2})$ which means that it diverges
to $\infty$ as $\mu \to 0^+$. This, with continuity, imply that there exists $n_0\ge0$ and a sequence
$\{ \mu_n \}_{n\ge n_0}\subset \R_+$ 
with $\ds \lim_{n \to \infty} \mu_n = 0$ for which $\om_{\mu_n} \phi_{\mu_n,1}^\infty = (2n+1)\pi/2$.
Thus we have $\sin(\omega_{\mu_n} \phi_{\mu_n,1}^\infty) = 1$ for all $n$ and in this way we see that $0 \in \overline{\tilde{M}}_c$.
 
 Now, take $\mu \in \tilde{M}_c$. By the addition of angles formula we have
 \begin{multline}\label{contin}
 \sin(\om_\mu \vartheta_\mu^\infty) - \sin(\om_\mu \phi_\muo^\infty)\\=\sin(\om_\mu \phi_\muo^\infty)\left( \cos\left(\om_\mu \left(\vartheta_\mu^\infty-\phi^\infty_\muo \right)\right)-1\right)+
  \cos(\om_\mu \phi_\muo^\infty) \sin\left(\om_\mu \left(\vartheta_\mu^\infty-\phi^\infty_\muo \right)\right).
 \end{multline}
From $\eqref{vartheta est}$ we see that $|\om_\mu \left(\vartheta_\mu^\infty-\phi^\infty_\muo \right)|\le C\mu^{1/4}$
and thus, for $\mu>0$ small enough (call the threshold $\mu_\vartheta$) we have 
$$
|\sin\left(\om_\mu \left(\vartheta_\mu^\infty-\phi^\infty_\muo \right)\right)| + |\cos\left(\om_\mu \left(\vartheta_\mu^\infty-\phi^\infty_\muo \right)\right) -1|\le 1/4.
$$
This, with \eqref{contin} give
$$
| \sin(\om_\mu \vartheta_\mu^\infty) - \sin(\om_\mu \phi_\muo^\infty)| \le 1/4.
$$
Since $\mu \in \tilde{M}_c$ we know that $\sin(\om_\mu \phi_\muo^\infty) >3/4$ and therefore
if $\mu \in (0,\mu_\vartheta)$ 
the triangle inequality give
$$
\sin(\omega_\mu \vartheta_\mu^\infty) >1/2.
$$
Thus if we put $M_c :=\tilde{M}_c \cap(0,\mu_\vartheta)$ we have \eqref{Mc def}.

\subsection{The oscillatory integral estimate}
As for 
\eqref{stationary phase}, we want to estimate:
$$
\iota_\mu[g]:=\int_\R \gamma_\mu(x) g(x) dx .
$$
Since $\L_\mu^* \gamma_\mu = 0$ we can rearrange terms to find
$$
\gamma_\mu(x) =\mu \mathcal{R}_\mu \gamma_\mu.
$$
where
$$
\mathcal{R}_\mu f(x):=
 -{1 \over 2+4\sigma_c(x)}
\left(
c^2 f'' (x) +  A^2 f(x) + \mu \tau_\mu f(x)
+ K_\mu^*f(x)
\right) .
$$
Our previous estimates tell us that $\| \RR f \|_{s,b} \le C\| f\|_{s+2,b}$
and $\| \RR^* f\|_{s,b} \le C \| f\|_{s+2,b}$.

Thus we have
$$
\iota_\mu[g]=
\mu \int_\R (\RR_\mu \gamma_\mu(x) ) g(x)dx = 
\mu \int_\R  \gamma_\mu(x)  \RR^* g(x)dx. 
$$
Repeated replacement of  $\gamma_\mu$ with $\mu \RR_\mu \gamma_\mu$
then gives us, for any positive integer $n$:
$$
\iota_\mu[g] := \mu^n \int_\R  \gamma_\mu(x)  (\RR^*)^n g(x)dx. 
$$
This leads to
$$
|\iota_\mu[g]| \le \mu^n \| \gamma_\mu\|_{L^\infty} \|  (\RR^*)^n g\|_{L^1}  \le C \|  (\RR^*)^n g\|_{L^1} 
$$
Then we use the estimate $\| f\|_{L^1} \le C b^{-1/2} \| f\|_{0,b}$ to get
$$
|\iota_\mu[g]| \le Cb^{-1/2}\mu^n  \|  (\RR^*)^n g\|_{0,b} \le   Cb^{-1/2}\mu^n  \|  g\|_{2n,b}. 
$$
This is the estimate \eqref{stationary phase} for $s = 2n$ and interpolation estimates can be used to get it 
for odd values of $s$.

%
%
%

\subsection{Sufficiency}\label{fred sect}
At this stage we have proven everything in Lemmas \ref{coerce lemma}, \ref{range lemma} and \ref{gamma props}
with the exception of the sufficient condition \eqref{B solvability suff} in Lemma \ref{range lemma}.
The existence of the bounded function $\gamma_\mu$ for which $\L_\mu^* \gamma_\mu = 0$ implies
that codimension of the range of $\L_\mu$ (when viewed as an unbounded map on $O^s_b$ with $b>0$) is at least one. In this section we prove that the codimension of 
the range is exactly one, which then implies \eqref{B solvability suff}.

Recollecting the defintion of  $\L_\mu$ at \eqref{light equation v1}, we let $\L_\mu^0:= c^2 \mu \partial_x^2 + 2(1+\mu A^2 + \mu^2 \tau_\mu)$ so that we have $\L_\mu= \L_\mu^0 + \Sigma_{\mu,2}$.
We are viewing both $\L_\mu^0$ and $\L_\mu$ as unbounded operators on $O^s_b$ with domain $O^{s+2}_b$.
We have the compact embedding $O^{s'}_{b'} \subset \subset O^s_b$ when $s' > s$ and $b' > b$. 
Thus, the localizing property of $\Sigma_{\mu,2}$ described in \eqref{sigma estimate} implies that $\Sigma_{\mu,2}$ is compact
relative to $\L_\mu^0$. This in turn implies that the Fredholm index\footnote{We use the defintion in \cite{kato}
for the  Fredholm index of an unbounded operator.}  of $\L_\mu$ and $\L_\mu^0$ coincide.

The Fredholm index of $\L_\mu^0$ can be computed relatively easily using Fourier methods. We have
$$
\L_\mu^0 e^{i \omega x} = \bunderbrace{[-c^2 \mu \om^2 + 2(1 + \mu \cos^2(\om) + \mu \tau_\mu^2)]}{\tilde{\L}_\mu^0(\om)}e^{i \om x}.
$$
We can view $\L_\mu^0$ as a Fourier multiplier operator with symbol $\tilde{\L}_\mu^0(\om)$. 
The definition of $\tau_\mu$ at \eqref{this is tau} implies that $\tilde{\L}_\mu^0(\om) = 0$ if and only $\om = \pm \om_\mu$. 
Lemma 3 in Beale's article \cite{beale1} tells us how to convert such information into necessary and sufficient conditions for solving
the equation $\L_\mu^0 f = g$ when $f$ and $g$ are in exponentially weighted spaces like ours. The outcome is that, so long as $b > 0$:
\be
\label{L necc}
\L_\mu^0 f = g \in O^s_b
\iff \Fo[g]( \om_c) = 0.
\ee 
This condition tells us that the range of $\L_\mu^0$ is a codimension one subspace of $O^s_b$.
Moreover Beale's lemma implies that $\L_\mu^0$ is injective on $O^s_b$.  Thus the Fredholm index
of $\L_\mu^0$ is $-1$. And so we have proven:
\begin{lemma}\label{S fred lemma}
For all $s \ge 0$ and $b >0$, the Fredholm index
of $\L_\mu$, viewed as unbounded operator on $O^s_b$, is equal to~$-1$.
\end{lemma}
Finally, the coercive estimates in Lemma \ref{coerce lemma} imply that the kernel of $\L_\mu$ 
is trivial. And since the Fredholm index is the difference of the dimensions of the kernel and 
the codimension of the range.
\eqref{S fred lemma} implies that the codimension of the range of $\L_\mu$ is exactly one.
And we can move on.

\section{Proof of Lemma \ref{kappa lemma}---the computation of $\kappa_\mu$}\label{kappa appendix}

By definition
$$
\kappa_\mu := \int_\R \gamma_\mu(x) \chi_\mu(x) dx
= \int_\R \gamma_\mu(x) \left(\Sigma_{\mu,2} \varphi_{\mu,2}^0(x) + \mu^2 \Omega_{\mu,2} \varphi_{\mu,1}^0(x)\right) dx.
$$
Using the localizing properties of $\Omega_{\mu,2}$ and $\Sigma_{\mu,2}$ we know the integral converges even
though $\gamma_\mu$ and $\pb_\mu^0$ are asymptotically periodic.
Moreover, since $\Omega_{\mu,2}$ is a bounded map, we have
$
|\kappa_\mu - \tilde{\kappa}_\mu| \le C \mu^2
$
where $\tilde{\kappa}_\mu$ is just the part of $\kappa_\mu$ involving $\Sigma_{\mu,2}$. Using the adjoint property, together with the fact
that $\varphi_{\mu,2}^0 = \sin(\omega_\mu x)$, gives
$$
\tilde{\kappa}_\mu = \int_\R (\Sigma_{\mu,2}^* \gamma_\mu(x)) \sin(\omega_\mu x) dx.
$$
Since $\L^*_{\mu} \gamma_\mu =0$ we can make the substitution
$$
\tilde{\kappa}_\mu =- \int_\R (c^2 \mu  \gamma''_\mu(x) + 2(1+\mu A^2 + \mu^2 \tau_\mu) \gamma_\mu(x)) \sin(\omega_\mu x) dx.
$$
Now the convergence of the integral is less obvious, though of course it must converge. We write
(since the integrand is even):
$$
\tilde{\kappa}_\mu =- 2 \lim_{R \to \infty} \int_0^R (c^2 \mu  \gamma''_\mu(x) + 2(1+\mu A^2 + \mu^2 \tau_\mu) \gamma_\mu(x)) \sin(\omega_\mu x) dx.
$$

Integration by parts and $u-$substitution tells us that (if $f$ and $g$ are odd) that
$$
\int_0^R f''(x) g(x) dx = f'(R) g(R) - f(R) g'(R) + \int_0^R f(x) g''(x) dx
$$
and
\bes\begin{split}
\int_0^R (A^2f(x)) g(x) dx = &\int_0^R f(x)  (A^2g(x))  dx \\+&{1 \over 4} \int_R^{R+2}f(x) S^{-2} g(x) dx
-{1 \over 4} \int_0^{2}f(x) S^{-2} g(x) dx\\
-&{1 \over 4} \int_{R-2}^{R}f(x) S^{2} g(x) dx
+{1 \over 4} \int_{-2}^{0}f(x) S^{-2} g(x) dx.
\end{split}
\ees
These imply that
\be\label{at}\begin{split}
\tilde{\kappa}_\mu
=&-2 \lim_{R\to \infty} \int_0^R \gamma_\mu(x) (c^2 \mu  \partial_x^2 + 2(1+\mu A^2 + \mu^2 \tau_\mu) ) \sin(\omega_\mu x) dx\\
&-2 c^2 \mu  \lim_{R\to \infty} \left( 
\gamma_\mu'(R) \sin(\omega_\mu R) - \omega_\mu \gamma_\mu(R) \cos(\omega_\mu R)
\right)\\
 &+ \lim_{R\to \infty} \left({\mu \over 2} \int_R^{R+2}\gamma_\mu(x) \sin(\omega_\mu (x-2)) dx
+{\mu \over 2} \int_0^{2}\gamma_\mu(x)  \sin(\omega_\mu (x-2))  dx \right.\\
&\left.+{\mu \over 2} \int_{R-2}^{R}\gamma_\mu(x)  \sin(\omega_\mu (x+2)) dx
-{\mu \over 2} \int_{-2}^{0}\gamma_\mu(x)  \sin(\omega_\mu (x+2))  dx \right).
\end{split}
\ee

The first line vanishes by virtue of \eqref{this is tau} and \eqref{why Delta}. From \eqref{gamma limit}
we have
$|\gamma_\mu(R) - \sin(\omega_\mu (R+\vartheta_\mu^\infty))| \to 0
$
and
$|\gamma_\mu'(R) - \omega_\mu \cos(\omega_\mu (R+\vartheta_\mu^\infty))| \to 0
$
as $R \to \infty$.
And so the second line is equal to 
$$
-2 c^2 \mu  \lim_{R\to \infty} \left( 
\omega_\mu \cos(\omega_\mu (R+\vartheta_\mu^\infty)) \sin(\omega_\mu R) - \omega_\mu \sin(\omega_\mu (R+\vartheta_\mu^\infty))\cos(\omega_\mu R)\right)
$$
Trigonometry identies tell us that $\cos(\theta+\theta')\sin(\theta)-\sin(\theta+\theta')\cos(\theta) = - \sin(\theta')$ and so the above is
$$
2 c^2 \mu \omega\mu \lim_{R\to \infty} \sin(\omega_\mu \vartheta_\mu^\infty) = 2 c^2 \mu \omega\mu \sin(\omega_\mu \vartheta_\mu^\infty).
$$
As for the third line of \eqref{at}, we could compute it exactly in this same fashion. Note however, that all the integrals are over intervals of fixed length and the integrands are $\O(1)$. The prefactor $\mu$ thus means these terms are no bigger than $C \mu$. And so all together we have shown
$|\tilde{\kappa}_\mu-2 c^2 \mu \omega\mu \sin(\omega_\mu \vartheta_\mu^\infty)| \le C\mu$.

\subsection{Proof of Lemma \ref{L inv lemma}---estimate of $\L_\mu^{-1} \P_\mu$}
The coercive estimate \eqref{B coerce} is essentially an estimate for $\L_\mu^{-1}$, so what we need
is an estimate for $\P_\mu$.
From its definition, \eqref{sigma estimate} and \eqref{varphi estimate} we have
$$
\| \chi_\mu \|_{s,b_c} \le C \mu^{-s/2}. 
$$
Thus, for $s' \ge s$ and $b \in [0,b_c]$, we have, using \eqref{kappa est} and \eqref{stationary phase}:
\be\label{pi est}
\begin{split}
\|\P_\mu f\|_{s,b} \le &\|f\|_{s,b} + |\kappa_\mu^{-1}||\iota_\mu[f]|\|\chi_\mu\|_{s,b} \\  \le & \|f\|_{s,b} + C
\mu^{(s'-s-1)/2}
\| f\|_{s',b} \\ \le& C(1 + \mu^{(s'-s-1)/2}) \|f\|_{s',b}.
\end{split}\ee

Then the coercive estimates \eqref{B coerce} imply, for $k = -1,0,1,2$
$$
\| \L_\mu^{-1} \P_\mu f \|_{s+k,b} \le C_b \mu^{-(k+1)/2} \left( \| f\|_{s,b} + \|\P_\mu f\|_{s,b}\right).
$$
Then we use \eqref{pi est} to get,  for $s' \ge s$:
\bes\label{general inverse}
\| \L_\mu^{-1} \P_\mu f \|_{s+k,b} \le C \mu^{-(k+1)/2}  (1 + \mu^{(s'-s-1)/2}) \|f\|_{s',b}.
\ees
%
These estimates, together with a little accounting lead to the estimates in Lemma \ref{L inv lemma}.

\section{Proof of Lemma \ref{mover}---estimates and more estimates}\label{estimate appendix}

Now we begin in earnest our proof of Lemma \ref{mover}. 
Let
$$
\mu_{\star}:=\min\left\{\mu_\H(b_*/2),\mu_\L(b_*/2),\mu_\gamma,\mu_\kappa,\mu_\per,\mu_\omega \right\}.
$$
We will be working almost entirely in the sets $U^s_{\mu,\r}$ which have decay rates fixed at $b_*:=b_c/2$.
We take $\mu \in (0,\mu_\star]\cap{M_c}$ throughout the remainder of this appendix and thus can use 
all previous estimates about $\H_\mu$, $\L_\mu$, $\gamma_\mu$, $\kappa_\mu$, $\omega_\mu$ and $a \pb_\mu^a$ freely.

The ``bootstrapping" estimate \eqref{boot} follows from the estimates for $\pb_\mu^a$ in Theorem \ref{periodic solutions},
the fact that $H^s_b$ is an algebra when $s \ge 1$ and $b\ge 0$, and the smoothing properties of $\H_\mu^{-1}$ and $\L_\mu^{-1}\P_\mu$ described in Lemmas \ref{Amu props} 
and \ref{coerce lemma}.  We leave out the details since most of the key ideas will appear below.

So we turn out attention to establishing \eqref{really small}. Fix $|c|\in (c_0,c_1]$, $s\ge 1$ and $\r \in \R^3_+$ (with $\r_3 < a_\per$).
Assume that $(\etab,a) \in U^s_{\mu,\r}$.
We adhere to the following conventions for constants ``$C$" which appear in the remainder of
this appendix. First, any unadorned $C$ is positive and is determined
only by $s$ and $c$. Second, a constant denoted $C_{\r}$ is positive and is determined by $s$, $c$ and the triple $\r \in \R^3_+$.

\subsection{First estimates for $N_1^\mu$}
The definition of 
$N_1^\mu$ is at \eqref{heavy equation v2} and it is made primarily of the ``$j$" functions.
We have $j_k : = J_k \cdot \ib$ where the $J_k$ are defined in Section~\ref{beales ansatz}.

Using the definition of $j_2$ (in \eqref{J2}) together with \eqref{Theta is bounded} and \eqref{sigma estimate} we have
\be\label{j2 est}
\| j_2\|_{s,b_*} \le \mu\|\Theta_{\mu,1} \eta_2\|_{s,b_*} + \mu\|\Omega_{\mu,1} \eta_2\|_{s,b_*} \le C \mu\|\eta_2\|_{s,b_*}. 
\ee
Then we use the estimates implied by membership in $U^s_{\mu,\r}$ to get
$$
\| j_2\|_{s,b_*} \le C \r_2 \mu^3.
$$

Calling back to the definitions of $J_3$ at \eqref{J0345} as well those of $\Sigma_\mu$ and $\Omega_\mu$ in \eqref{Sigmas}-\eqref{Omegas} and $a \pb^a_\mu$ in \eqref{this is p} we have
\be\label{this is j3}
j_3 = -2 a L_\mu Q_\mu(\sigmab_{c,\mu},\pb_\mu^a) \cdot \ib =-\mu a  \Sigma_{\mu,1} \varphi^a_{\mu,1} - \mu a \Omega_{\mu,1} \varphi_{\mu,2}^a.
\ee
Thus
$$
\|j_3\|_{s,b_*} \le \mu|a| \| \Sigma_{\mu,1} \varphi_{\mu,1}^a\|_{s,b_*} + \mu|a|  \| \Omega_{\mu,1} \varphi_{\mu,2}^a\|_{s,b_*}.
$$
Then applying \eqref{sigma estimate} gets us to
$$
\|j_3\|_{s,b_*} \le C\mu |a| \left( \| \varphi_{\mu,1}^a\|_{W^{s,\infty}} 
+ \| \varphi_{\mu,2}^a\|_{W^{s,\infty}} 
\right).
$$
The estimates \eqref{varphi estimate} then yield:
$$
\|j_3\|_{s,b} \le C \mu^{(2-s)/2}|a|.
$$
Membership in $U^s_{\mu,\r}$ give
$$
\|j_3\|_{s,b} \le C \r_3 \mu^{(2-s)/2} \mu^{(6+s)/2} \le C_{\r} \mu^{4}.
$$

Next, the definition of $j_4$ in \eqref{J0345} give
$$
\| j_4\|_{s,b_*} = 2|a| \| L_\mu Q_\mu(\pb_{\mu}^a,\etab) \cdot \ib \|_{s,b_*}.
$$
Then we use \eqref{E1} and \eqref{this is p}
\bes\begin{split}
\| j_4\|_{s,b_*} \le  C |a| &\left(
 \mu\|\varphi_{\mu,1}^a\|_{W^{s,\infty}} \|\eta_1\|_{s,b_*}
 +\|\varphi_{\mu,2}^a\|_{W^{s,\infty}} \|\eta_2\|_{s,b_*} \right. \\ &\left.
+\mu  \|\varphi_{\mu,2}^a\|_{W^{s,\infty}} \|\eta_1\|_{s,b_*}
+\mu^2 \|\varphi_{\mu,1}^a\|_{W^{s,\infty}} \|\eta_2\|_{s,b_*}
\right).
\end{split}\ees
Using the estimates \eqref{varphi estimate} results in
$$
\| j_4\|_{s,b_*} \le  C |a| \left(
\mu^{(2-s)/2} \|\eta_1\|_{s,b_*}
 +\mu^{-s/2} \|\eta_2\|_{s,b_*}
\right).
$$
Since $(\etab,a) \in U^s_{\mu,\r}$, the above give
$$
\| j_4\|_{s,b_*} \le  C_{\r} \mu^{5}.
$$

Next we see that
the definition of $j_5$ in \eqref{J0345} and 
the estimates in \eqref{E3} give:
$$
\|j_5\|_{s,b_*} \le C \left(\|\eta_1\|^2_{s,b_*} + \|\eta_2\|^2_{s,b_*}\right).
$$
Properties of $U^s_{\mu,\r}$ convert this to 
$$
\|j_5\|_{s,b_*} \le C_{\r} \mu^4.
$$

Now that the $j$ functions are estimated, we can bound $N^\mu_1$.
Using the estimate for $\H_\mu^{-1}$ in Proposition \ref{Amu props} 
together with the above estimates  gives
\be\label{general N1 estimate}
(\etab,a) \in U^s_{\mu,\r} \implies \| N^\mu_1(\etab,a)\|_{s+2,b_*} \le C \r_2 \mu^3 + C_{\r} \mu^{4}.
\ee
\subsection{First estimates for $N_2^\mu$ and $N_3^\mu$}
The definitions of 
$N_2^\mu$ and $N_3^\mu$ are at \eqref{N2} and \eqref{a eqn} . They are made primarily of the ``$l$" functions.
We have $l_k : = J_k \cdot \jb$ where the $J_k$ are defined in Section~\ref{beales ansatz}.

Using the definition of $l_0$ at \eqref{J0345}, its representation at \eqref{J0} and the estimates in Lemma~\ref{refined} we
have
$$
\|l_0\|_{s+10,b_*} = c^2 \mu^2 \| \xi_{\mu,2}''\|_{s+10,b_*} \le C \mu^2.
$$

From the defintion of $l_1$ at \eqref{this is l1} and the estimate for $\tau_\mu$ at \eqref{tau is small}, we have
\be\label{l1 est}
\|l_1\|_{s,b_*} \le C \mu^2 \|\eta_2\|_{s,b_*}.
\ee
Membership in $U^s_{\mu,\r}$ then implies
$$
\|l_1\|_{s,b_*} \le C \r_2 \mu^4.
$$

Using the definition of $l_2$ (in \eqref{J2}) together with \eqref{Theta is bounded} and \eqref{sigma estimate} we have
\be\label{l2 est}
\|l_2\|_{s+1,b_*} \le \mu\| \Theta_{\mu,2} \eta_1\|_{s+1,b_*} + \mu\| \Omega_{\mu,2} \eta_1\|_{s+1,b_*} 
\le C \mu \|\eta_1\|_{s+1,b_*}
\ee
And since $(\etab,a) \in U^s_{\mu,\r}$ this yields:
$$
\|l_2\|_{s+1,b_*} \le C_{\r} \mu^4.
$$

Now we estimate $l_{31}$. From \eqref{this is l31} we see $l_{31} = l_3 + a \chi_\mu$.
 $l_3$ is derived from $J_3$, found at \eqref{J0345}. With the definitions of $\Sigma_{\mu,2}$ and $\Omega_{\mu,2}$ in \eqref{Sigmas}-\eqref{Omegas} and $a \pb^a_\mu$ in \eqref{this is p}, this give
$$
l_3 = -2 a L_\mu Q_\mu(\sigmab_{c,\mu},\pb_\mu^a) \cdot \jb =- a  \Sigma_{\mu,2} \varphi^a_{\mu,2} - \mu^2 a \Omega_{\mu,2} \varphi_{\mu,1}^a.
$$
Then we refer to the definition of $\chi_\mu$ at \eqref{chi} to see that
\be\label{l31}
l_{31} = - a \Sigma_{\mu,2} (\varphi_{\mu,2}^a-\varphi_{\mu,2}^0)-  \mu^2 a \Omega_{\mu,2} (\varphi^a_{\mu,1}-\varphi_{\mu,1}^0).
\ee
Using the estimates in \eqref{sigma estimate} and the fact $b_*-b_c=-b_*$ we get
$$
\|l_{31}\|_{s,b_*} \le C|a| 
\| \varphi_{\mu,2}^a - \varphi_{\mu,2}^0\|_{W^{s,\infty}_{-b_*}}.
$$
Then we use \eqref{varphi lipschitz}:
$$
\|l_{31}\|_{s,b_*} \le C\mu^{-s/2}a^2. 
$$
Then properties of $U^s_{\mu,\r}$ convert this to
$$
\|l_{31}\|_{s,b_*} \le C_{\r}
\mu^{6+s/2} \le C_{\r} \mu^6.
$$

Next, the definition of $l_4$ in \eqref{J0345} give
$$
\| l_4\|_{s,b_*} = 2|a| \| L_\mu Q_\mu(\pb_{\mu}^a,\etab) \cdot \jb \|_{s,b_*}.
$$
We use \eqref{E3} and \eqref{this is p} on this and get
\begin{multline*}
\|l_4\|_{s,b_*} \le C |a| \left( 
 \|\varphi_{\mu,2}^a\|_{W^{s,\infty}} \|\eta_1\|_{s,b_*}
 +\|\varphi_{\mu,1}^a\|_{W^{s,\infty}} \|\eta_2\|_{s,b_*} \right. \\ \left.
+\mu  \|\varphi_{\mu,1}^a\|_{W^{s,\infty}} \|\eta_1\|_{s,b_*}
+\mu \|\varphi_{\mu,2}^a\|_{W^{s,\infty}} \|\eta_2\|_{s,b_*}
\right).
\end{multline*}
Then using estimates in \eqref{varphi estimate}:
$$
\|l_4\|_{s,b_*} \le C\mu^{-s/2} |a| \left( \|\eta_1\|_{s,b_*} + \mu \|\eta_2\|_{s,b_*} \right).
$$
The properties of $U^s_{\mu,\r}$ then give:
$$
\|l_4\|_{s,b_*} \le C_{\r}\mu^{6} .
$$

Estimating $l_5$ using \eqref{E3} and \eqref{E4} give
$$
\|l_5\|_{s,b_*} \le C \left( \|\eta_1\|_{s,b_*} \|\eta_2\|_{s,b_*}   + \mu \|\eta_1\|^2_{s,b_*} + \mu \|\eta_2\|^2_{s,b_*}
\right).
$$
Within $U^s_{\mu,\r}$, this estimate gives us
$$
\|l_5\|_{s,b_*} \le C_{\r} \mu^5.
$$

Now we can use the various estimates for $\L_\mu^{-1} \P_\mu$ in \eqref{B inv estimate} to see that
\bes\begin{split}
\| N^\mu_2(\etab,a)\|_{s+1,b_*}
\le C&\left(
\|l_0\|_{s+3,b_*} + \mu^{-3/2} \|l_1\|_{s,b_*} + \mu^{-1} \|l_2 \|_{s+1,b_*}  \right. \\ &\left. + \mu^{-3/2} \|l_{31}\|_{s,b_*} + \mu^{-3/2}\|l_4\|_{s,b_*} + \mu^{-3/2}\|l_5\|_{s,b_*}
\right).
\end{split}
\ees
Then we use the preceding estimates for the $l$-functions to get
\be\label{general N2 estimate}
(\etab,a) \in U^s_{\mu,\r} \implies \| N^\mu_2(\etab,a)\|_{s+1,b_*} \le C  \mu^2 + C_{\r} \mu^{5/2}.
\ee

To estimate $N^\mu_3$ we use the estimate \eqref{stationary phase} for $\iota_\mu$ together with the estimate that $\kappa_\mu = \O(\mu^{1/2})$ for $\mu \in M_c$ to get
\bes\begin{split}
\left \vert N^\mu_3(\etab,a) \right \vert
\le C\mu^{-1/2}&\left( 
\mu^{(s+10)/2}\|l_0\|_{s+10,b_*}
+\mu^{s/2}\|l_1\|_{s,b_*}
+\mu^{(s+1)/2}\|l_2\|_{s+1,b_*} \right. \\ &\left. 
+\mu^{s/2}\|l_{31}\|_{s,b_*}
+\mu^{s/2}\|l_4\|_{s,b_*}
+\mu^{s/2}\|l_5\|_{s,b_*}
\right).
\end{split}\ees
Using the above estimates for $l_0$ through $l_5$ we get
\be\label{general N3 estimate}
(\etab,a) \in U^s_{\mu,\r} \implies
\left \vert N^\mu_3(\etab,a) \right \vert
\le C \r_2 \mu^{(7+s)/2} + C_{\r} \mu^{(8+s)/2}.
\ee

\subsection{The estimates \eqref{really small} and \eqref{map to self}}
The estimates 
\eqref{general N1 estimate},
\eqref{general N2 estimate} and 
\eqref{general N3 estimate} give us \eqref{really small} and \eqref{map to self} in Lemma \ref{mover}.
For $\eqref{map to self}$ put
$s =1$ in those estimates to see that $(\etab,a) \in U^1_{\mu,\r}$ implies
\bes\begin{split}
\| N^\mu_1(\etab,a)\|_{2,b_*} &\le C_*\r_2 \mu^3 + C_{\r} \mu^4, \\
\| N^\mu_2(\etab,a)\|_{1,b_*} &\le C_* \mu^2 + C_{\r} \mu^{5/2}\mand\\
| N^\mu_2(\etab,a)| &\le C_*  \r_2 \mu^4 + C_{\r} \mu^{9/2}.
\end{split}
\ees
The constant $C_*>0$ is determined only by $c$ and is the same across all three inequalities.
The inequalities hold for $\mu \in (0,\mu_\star]\cap M_c$.
Put $\r_2 = \r_{*,2}:=2C_*$, $\r_1=\r_{*,1} = 4C_*^2$ and
$\r_3 =\r_{*,3} :=\r_{*,2}$ and denote the resulting triple by $\r_*$.
Then there exists $\mu_{\star\star} \in (0,\mu_\star]$ so that $\mu \in (0,\mu_{\star\star}] \cap M_c$ turns the last set 
of estimates into
\bes
\| N^\mu_1(\etab,a)\|_{2,b_*} \le \r_{*,1} \mu^3,\quad
\| N^\mu_2(\etab,a)\|_{1,b_*} \le \r_{*,2} \mu^2\mand
| N^\mu_2(\etab,a)| \le \r_{*,3} \mu^{7/2}.
\ees
Which is to say that if $(\etab,a) \in U^1_{\mu,\r_*}$ so is $N^\mu(\etab,a)$. This is \eqref{map to self}.

As for \eqref{really small}, we return our attention to \eqref{general N1 estimate},
\eqref{general N2 estimate} and 
\eqref{general N3 estimate}. We put $\tilde{\r}_2 := C+C_{\r}$,
$\tilde{\r}_1:=C \tilde{\r}_2 + C_{\r}$ and $\tilde{\r}_3:= C \tilde{\r_2} + C_{\r}$ where by $C$ and $C_{\r}$
we mean the constants that appear in those estimates at order $s$ once $\r$ is selected; the implication
in \eqref{really small} follows immediately.

\subsection{The contraction estimate}
Now we turn our  gaze towards \eqref{L2 contract}, which is an estimate
for $N(\etab,a) - N(\grave{\etab},\ga)$.
Let us assume that $(\etab,a),(\grave{\etab},\ga) \in U^1_{\mu,\r_*}$.  Let $\grave{j}_n$
and $\grave{l}_n$ to be the same as $j_n$ and $l_n$ but evaluated at $(\grave{\etab},\ga)$.
We need to estimate $$\|\H_\mu^{-1} (j_n - \gj_n)\|_{2,b_*/2},\quad
\| \L_\mu^{-1} \P_\mu (l_n - \gl_n)\|_{0,b_*/2} \mand
\kappa_\mu^{-1}|\iota_\mu[ l_n - \gl_n]|.$$

Revisiting $j_2$ (at \eqref{J2}) shows that it is linear in $\eta_2$. Thus the steps that led to \eqref{j2 est} lead us to
$$
\|\H_\mu^{-1} (j_2 - \gj_2)\|_{2,b_*/2}
\le C \mu\|\eta_2 - \geta_2\|_{0,b_*/2}.
$$

Using the the formula for $j_3$ at \eqref{this is j3}, the estimates for $\H_\mu^{-1}$ in Lemma \ref{A props} and the triangle inequality
\bes\begin{split}
\| \H_\mu^{-1}(j_3-\gj_3)\|_{2,b_*/2}
\le &C\mu|a-\ga| \| \Sigma_{\mu,1} \varphi^a_{\mu,1}\|_{0,b_*/2}+
C\mu|\ga| \| \Sigma_{\mu,1} (\varphi^a_{\mu,1}-\varphi^{\ga}_{\mu,1})\|_{0,b_*/2}\\
+&C\mu|a-\ga| \| \Omega_{\mu,1} \varphi^a_{\mu,1}\|_{0,b_*/2}+
C\mu|\ga| \| \Omega_{\mu,1} (\varphi^a_{\mu,1}-\varphi^{\ga}_{\mu,1})\|_{0,b_*/2}.
\end{split}\ees
Using \eqref{sigma estimate} and recalling that $b_* = b_c/2$ on the right hand side give
\bes\begin{split}
\| \H_\mu^{-1}(j_3-\gj_3)\|_{2,b_*/2}
\le &C\mu|a-\ga| \|\varphi^a_{\mu,1}\|_{0,-3b_*/2}+
C\mu|\ga| \| \varphi^a_{\mu,1}-\varphi^{\ga}_{\mu,1}\|_{0,-3b_*/2}.\\
\end{split}\ees
The estimates in \eqref{varphi estimate}  and \eqref{varphi lipschitz} then give:
\bes\begin{split}
\| \H_\mu^{-1}(j_3-\gj_3)\|_{2,b_*/2}
\le &C\mu|a-\ga|.
\end{split}\ees

Using the definition of $j_4$ in \eqref{J0345}, boundedness of $\H_\mu^{-1}$ and the triangle inequality gets us
\be\label{heyo}\begin{split}
\| \H_\mu^{-1}(j_4-\gj_4)\|_{2,b_*/2} \le 
& C |a-\ga|\| L_\mu Q_\mu(\pb_\mu^a,\etab)\cdot \ib \|_{0,b_*/2}\\
+&C  |\ga|\| L_\mu Q_\mu(\pb_\mu^a-\pb_\mu^\ga,\etab)\cdot \ib \|_{0,b_*/2}\\
+&C  |\ga|\| L_\mu Q_\mu(\pb_\mu^\ga,\etab-\getab)\cdot \ib \|_{0,b_*/2}.
\end{split}\ee
Let us focus on the middle term. Using \eqref{E1} with $s=0$, $b = b_*/2$ and $b' = b'' = b''' = b'''' = -b_*$ give
\bes\begin{split}
       &C  |\ga|\| L_\mu Q_\mu(\pb_\mu^a-\pb_\mu^\ga,\etab)\cdot \ib \|_{0,b_*/2}\\
\le &C|\ga|\left(
\mu \| \varphi_{\mu,1}^a- \varphi_{\mu,1}^\ga\|_{W^{0,\infty}_{-b_*/2}} \|\eta_1\|_{0,b_*} 
     +\| \varphi_{\mu,2}^a -\varphi_{\mu,2}^\ga\|_{W^{0,\infty}_{-b_*/2}} \|\eta_2\|_{0,b_*}\right)\\
 +&C|\ga|\mu \left(\mu \| \varphi_{\mu,1}^a- \varphi_{\mu,1}^\ga\|_{W^{0,\infty}_{-b_*/2}} \|\eta_2\|_{0,b_*}
  + \| \varphi_{\mu,2}^a -\varphi_{\mu,2}^\ga\|_{W^{0,\infty}_{-b_*/2}} \|\eta_1\|_{0,b_*} \right).
\end{split}\ees
Then \eqref{varphi lipschitz} gives
\bes\begin{split}
       C  |\ga|\| L_\mu Q_\mu(\pb_\mu^a-\pb_\mu^\ga,\etab)\cdot \ib \|_{0,b_*/2}
\le & C|\ga|\left( \|\eta_1\|_{0,b_*} + \|\eta_2\|_{0,b_*}\right)|a-\ga|
\end{split}\ees
and since we are working in $U^1_{\mu,\r_*}$:
\bes\begin{split}
       C  |\ga|\| L_\mu Q_\mu(\pb_\mu^a-\pb_\mu^\ga,\etab)\cdot \ib \|_{L^2}
\le & C\mu^{11/2}|a-\ga|.
\end{split}\ees
The same sort of reasoning on the other two lines in \eqref{heyo} leads us to
\bes\begin{split}
\| \H_\mu^{-1}(j_3-\gj_3)\|_{2,b_*/2} \le C \mu^{2}|a-\ga| + C \mu^{7/2}\| \eta_1 - \geta_1\|_{0,b_*/2}+C\mu^{7/2}\| \eta_2 - \geta_2\|_{0,b_*/2}.
\end{split}\ees

The definition of $j_5$ and boundedness of $\H_\mu^{-1}$ yield:
\be\label{heyo2}\begin{split}
\| \H_\mu^{-1}(j_5-\gj_5)\|_{2,b_*/2} \le \| L_\mu Q_\mu(\etab+\getab,\etab-\getab) \cdot \ib\|_{0,b_*/2}
\end{split}\ee
Using \eqref{E1} converts this to
\bes\begin{split}
\| \H_\mu^{-1}(j_5-\gj_5)\|_{2,b_*/2} \le &C\left( \| \eta_{\mu,1}+\geta_{\mu,1}\|_{L^\infty}
\| \eta_{\mu,1}-\geta_{\mu,1}\|_{0,b_*/2} + \| \eta_{\mu,2}+\geta_{\mu,2}\|_{L^\infty}
\| \eta_{\mu,2}-\geta_{\mu,2}\|_{0,b_*/2}\right)\\
+ &C\mu\left( \| \eta_{\mu,2}+\geta_{\mu,2}\|_{L^\infty}
\| \eta_{\mu,1}-\geta_{\mu,1}\|_{0,b_*/2} + \| \eta_{\mu,2}+\geta_{\mu,2}\|_{L^\infty}
\| \eta_{\mu,1}-\geta_{\mu,1}\|_{0,b_*/2}
\right)
\end{split}\ees
Then using Sobolev embedding and the properties of $U^1_{\mu,\r_*}$ give
$$
\| \H_\mu^{-1}(j_5-\gj_5)\|_{2,b_*/2} \le C\mu^3 \| \eta_{\mu,1}-\geta_{\mu,1}\|_{0,b_*/2} + C \mu^2\| \eta_{\mu,2}-\geta_{\mu,2}\|_{0,b_*/2}.
$$

With this, we have
$$
(\etab,a),(\getab,\ga) \in U^1_{\mu,\r_*} \implies \|N^\mu_1(\etab,a) - N^\mu_1(\getab,\ga)\|_{2,b_*/2} \le C \mu \|(\etab,a)-(\getab,\ga)\|_{X_0}.
$$

Now we work on the $l$ functions.
First, $l_0$ (defined at \eqref{J0345} and \eqref{J0}) depends on neither $a$ nor $\etab$ and so $l_0 - \grave{l_0} = 0$.
Second, $l_1$ and $l_2$ (see  \eqref{J2} and \eqref{this is l1}) are linear in $\etab$ and thus we have, using the same ideas
that gave us \eqref{l2 est} and \eqref{l1 est}:
\be
\|l_1-\gl_1\|_{0,b_*/2} \le C \mu^2 \|\eta_2-\geta_2\|_{0,b_*/2}
\mand
\|l_2-\gl_2\|_{2,b_*/2} \le C \mu \|\eta_1-\geta_1\|_{2,b_*/2}.
\ee

For $l_{31}$ we use the formula \eqref{l31} and see, by way of the triangle inequality, that
\be\begin{split}
\| l_{31} - \gl_{31}\|_{0,b_*/2}& \le 
2|a-\ga|\| \Sigma_{\mu,2} (\varphi_{\mu,2}^a -  \varphi_{\mu,2}^0)\|_{0,b_*/2}
+ 2 |\ga| \| \Sigma_{\mu,2} (\varphi_{\mu,2}^a -  \varphi_{\mu,2}^\ga)\|_{0,b_*/2}\\
&+2\mu^2|a-\ga|\| \Omega_{\mu,2} (\varphi_{\mu,1}^a -  \varphi_{\mu,1}^0)\|_{0,b_*/2}
+ 2 \mu^2|\ga| \| \Omega_{\mu,2} (\varphi_{\mu,1}^a -  \varphi_{\mu,1}^\ga)\|_{0,b_*/2}.
\end{split}\ee
Using \eqref{sigma estimate} and the fact that $b_*=b_c/2$ turns this into
\bes\begin{split}
\| l_{31} - \gl_{31}\|_{0,b_*/2} \le
& C|a-\ga|\| \varphi_{\mu,2}^a -  \varphi_{\mu,2}^0\|_{0,-3b_*/4}
+ C |\ga| \| \varphi_{\mu,2}^a -  \varphi_{\mu,2}^\ga\|_{0,-3b_*/4}\\
&+ C\mu^2|a-\ga|\| \varphi_{\mu,1}^a -  \varphi_{\mu,1}^0\|_{0,-3b_*/4}
+ C \mu^2|\ga| \| \varphi_{\mu,1}^a -  \varphi_{\mu,1}^\ga\|_{0,-3b_*/4}.
\end{split}\ees
Then we \eqref{varphi estimate} and \eqref{varphi lipschitz} to get
$$
\| l_{31} - \gl_{31}\|_{0,b_*/2} \le C(|a| + |\ga|) |a-\ga|.
$$
And since we are in $U^1_{\mu,\r_*}$:
$$
\| l_{31} - \gl_{31}\|_{0,b_*/2} \le C \mu^{7/2} |a-\ga|.
$$

Using the definition of $l_4$ in \eqref{J0345} and the triangle inequality gives us
\be\label{ug}\begin{split}
\| l_4-\gl_4\|_{0,b_*/2} \le 
& C |a-\ga|\| L_\mu Q_\mu(\pb_\mu^a,\etab)\cdot \jb \|_{0,b_*/2}\\
+&C  |\ga|\| L_\mu Q_\mu(\pb_\mu^a-\pb_\mu^\ga,\etab)\cdot \jb \|_{0,b_*/2}\\
+&C  |\ga|\| L_\mu Q_\mu(\pb_\mu^\ga,\etab-\getab)\cdot \jb \|_{0,b_*/2}.
\end{split}\ee
Focus on the middle line. We use \eqref{E2} to get
\bes\begin{split}
       &C  |\ga|\| L_\mu Q_\mu(\pb_\mu^a-\pb_\mu^\ga,\etab)\cdot \jb \|_{0,b_*/2}\\
\le &C|\ga|\left(
\mu \| \varphi_{\mu,1}^a- \varphi_{\mu,1}^\ga\|_{W^{0,\infty}_{-b_*/2}} \|\eta_2\|_{0,b_*} 
     +\| \varphi_{\mu,2}^a -\varphi_{\mu,2}^\ga\|_{W^{0,\infty}_{-b_*/2}} \|\eta_1\|_{0,b_*}\right)\\
 +&C|\ga|\mu \left(\mu \| \varphi_{\mu,1}^a- \varphi_{\mu,1}^\ga\|_{W^{0,\infty}_{-b_*/2}} \|\eta_1\|_{0,b_*}
  + \| \varphi_{\mu,2}^a -\varphi_{\mu,2}^\ga\|_{W^{0,\infty}_{-b_*/2}} \|\eta_2\|_{0,b_*} \right).
\end{split}\ees
Then \eqref{varphi lipschitz} give
$$
       C  |\ga|\| L_\mu Q_\mu(\pb_\mu^a-\pb_\mu^\ga,\etab)\cdot \jb \|_{0,b_*/2}
        \le C|\ga| \left( \|\eta_1\|_{b_*} + \mu \| \eta_2\|_{b_*} \right)|a-\ga|.
$$
And then membership in $U^1_{\mu,\r_*}$ give
$$
       C  |\ga|\| L_\mu Q_\mu(\pb_\mu^a-\pb_\mu^\ga,\etab)\cdot \jb \|_{0,b_*/2} \le C \mu^{13/2}|a-\ga|.
$$
The remaining terms in \eqref{ug} are handled similarly and we find
$$
\| l_4-\gl_4\|_{0,b_*/2} \le C \mu^4 |a-\ga| + C \mu^{7/2} \| \eta_1 - \geta_1\|_{0,b_*/2} + C \mu^{7/2} \|\eta_2-\geta_2 \|_{0,b_*/2} .
$$

The definition of $l_5$ at \eqref{J0345} give
\be\label{heyo4}\begin{split}
\| l_5-\gl_5\|_{0,b_*/2} \le \| L_\mu Q_\mu(\etab+\getab,\etab-\getab) \cdot \jb\|_{0,b_*/2}
\end{split}\ee
Using \eqref{E2} converts this to
\bes\begin{split}
\| l_5-\gl_5\|_{0,b_*/2} \le &C\left( \| \eta_{\mu,2}+\geta_{\mu,2}\|_{L^\infty}
\| \eta_{\mu,1}-\geta_{\mu,1}\|_{0,b_*/2} + \| \eta_{\mu,1}+\geta_{\mu,1}\|_{L^\infty}
\| \eta_{\mu,2}-\geta_{\mu,2}\|_{0,b_*/2}\right)\\
+ &C\mu\left( \| \eta_{\mu,1}+\geta_{\mu,1}\|_{L^\infty}
\| \eta_{\mu,1}-\geta_{\mu,1}\|_{0,b_*/2} + \| \eta_{\mu,2}+\geta_{\mu,2}\|_{L^\infty}
\| \eta_{\mu,2}-\geta_{\mu,2}\|_{0,b_*/2}
\right)
\end{split}\ees
Then using Sobolev embedding and the properties of $U^1_{\mu,\r_*}$ give
$$
\| l_5-\gl_5\|_{0,b_*/2}  \le C\mu^2 \| \eta_{\mu,1}-\geta_{\mu,1}\|_{0,b_*/2} + C \mu^3\| \eta_{\mu,2}-\geta_{\mu,2}\|_{0,b_*/2}.
$$

Then we can use the various estimates for $\L_\mu^{-1} \P_\mu$  in \eqref{B inv estimate} to get
\begin{multline*}
\| N^\mu_2(\etab,a) - N^\mu_2(\getab,\ga)\|_{0,b_*/2}
\le C \left( 
 \mu^{-1} \| l_1 - \gl_1\|_{0,b_*/2}
+ \| l_2 - \gl_2\|_{2,b_*/2} \right. \\ \left.
+ \mu^{-1}\| l_{31} - \gl_{31}\|_{0,b_*/2}
+ \mu^{-1}\| l_{4} - \gl_{4}\|_{0,b_*/2}
+ \mu^{-1}\| l_{5} - \gl_{5}\|_{0,b_*/2}
\right).
\end{multline*}
Using the preceding estimates together with this gives us
$$
(\etab,a),(\getab,\ga) \in U^1_{\mu,\r_*} \implies \|N^\mu_2(\etab,a) - N^\mu_2(\getab,\ga)\|_{0,b_*/2}\le C \mu \|(\etab,a)-(\getab,\ga)\|_{X_0}.
$$

Similarly using \eqref{stationary phase} shows that
\begin{multline*}
| N^\mu_3(\etab,a) - N^\mu_3(\getab,\ga)|
\le C\mu^{-1/2} \left( 
  \| l_1 - \gl_1\|_{0,b_*/2}
+ \mu \| l_2 - \gl_2\|_{2,b_*/2} \right. \\ \left.
+ \| l_{31} - \gl_{31}\|_{0,b_*/2}
+ \| l_{4} - \gl_{4}\|_{0,b_*/2}
+ \| l_{5} - \gl_{5}\|_{0,b_*/2}
\right).
\end{multline*}
Using the estimates above give
$$
(\etab,a),(\getab,\ga) \in U^1_{\mu,\r_*} \implies |N^\mu_3(\etab,a) - N^\mu_2(\getab,\ga)|
\le C \mu^{3/2} \|(\etab,a)-(\getab,\ga)\|_{X_0}.
$$

Thus all together we have \eqref{L2 contract} so long as we take $\mu \in M_c$ 
sufficiently close to zero; we call the threshold $\mu_*$.
This concludes the proof of Lemma \ref{mover} and also this paper.

\bibliographystyle{siam}
\bibliography{massdimer-smallmasslimit}{}

\end{document}